\documentclass[a4paper]{amsart}
\usepackage{amssymb,cmap,verbatim}
\usepackage{stmaryrd} 
\usepackage{tabularx}

\usepackage{xcolor}

\newtheorem{thm}{Theorem}[section]
\newtheorem{lemma}[thm]{Lemma}
\newtheorem{cor}[thm]{Corollary}
\newtheorem{claim}{Claim}[thm]

\newtheorem{prop}[thm]{Proposition}
\newtheorem{example}[thm]{Example}
\newtheorem{fact}[thm]{Fact}
\newtheorem{quest}[thm]{Question}
\newtheorem*{thmaa}{Theorem~A}
\newtheorem*{thmb}{Theorem~B}
\newtheorem*{thmc}{Theorem~C}

\theoremstyle{definition}
\newtheorem{defn}[thm]{Definition}

\theoremstyle{remark}
\newtheorem{remark}[thm]{Remark}

\DeclareMathOperator{\acc}{acc}

\DeclareMathOperator{\dom}{dom}

\DeclareMathOperator{\pr}{pr}

\DeclareMathOperator{\dl}{Dl}

\newcommand{\s}{\subseteq}

\renewcommand{\mid}{\mathrel{|}\allowbreak}
\renewcommand{\restriction}{\mathbin\upharpoonright}

\newcommand{\symdiff}{\mathbin\triangle}

\newcommand{\redur}{\mathrel{\hookrightarrow_r}}                 
\newcommand{\redum}{\mathrel{\hookrightarrow_{BM}}}     
\newcommand{\redub}{\mathrel{\hookrightarrow_B}}                 
\newcommand{\reduc}{\mathrel{\hookrightarrow_c}}                 
\newcommand{\redul}{\mathrel{\hookrightarrow_L}}                 
\newcommand{\redua}{\mathrel{\hookrightarrow_\alpha}}            
\newcommand{\redup}{\mathrel{\hookrightarrow_{\alpha+1}}}            

\title{Shelah's Main Gap and the generalized Borel reducibility}

\author{Miguel Moreno}
\address{Institute of Mathematics, University of Vienna, Vienna 1090, Austria.}
\address{Department of Mathematics and Statistics, University of Helsinki, Helsinki 00560, Finland.}
\address{Mathematics Research Centre, Tampere University, Tampere 33720, Finland.}
\urladdr{http://miguelmath.com}

\begin{document}
\begin{abstract}
We answer one of the main questions in generalized descriptive set theory, the Friedman-Hyttinen-Kulikov conjecture on the Borel reducibility of the Main Gap.
We show a correlation between Shelah's Main Gap and generalized Borel reducibility notions of complexity. 
For any $\kappa$ satisfying $\kappa=\lambda^+=2^\lambda$ and $2^{\mathfrak{c}}\leq\lambda=\lambda^{\omega_1}$, we show that if $T$ is a classifiable theory and $T'$ is a non-classifiable theory, then the isomorphism of models of $T'$ is strictly above the isomorphism of models of $T$ with respect to Borel reducibility.

We also show that the following can be forced: for any countable first-order theory in a countable vocabulary, $T$, the isomorphism of models of $T$ is either analytic co-analytic, or analytically-complete.

\end{abstract}


\maketitle

\tableofcontents

\section{Introduction}\label{section_intro}

In this article we show a correlation between Shelah's classification theory and Friedman-Stanley Borel reducibility. Two different methods to classify first-order countable complete theories developed in the 1980's, the former one from model theory and the latter one from set theory.

One of the classic problems in mathematics is the independence of Euclid's fifth postulate, the parallel postulate. In their approach to the problem, Khayy\'am (1077) and Saccheri (1733) considered the three different cases of the Khayy\'am-Saccheri quadrilateral (right, obtuse, and acute). These three cases correspond to Euclidean geometry, Elliptic geometry, and Bolyai-Lobachevsky geometry (Hyperbolic geometry). This exemplifies how mathematicians are interested in the different models of a theory. This interest is more notorious in logic, where the study of the models of a theory is a central topic. In particular the study of the spectrum function $I(\lambda,T)$, the number of non-isomorphic models of cardinality $\lambda$ of a theory $T$, was studied by many mathematicians in the 20th century.

The most ``simple" case for the spectrum function would be when it takes values in the set $\{0,1\}$.
In 1904 Veblen introduced the notion of categorical theory \cite{Veb}, a theory is categorical if it has a model and all its models are isomorphic. In the language of the spectrum function, a theory $T$ is categorical if there is a cardinal $\lambda$ such that $I(\lambda,T)=1$ and for all cardinals $\kappa\neq \lambda$, $I(\kappa,T)=0$.
In 1915 L\"owenheim proved one of the first results of model theory \cite{Lowe}, the L\"owenheim-Skolem theorem (it was also proved by Skolem in 1920 \cite{Skol}). The L\"owenheim-Skolem theorem for first-order countable theories tells us that if a first-order countable theory has an infinite model, then for every infinite cardinal $\lambda$, there is a model of cardinality $\lambda$. This implies that there is no first-order countable categorical theory with an infinite model. 

It is surprising that this result came before G\"odel's completeness theorem, one of the classic results of logic. G\"odel's completeness theorem was proved initially in his Ph.D. thesis \cite{Godt} in 1929 and published later \cite{God} in 1930. It tells us that if a theory $T$ is consistent, then there is a model of $T$. This is one of the fundamental results of logic.

In 1954 Łoś \cite{Los} and Vaught\cite{Vau} introduced the notion of $\kappa$-categorical theory. A theory $T$ is $\kappa$-categorical if there is only one model of $T$ of size $\kappa$ up to isomorphism. 
Łoś  announced that he had found three kinds of $\kappa$-categorical first-order countable complete theories: 

\begin{itemize}
\item {\bf Totally categorical:} $I(\kappa,T)=1$ for every infinite cardinal $\kappa$.
\item {\bf Uncountably categorical:} $I(\kappa,T)=1$ if and only if $\kappa$ is an uncountable cardinal.
\item {\bf Countably categorical:} $I(\kappa,T)=1$ if and only if $\kappa$ is countable.
\end{itemize}

Łoś raised the following question about first-order countable complete theories:

\textit{Is a theory categorical in one uncountable cardinal necessarily categorical in every uncountable cardinal?}

In 1965 Morley answered Łoś's question with his categoricity theorem \cite{Morl}.
\begin{fact}[{Morley's categoricity theorem, \cite[Theorem 5.6]{Morl}}]
Let $T$ be a first-order countable complete theory. If $T$ is categorical in one uncountable cardinal, then $T$ is categorical in every uncountable cardinal. 
\end{fact}

Morley's work motivated the study of the spectrum function in more detail.
In the 1960's it was conjectured that for every first-order countable complete theory over a countable language, $T$, and for all $\aleph_0<\kappa<\lambda$, $$I(\kappa,T)\leq I(\lambda,T).$$
This conjecture is known as Morley's conjecture. In 1969 \cite{Shstb}, Shelah generalized Morley's theorems from \cite{Morl}. During the 1970's Shelah extended Morley's theories, this lead to the develop of stability theory and the program of classification theory. To answer Morley's conjecture, Shelah studied the spectrum problem. In his book ``\textit{Classification theory}", \cite{Sh90}, Shelah published his renowned result \textit{The Main Gap}. With his classification theory program, Shelah was able to answer Morley's conjecture.

\begin{fact}[{Main Gap, Shelah \cite[XII, Theorem 6.1]{Sh90}}]\label{SHMGT}
Let $T$ be a first order countable complete theory and denote by $I(\lambda,T)$ the number of non-isomorphic models of $T$ of size $\lambda$.
\begin{enumerate}
    \item If $T$ is not superstable or (is superstable) deep or has the DOP or has the OTOP, then for every uncountable $\lambda$, $I(\lambda,T)=2^\lambda$.
    \item If $T$ is shallow superstable without the DOP and without the OTOP, then for every $\alpha>0$, $I(\aleph_\alpha,T)<\beth_{\omega_1}(|\alpha|)$.
\end{enumerate}
\end{fact}

\begin{fact}[{Morley's Conjecture, Shelah \cite[XIII, Theorem 3.7]{Sh90}}]\label{SHMORL}
Let $T$ be a countable complete first-order theory. Then for $\lambda>\mu\ge\aleph_0$, $I(\lambda,T)\ge I(\mu,T)$ except when $\lambda>\mu=\aleph_0$, $T$ is complete, $\aleph_1$-categorical not $\aleph_0$-categorical.
\end{fact}

In 2000 Hart, Hrushovski, and Laskowski extended Shelah's work and gave a complete solution to the spectrum problem for countable theories in uncountable cardinalities, see \cite{HHL} where all possible values of $I(\lambda,T)$ are given.

Thanks to these results we can see that there is a notion of complexity associated to the value of $I(\lambda,T)$, i.e. the more models a theory has at a fixed uncountable cardinal, more complex the theory is. However this notion of complexity doesn't tell us anything about how to compare the complexity of theories with $2^\lambda$ non-isomorphic models, i.e. theories in the first item of Fact \ref{SHMGT}. 

We can use the proof of Fact \ref{SHMGT} to understand the complexity of some theories. 
To prove Fact \ref{SHMGT}, Shelah used his method of dividing lines. Shelah used stability theory to defined the different dividing lines. Denote by $S^m(A)$ the set of all consistent types over $A$ in $m$ variables (modulo change of variables), and $S(A)=\cup_{m<\omega}S^m(A)$.

    \begin{itemize}
        \item We say that $T$ is $\xi$-stable if for any set $A$, $|A|\leq \xi$, $|S(A)|\leq \xi$.
        \item We say that $T$ is stable if there is an infinite $\xi$, such that $T$ is $\xi$-stable.
        \item We say that $T$ is unstable if there is no infinite $\xi$, such that $T$ is $\xi$-stable.
        \item We say that $T$ is superstable is there is an infinite $\xi$ such that for all $\xi'>\xi$, $T$ is $\xi'$-stable.
    \end{itemize}

The main dividing line in Fact \ref{SHMGT} is \textit{classifiable theories vs non-classifiable theories}. This line is defined by using the notion of stable theory, the OTOP, and the DOP. 

\begin{defn}[OTOP]
    A theory $T$ has the omitting type order property (OTOP) if there is a sequence $(\varphi_m)_{m<\omega}$ of first order formulas such that for every linear order $l$ there is a model $\mathcal{M}$ and tuples $\langle a_t\mid t\in l\rangle$ of $\mathcal{M}$, such that $s<t$ if and only if there is a tuple $c$ of $\mathcal{M}$ such that for every $m<\omega$, $$\mathcal{M}\models\varphi_m(c,a_s,a_t).$$
\end{defn}

We write $a\downarrow_AB$ if $tp(a,A\cup B)$ does not fork over $A$ and $C\downarrow_AB$ if $c\downarrow_AB$ holds for all $c\in C$. For the definition of forking see \cite[Definition III 1.4]{Sh90}.
The definition of DOP uses the notions $F^a_{\omega}$ (see \cite[Definition IV 2.1 (2)]{Sh90} or \cite[Definition 2.4)]{Mor}), $F^a_{\omega}$-saturated (see \cite[Definition IV 1.1 (1) and (4)]{Sh90}), $F^a_{\omega}$-prime (see \cite[Definition IV 1.4 (3)]{Sh90}), $F^a_{\omega}$-minimal (see \cite[Definition IV 4.4]{Sh90}).

\begin{defn}[DOP]
    A theory $T$ has the dimensional order property (DOP) if there are $F^a_{\omega}$-saturated models $(M_i)_{i<3}$, $M_0\subseteq M_1\cap M_2$, $M_1\downarrow_{M_0}M_2$, and the $F^a_{\omega}$-prime model over $M_1\cup M_2$ is not $F^a_{\omega}$-minimal over $M_1\cup M_2$.
\end{defn}

\begin{defn}\hfill
\begin{itemize}
    \item We say that $T$ is classifiable if $T$ is superstable without DOP and without OTOP. 
    \item We say that $T$ is non-classifiable if it satisfies one of the following:
\begin{enumerate}
\item $T$ is stable unsuperstable;
\item $T$ is superstable and has DOP;
\item $T$ is superstable and has OTOP;
\item $T$ is unstable.
\end{enumerate}
\end{itemize}

\end{defn}

The main dividing line and the Main Gap Theorem captures a general idea of complexity. 

\textit{Classifiable theories are less complex than non-classifiable theories, and their complexities are far apart.} 

Even thought the different dividing lines give us a way to understand the complexity of theories, it would be good to have a  complexity notion that extends the one given by the Main Gap and allows us to compare the complexity of any two theories.

The existence of a complexity notion that goes beyond the number of non-isomorphic models was being sought out before the publication of the Main Gap Theorem, with the work of Vaught in \cite{Vau2}. 
Motivated by Vaught conjecture, in 1989 Friedman and Stanley introduced the Borel reducibility theory for classes of countable structures \cite{FS}. This led to a new complexity notion for classes of countable structures. The Borel reducibility complexity in descriptive set theory was developed during the 1990's by Becker, Kechris, Hjorth, and others (see \cite{BK}, \cite{HKe}, \cite{HKe2}). A subspace $K$ is \textit{invariant} if $K$ is closed under isomorphism. Given $K$ and $K'$ invariant sets, we say that $K$ is \textit{Borel reducible} to $K'$ if there is a Borel function $f$ from $K$ to $K'$such that for all $\mathcal{A},\mathcal{B}\in K$, $$\mathcal{A}\ \cong\ \mathcal{B}\ \Leftrightarrow\ f(\mathcal{A})\ \cong\ f(\mathcal{B}).$$

In the 2000's the question of whether the notion of Borel reducibility correlates with model-theoretic notions of complexity was studied. Koerwien, Laskowski, Marker, Shelah, and others studied the Borel reducibility of theories (i.e. the Borel reducibility of models of a theory $T$ with universe $\omega$), see e.g. \cite{Koe}, \cite{Las}, \cite{LS}, \cite{Mark}.
Of particular impact in these investigation is the notion of Borel completeness. 
We say that $K$ is \textit{Borel complete} if every invariant Borel $K'$ is Borel reducible to $K$. 
A theory $T$ is \textit{Borel complete} if the set of models of $T$ with universe $\omega$ is Borel complete. Laskowski \cite{Las} showed that theories with ENI-DOP are Borel-complete. Laskowski and Shelah \cite{LS} showed that if $T$ is ENI-deep, then $T$ is Borel-complete.

This shows that there are some correlations between the model-theoretic notions and the Borel reducibility notions. Thus, the Borel reducibility is a good candidate for a complexity notion that extends the complexity notion given by the Main Gap and allows us to compare the complexity of any two theories. To study this, it required the development of descriptive set theory at uncountable cardinals extending the work of Friedman and Stanley. The development of descriptive set theory at uncountable cardinals started in 1991 with V\"a\"an\"anen's work \cite{Van3}.

In the 1990's the connection between model theory of uncountable models and descriptive set theory was studied by Mekler, V\"a\"an\"anen, and others, see \cite{MV}, \cite{Van3}, and \cite{Van}. In 1993, Mekler and V\"a\"an\"anen's article \cite{MV} reinforced the basis for descriptive set theory at uncountable cardinals (later known as \textit{Generalized Descriptive Set Theory}, abbreviated by \textit{GDST}). 
During the 2000's Generalized Descriptive Set Theory developed further (see \cite{Van2} chapter 9.6). 
 In 2014, Friedman, Hyttinen, and Weinstein (n\'e Kulikov) \footnote{Kulikov's last name changed to Weinstein} \cite{FHK13}, provided significant advances on the way towards a fully developed Generalized Descriptive Set Theory. In
particular, they started a systematic comparison between the Main Gap dividing lines and the complexity given by Borel reducibility. Apart from the work of Friedman, Hyttinen, and Weinstein, during the 2010's, Generalized Descriptive Set Theory was intensively developed and became popular with the works of L\"ucke, Moreno, Motto Ros, Schlicht, and others (see \cite{Mort}, \cite{Mot}, and \cite{LuS}).

By studying the bounded topology, Mekler and Väänänen proved the separation theorem in Generalized Descriptive Set Theory, \cite{MV}. Motivated by the interaction between Generalized Descriptive Set Theory and Model Theory, Friedman, Hyttinen, and Weinstein used the separation result to identify one of Shelah's dividing lines of classification theory: 

\textit{The isomorphism relation of a theory is $\kappa$-Borel if and only if it is a classifiable shallow theory.}

In 2020 Calderoni, Mildenberger, and Motto Ros show that the embeddability relation between $\kappa$-sized structures is strongly invariantly universal \cite{CMM}. This is a generalization of Louveau-Rosendal result of countable structures \cite{LR}.

The problem of classifying uncountable structures has been an important theme in Generalized Descriptive Set Theory.
The interaction between Borel reducibility and classification theory has been one of the biggest motivations behind the development of Generalized Descriptive Set Theory. 
Identifying counterparts to classification theory in the setting of Borel reducibility, has played a main role in the development of Generalized Descriptive Set Theory. 

As it was mentioned above, the main problem in this theme is to  find a counterpart of Shelah's main division line, classifiable theories vs non-classifiable theories. 
Friedman, Hyttinen, and Weinstein  (n\'e Kulikov) conjectured that in Generalized Descriptive Set Theory, the isomorphism relation of a classifiable theory is Borel reducible to the isomorphism relation of a non-classifiable theory. 

\begin{quest}[Friedman-Hyttinen-Kulikov \cite{FHK13}]\label{MainGapQ}
    Work in the generalized Baire space $\kappa^\kappa$ with $\kappa^{<\kappa}=\kappa$. Is the isomorphism relation of any classifiable theory Borel reducible to the isomorphism relation of any non-classifiable theory?
\end{quest}

From their work, Friedman, Hyttinen, and Weinstein conjectured that the answer to Question \ref{MainGapQ} is \textit{``yes"}. The solution of the conjecture relies on additional cardinal arithmetic assumptions

\subsection{Contributions}

In this article we prove Friedman-Hyttinen-Kulikov conjecture. We provide a Borel reducibility counterpart of Shelah's Main Gap Theorem. 

Denote by $\cong_T$ the isomorphism relation of models of $T$ of size $\kappa$, $\redub$ the Borel reducibility, $\reduc$ the continuous reducibility, and $\not\redub$ the non-existence of a Borel reduction (this notions are properly defined in the related work sub-section). 

As discussed previously, we use reductions to compare the complexity of different binary relations and we will use it to compare the isomorphism relation of two different theories. In our framework, $R_1\redub R_2$ means that the complexities of $R_1$ and $R_2$ are comparable and $R_1$ is no more complex than $R_2$. Clearly $R_1\redub R_2$ leaves the possibility that $R_1$ and $R_2$ are equally complex, i.e. $R_1\redub R_2$ and $R_2\redub R_1$ ($R_1$ and $R_2$ are bi-reducible). On the other hand, $R_2\not\redub R_1$ means that either the complexities of $R_1$ and $R_2$ are not comparable, or they are comparable and $R_1$ is less complex than $R_2$, i.e. $R_2\not\redub R_1$ and $R_1\redub R_2$. We say that a reduction, e.g. $R_1\redub R_2$, is Borel-strict if $R_2\not\redub R_1$.

Continuous reductions are stronger than Borel reductions (see subsection \ref{subsection1.3}), thus the non-existence of a Borel reduction implies the non-existence of a continuous reduction. Therefore, a great way to compare and separate complexities is by Borel-strict continuous reductions, i.e. $R_1\reduc R_2$ and $R_2\not\redub R_1$.

\begin{thmaa}[Borel reducibility Main Gap]
    Suppose $\kappa=\lambda^+=2^\lambda$, $2^{\mathfrak{c}}\leq\lambda=\lambda^{\omega_1}$, and $T_1$ and $T_2$ are countable complete first-order theories in a countable vocabulary. If $T_1$ is a classifiable theory and $T_2$ is a non-classifiable theory, then $$\cong_{T_1}\ \reduc\ \cong_{T_2} \textit{ and } \cong_{T_2}\ \not\redub\ \cong_{T_1}.$$
\end{thmaa}

The Borel reducibility Main Gap compare and separates the complexity of classifiable and non-classifiable theories, but it doesn't tell how far apart are the complexities.
We study the size of the gap (classifiable vs non-classifiable) and compare the complexity of different theories.

\begin{thmb}[Main Gap Dichotomy]
  Let $\kappa$ be inaccessible or $\kappa=\lambda^+=2^\lambda$, and $2^{\mathfrak{c}}\leq\lambda=\lambda^{\omega_1}$. There exists a forcing extension by a $<\kappa$-closed $\kappa^+$-cc forcing notion in which for any countable first-order theory in a countable vocabulary (not necessarily complete), $T$, one of the following holds:
\begin{itemize}
\item $\cong_T$ is $\Delta^1_1(\kappa)$;
\item $\cong_T$ is $\Sigma^1_1(\kappa)$-complete.
\end{itemize}
\end{thmb}

We study how far are the complexities of classifiable shallow theories and classifiable non-shallow theories. This implies the study of isomorphism relations with few equivalence classes. Since $\kappa^{<\kappa}=\kappa$, we study the complexity of the isomorphism relation of models with size smaller than $\kappa$.
 We show that there is a gap in between the isomorphism relations of models with different sizes and this gap is not empty. Denote by $\cong_T^\lambda$ the isomorphism relation of models of $T$ of size $\lambda$ and $id_2$ the identity relation of $2^\kappa$.

 \begin{thmc}
 Suppose $\kappa=\lambda^+=2^\lambda$. The following reductions are Borel-strict.
 \begin{enumerate}
  \item  Let $2^{\mathfrak{c}}\leq\lambda=\lambda^{<\omega_1}$ and $\delta<\kappa$. If $T$ is a non-classifiable theory, then $$\cong_{T}^\delta\ \reduc\ id_2\ \reduc\ \cong_{T}.$$
  \item Let $2^{\mathfrak{c}}\leq\lambda=\lambda^{<\omega_1}$. If $T_1$ is a classifiable non-shallow theory and $T_2$ is a non-classifiable theory, then: $$\cong_{T_2}^\lambda\ \reduc\ \cong_{T_1}\ \reduc\ \cong_{T_2}.$$
 \item Let $\kappa=\aleph_\gamma$ be such that $\beth_{\omega_1}(\mid\gamma\mid)\leq\kappa$. Suppose $T_1$ is a classifiable shallow theory, $T_2$ a classifiable non-shallow theory, and $T_3$ non-classifiable theory. Then $$\cong_{T_1}\ \redub\ \cong_{T_3}^\lambda\ \reduc\ \cong_{T_2}.$$

 \end{enumerate}
 
\end{thmc}

The previous results give us a counterpart of Shelah's main gap in the setting of Borel reducibility. In particular, Theorem A and B show us that the Borel reducibility complexity notion is a refinement of the complexity notion given by the Main Gap. 
This answers Question \ref{MainGapQ} in a positive way. Recall that this question was initially asked by Friedman, Hyttinen, and Weinstein in \cite{FHK13}, later re-stated in \cite[Question 3.45]{KLLS}, \cite[Question 4.7]{Mort}, \cite{FHK15}, and \cite{HKM}. 

These are the main results, naturally there are other results in this article, many of them follow from the proofs of the main results.
In the effort to answer Question \ref{MainGapQ}, the equivalence modulo $\omega$-cofinal, $=^2_\omega$ (see Definition \ref{clubrel}), was studied in the past. From \cite{HKM18} and \cite{FMR20} it is known that the following is consistent: 

\textit{If $T$ is a non-classifiable theory, then $=^2_\omega\ \reduc\ \cong_T$}.

In \cite[Question 4.9]{Mort} it was asked whether this is a theorem of ZFC.
We give a partial answer to this question with Corollary \ref{omega_case}. Only the case when $T$ is a superstable theory with the DOP remains without an answer.

\begin{quest}
Is $=^2_\omega$ Borel reducible to the isomorphism relation of any superstable theory with the DOP?
\end{quest}

In their work about the Borel reducibility, Mangraviti and Motto Ros \cite{MM} were interested in the gap between classifiable shallow theories and classifiable non-shallow theories. In \cite[Question 6.9]{MM} it was asked how big is this gap:

\begin{quest}\label{ManmotQ}
Let $\kappa=\aleph_\mu$ be such that $\beth_{\omega_1}(\mid\mu\mid)\leq\kappa$. How large can be the gap between $\cong_{T_1}$ and $\cong_{T_2}$ when $T_1$ is classifiable and shallow and $T_2$ is classifiable non-shallow?
\end{quest}

We study this question and make progress towards its solution, we show that continuous reductions give us a better picture of the gap, instead of Borel reductions. 

The results presented in this article convince us that the complexity between the isomorphism relation of theories should be studied by the whole spectrum of reductions, continuous reductions, Lipschitz reductions, etc, instead of limiting ourselves to Borel reductions only.

We summarize most of the Borel reducibility properties from the main results in the following tables. Let $\kappa$, $\lambda$, and $\gamma$ be such that $\kappa=\kappa^{<\kappa}=\lambda^+=2^\lambda$ and $\omega\leq\gamma< \lambda$:

\begin{itemize}


    \item  The Borel reduction $\cong_T\ \reduc\ =^2_\mu$ holds for a countable complete first-order theory in a countable vocabulary and a regular cardinal $\mu<\kappa$, depending on what kind of theory is $T$ and some properties of $\lambda$ (see Section \ref{Section_main_gap} for the definition of the combinatorial principles $\diamondsuit_\lambda$ and $\dl^*_{S^\kappa_\gamma}(\Pi^1_2)$).

    The following are the possible values of $\mu$:

\begin{center}

$\cong_T\ \reduc\ =^2_\mu$

\begin{tabularx}{0.8\textwidth} { 
  || >{\centering\arraybackslash}X 
  || >{\centering\arraybackslash}X 
  | >{\centering\arraybackslash}X 
  | >{\centering\arraybackslash}X || }
 \hline
  & & &\\
 Theory & $\lambda=\lambda^{\gamma}$ & $\diamondsuit_\lambda$ & $\dl^*_{S^\kappa_\gamma}(\Pi^1_2)$\\
  & & &\\
 \hline
  & & &\\
 Classifiable  & $\omega\leq\mu\leq\gamma$ & $\mu=\lambda$ & $\mu=\gamma$ \\
   & & &\\
 & & &\\
 Non-classifiable  & Independent & Independent & $\mu=\gamma$ \\
\hline
\end{tabularx}

\end{center}

\bigskip

   \item  The Borel reduction $=^2_\mu\ \reduc\ \cong_T$ holds for a countable complete first-order theory in a countable vocabulary and a regular cardinal $\mu<\kappa$, depending on what kind of theory is $T$ and some properties of $\lambda$. Denote by $\mathfrak{c}$ the cardinal $2^\omega$.

    The following are the possible values of $\mu$:

\begin{center}

$=^2_\mu\ \reduc\ \cong_T$

\begin{tabularx}{0.8\textwidth} { 
  || >{\centering\arraybackslash}X 
  || >{\centering\arraybackslash}X 
  | >{\centering\arraybackslash}X 
  | >{\centering\arraybackslash}X | |}
 \hline
  & & &\\
Theory & $\lambda=\lambda^{\gamma}$ & $2^{\mathfrak{c}}\leq\lambda=\lambda^{\gamma}$ & $2^{\mathfrak{c}}\leq\lambda=\lambda^{<\lambda}$\\
  & & &\\
 \hline
  & & &\\
 Classifiable  & $\mu$ doesn't exist & $\mu$ doesn't exist & $\mu$ doesn't exist \\
 & & &\\
 Stable Unsuperstable  & $\mu=\omega$ & $\mu=\omega$ & $\mu=\omega$ \\
  & & &\\
  Unstable   & $\omega\leq \mu\leq\gamma$ & $\omega\leq \mu\leq\gamma$ & $\omega\leq \mu\leq\lambda$ \\
  & & &\\
  & & &\\
 Superstable with the OTOP  & $\omega\leq \mu\leq\gamma$ & $\omega\leq \mu\leq\gamma$ & $\omega\leq \mu\leq\lambda$\\
  & & &\\
 Superstable with the DOP  & ? & $\omega_1\leq \mu\leq\gamma$ & $\omega_1\leq \mu\leq\lambda$\\
\hline
\end{tabularx}

\end{center}

\end{itemize}

\bigskip

\subsection{Assumptions}
Throughout this article $\kappa$ is a regular uncountable cardinal that satisfies $\kappa^{<\kappa}=\kappa$. This is the main assumption in generalized descriptive set theory, the reason is that there is no consensus on the definition of a $\kappa$-Borel set when $\kappa^{<\kappa}>\kappa$. Some work has been done in this direction, see \cite{AM}, \cite{DM} and \cite{HMV}. At this stage, Question \ref{MainGapQ} cannot be asked when $2^{<\kappa}>\kappa$.

The theories are first-order countable complete theories in a countable vocabulary,  
unless otherwise stated.

In Section \ref{Section_main_gap} we study the Borel reducibility of classifiable theories and non-classifiable theories. Most of the results concerning non-classifiable theories have the cardinality assumption: 

\textit{NC: $\kappa=\lambda^+=2^\lambda$, and $2^{\mathfrak{c}}\leq\lambda=\lambda^{\omega_1}$.}

The cardinal assumptions NC are requiered in the proof of the Borel reducibility Main Gap (Theorem \ref{Main_Gap}). This proof is divided in two different and independent reductions. These assumptions are needed for the construction of the models of non-classifiable theories (Section \ref{seccion1} to \ref{Section_modelsf}) and the use of Jensen's diamond principle (the keys of each reduction). All the results in these sections are presented in the most general way, with the least amount of assumptions, in their weakest version of the assumptions. In Section \ref{Section_main_gap} we use the main results of the previous sections, and Jensen's diamond principle. At the moment of putting all the assumptions together, these ones are reduced to NC.

In Theorem \ref{main_order_prop} the assumptions \textit{$\kappa=\lambda^+$, $2^\mathfrak{c}\leq\lambda=\lambda^{<\omega_1}$} ensure us the existence of an $\omega_1$-dense, $(<\kappa)$-stable, $(\kappa, \omega_1)$-nice, and $\kappa$-colorable order, which is the starting point for the construction of the Ehrenfeucht-Mostowski models and crucial for the proof of the Isomorphism Theorem (Theorem \ref{models_iso}). The assumption $\mathfrak{c}=2^\omega<\kappa$, ensure the existence of an $\omega_1$-saturated model of size less than $\kappa$, we use it for the construction of the linear order in Section \ref{seccion1}. The assumption $2^{\mathfrak{c}}\leq\lambda=\lambda^{<\omega_1}$ is used to count the number of types on the linear order, i.e. to prove that the linear order is $(<\kappa)$-stable.

Without the assumption $2^{\mathfrak{c}}\leq\lambda=\lambda^{\omega_1}$, the construction in Section \ref{seccion1} fails to have the properties: $\omega_1$-dense, $(<\kappa)$-stable, $(\kappa, \omega_1)$-nice, and $\kappa$-colorable. These are essential in the proof of the Isomorphism Theorem for non-classifiable theories with DOP. 
On the other hand, the assumption $\lambda=\lambda^{\omega}$ is enough for the linear order to have the properties $\omega$-dense, $(<\kappa)$-stable, $(\kappa, \omega)$-nice, and $\kappa$-colorable. These properties are enough for the Isomorphism Theorem to hold for the other cases of non-classifiable theories (stable unsupersatable, superstable with OTOP, unstable). 
It remains open whether the assumptions for the DOP case can be relaxed to $\lambda=\lambda^{\omega}$ (as in the non-DOP cases). 
In general, the constructed linear order and the Isomorphism Theorem can be obtained from the assumption $\kappa$ is strongly inaccessible, but the proof of the Borel reducibility Main Gap fails due to a different reason.

We use [Hyttinen-Kulikov-Moreno, \cite[Lemma 2]{HKM}], Fact \ref{HKM1}, for one of the reductions in the proof of the Borel reducibility Main Gap. The assumption of this result is $\diamondsuit_{\omega_1}$. The assumptions $\kappa=\lambda^+=2^\lambda$ and $\lambda=\lambda^{\omega_1}$ imply $cf (\lambda)>\omega_1$, and we use them to have $\diamondsuit_{\omega_1}$ (see \cite{Sh1}). When the assumption is relaxed to $\kappa=\lambda^+=2^\lambda$ and $\lambda=\lambda^{\omega}$, the proof of the Borel reducibility Main Gap still works for the non-DOP cases. Unfortunately, diamond is essential in the proof of Fact \ref{HKM1}. Thus, the proof of Fact \ref{HKM1} does not hold (in general) when $\kappa$ is strongly inaccessible.

It remains open whether the assumption $\kappa=\lambda^+$ can be exchange for $\kappa$ is strongly inaccessible.

Most of the results concerning classifiable shallow theories have the cardinality assumption:

\textit{CS: $\kappa=\aleph_\gamma$ is such that $\beth_{\omega_1}(|\gamma|)\leq\kappa$.}

The cardinal assumption  CS is precisely the bound obtain from Shelah's Main Gap Fact \ref{SHMGT}, for the number of non-isomorphic models of classifiable shallow theories of size $\kappa$. Under this assumptions, classifiable shallow theories have less than $\kappa$ non-isomorphic models of size $\kappa$. This give us a very good encoding power to construct Borel reductions.

As it is mentioned by Mangraviti and Motto Ros [second page of \cite{MM}], there are many cardinals satisfying CS and under GCH there are unboundedly many $\kappa$'s satisfying NC and CS. Thus it is easy to find cardinals satisfying both assumptions (NC and CS).

\begin{example}[GCH]
    Let us show classes of cardinals that satisfy the assumptions NC, CS, and both simultaneously, under GCH.
    \begin{itemize}
        \item {\bf NC.} Let $1\leq \alpha$ be an ordinal, $\kappa=\aleph_{\alpha+2}$, and $\lambda=\aleph_{\alpha+1}$. Clearly under GCH $\kappa=\lambda^+=2^\lambda$, this implies $\kappa^{<\kappa}=\kappa$. On the other hand $\lambda=\aleph_{\alpha+1}=2^{\aleph_{\alpha}}$ and $2^\mathfrak{c}=2^{\aleph_1}\leq 2^{\aleph_{\alpha}}=\lambda$, finally $$\lambda^{\omega_1}=(2^{\aleph_{\alpha}})^{\omega_1}=2^{\aleph_{\alpha}}=\lambda.$$
        Thus the minimum $\kappa$ that satisfy NC is $\aleph_3$ and any cardinal in the class $\{\aleph_{\alpha+2}\mid 1\leq\alpha\in Ord\}$ satisfies NC.
        \item {\bf CS.} Let $\alpha$ be an ordinal and $1\leq n$ a natural number, $\gamma=\alpha+\omega_1+n$ and $\kappa=\aleph_{\gamma}$. If $|\alpha|\leq\aleph_1$, then $|\gamma|=\aleph_1$ and $$\beth_{\omega_1}(|\gamma|)=\aleph_{\omega_1}< \aleph_{\alpha+\omega_1+n}=\kappa.$$ If $|\alpha|>\aleph_1$, then $|\gamma|=|\alpha|=\aleph_\delta$, for some $\delta$, and $\delta\leq\alpha$. Therefore $$\beth_{\omega_1}(|\gamma|)=\aleph_{\delta+\omega_1}< \aleph_{\alpha+\omega_1+n}=\kappa.$$
        Thus the minimum $\kappa$ that satisfy CS is $\aleph_{\omega_1+1}$ and any cardinal in the class $\{\aleph_{\alpha+\omega_1+n}\mid \alpha\in Ord\ \&\  n\in \omega\}$ satisfies CS.
        \item {\bf NC \& CS.}  The minimum $\kappa$ that satisfy NC and CS is $\aleph_{\omega_1+2}$ and any cardinal in the class $\{\aleph_{\alpha+\omega_1+2}\mid \alpha\in Ord\}$ satisfies NC and CS.
    \end{itemize}
\end{example}
\begin{example}[ZFC]
    In this caes GCh is not needed but some local CH is requiered. Let us show classes of cardinals that satisfy the assumptions NC, CS, and both simultaneously.
    \begin{itemize}
        \item {\bf NC.} An easy example for NC would be $\lambda$ a strong limit such that $cf (\lambda)>\omega_1$ and $\kappa=\lambda^+=2^\lambda$. The strong limit condition can be relaxed to 
        \begin{itemize}
            \item $\lambda>2^\mathfrak{c}$,
            \item $cf (\lambda)>\omega_1$,
            \item $2^{<\lambda}\leq \lambda$,
            \item $\kappa=2^\lambda=\lambda^+$.
        \end{itemize}
        A ZFC-definable class of possible examples for $\kappa$ is $\{(2^\mu)^+\mid \mu\ge\mathfrak{c}\ \&\ (2^\mu)^+=2^{(2^\mu)}\}$. Clearly in this case $\lambda=2^\mu$ implies $2^\mathfrak{c}\leq \lambda=\lambda^{\omega_1}$.
        \item {\bf CS.} In this case it is enough to have the following conditions:
            \begin{itemize}
            \item $\kappa=\aleph_{\alpha+1}=2^{\aleph_\alpha}$,
            \item $\beth_{\omega_1}(|\alpha+1|)\leq \aleph_{\alpha+1}$.
        \end{itemize}
        Notice that in CS, $\gamma$ cannot be a limit ordinal. Otherwise, since $\kappa$ is regular, $\kappa=cf(\kappa)=cf (\gamma)\leq |\gamma|\leq \kappa$ and $\beth_{\omega_1}(|\gamma|)\ge 2^{|\gamma|}=2^\kappa>\kappa$. So CS would no be satified.
        \item {\bf NC \& CS.}  Notice that for $\lambda=\aleph_\alpha$ a strong limit such that $cf (\lambda)>\omega_1$ and $|\alpha|<\lambda$, it is enough to have $\kappa=\lambda^+=2^\lambda$, to satisfy NC and CS.
    \end{itemize}
\end{example}

\subsection{Related work}\label{subsection1.3}

It was established that Borel reducibility is a good candidate as a complexity notion in GDST. By understanding some previous results we can understand the different notions of complexity given by reductions.

The \textit{Generalized Baire Space}, abbreviated by GBS, is the set $\kappa^\kappa$ endowed with the bounded topology. In this topology the basic open sets are of the form $$[\zeta]=\{\eta\in \kappa^\kappa\mid \zeta\subseteq \eta\}$$
where $\zeta\in \kappa^{<\kappa}$. The collection of \textit{$\kappa$-Borel} subsets of $\kappa^\kappa$ is the smallest set that contains the basic open sets and is closed under complement, unions and intersections both of length $\kappa$. 
The sets of this collection are called $\kappa$-Borel. 
A subset $X\s \kappa^\kappa$ is a $\Sigma_1^1(\kappa)$ subset of $\kappa^\kappa$ if there is a closed set $Y\s \kappa^\kappa\times \kappa^\kappa$ such that the projection $\pr(Y):=\{x\in \kappa^\kappa\mid \exists y\in\kappa^\kappa,~(x,y)\in Y\}$ is equal to $X$. 
These definitions also extend to the product space $\kappa^\kappa\times\kappa^\kappa$. A subset $X\s \kappa^\kappa$ is a $\Delta_1^1(\kappa)$ set if $X$ and its complement are $\Sigma_1^1(\kappa)$. In this setting, the classes of $\kappa$-Borel and $\Delta_1^1(\kappa)$ subsets do not coincide \cite[Theorem 18 (1)]{FHK13}.

The \textit{Generalized Cantor Space}, abbreviated by GCS, is the subspace $2^\kappa$ with the induced topology. The main equivalence relations that we will use are \textit{the equivalence modulo $\mu$-cofinal}, $=_\mu^2$, and \textit{the isomorphism relation}, $\cong_T$. The following relations can be defined in any subspace $\beta^\kappa$. 

\begin{defn}\label{clubrel}
Given $S\subseteq \kappa$ and $\beta\leq\kappa$, we define the equivalence relation $=_S^\beta\ \subseteq\ \beta^\kappa\times \beta^\kappa$, 
as follows $$\eta\mathrel{=^\beta_S}\xi \iff \{\alpha<\kappa\mid \eta(\alpha)\neq\xi(\alpha)\}\cap S \text{ is non-stationary}.$$
\end{defn}

Notice that $=_S^\beta$ is an interesting relation when $S$ is stationary, otherwise it has only one equivalence class.
Let $\mu<\kappa$ be a regular cardinal. We will denote by $=_\mu^\beta$ the relation $=_S^\beta$ when $S$ is the stationary set $S^\kappa_\mu:=\{\alpha<\kappa\mid cf(\alpha)=\mu\}$. Let us denote by $CUB$ the club filter on $\kappa$ and $=^\beta_{CUB}$ the relation $=_S^\beta$ when $S=\kappa$.

To define the isomorphism relation we need to code structures of size $\kappa$ with elements of the GBS. We can code structures of any size (not bigger than $\kappa$) with elements of GBS.

\begin{defn}\label{struct}
Let $\omega\leq \mu\leq\kappa$ be a cardinal and $\mathbb{L}=\{Q_m\mid m\in\omega\}$ be a countable relational language.
Fix a bijection $\pi_\mu$ between $\mu^{<\omega}$ and $\mu$. For every $\eta\in \kappa^\kappa$ define the structure $\mathcal{A}_{\eta\restriction\mu}$ with domain $\mu$ as follows:
For every tuple $(a_1,a_2,\ldots , a_n)$ in $\mu^n$ $$(a_1,a_2,\ldots , a_n)\in Q_m^{\mathcal{A}_{\eta\restriction\mu}}\Leftrightarrow Q_m \text{ has arity } n \text{ and }\eta(\pi_\mu(m,a_1,a_2,\ldots,a_n))>0.$$
\end{defn}

For every first-order theory in a relational countable language (not necessarily complete), we have coded the models of $T$ of size $\mu\leq\kappa$ in the GBS, $\kappa^\kappa$. In the same way we can define these structures in the GCS, $2^\kappa$. Notice that different elements of $\kappa^\kappa$ may code the same structure. 

\begin{defn}\label{def_iso_rel}
    Let $\omega\leq \mu\leq\kappa$ be a cardinal and $T$ a first-order theory in a relational countable language. We define the \textit{isomorphism relation of models of size $\mu$}, $\cong_T^\mu~\subseteq \kappa^\kappa\times \kappa^\kappa$, as the relation $$\{(\eta,\xi)|(\mathcal{A}_{\eta\restriction\mu}\models T, \mathcal{A}_{\xi\restriction\mu}\models T, \mathcal{A}_{\eta\restriction\mu}\cong \mathcal{A}_{\xi\restriction\mu})\text{ or } (\mathcal{A}_{\eta\restriction\mu}\not\models T, \mathcal{A}_{\xi\restriction\mu}\not\models T)\}$$
\end{defn}

Let us denote by $\cong_T$ the isomorphism relation of models of size $\kappa$ of $T$ (i.e. $\cong_T^\kappa$). To simplify notation we will refer to $\cong_T$ as the isomorphism relation of $T$. We will also denote by $\mathcal{A}_\eta$ the structure $\mathcal{A}_{\eta\restriction\kappa}$, for obvious reasons.

By defining a good complexity notion for equivalence relations, we can study the complexity of $T$ by studying the complexity of  $\cong_T$. 
Let $\beta,\theta\in \{2,\kappa\}$, and $E_1$ and $E_2$ be
equivalence relations on $\beta^\kappa$ and $\theta^\kappa$, respectively. We say that 
\emph{$E_1$ is reducible to $E_2$} if there is a function $f\colon
\beta^\kappa\rightarrow \theta^\kappa$ that satisfies $$(\eta,\xi)\in E_1\iff
(f(\eta),f(\xi))\in E_2.$$  We call $f$ a \textit{reduction} of $E_1$ to
$E_2$ and we denote by $E_1\redur E_2$ the existence of a reduction of $E_1$ to $E_2$. It is clear that $E_1\redur E_2$ holds if and only if $E_1$ doesn't have more equivalence classes than $E_2$.

When we restrict ourselves to the isomorphism relation, this complexity notion is equivalent to the model theory notion give by the number of non-isomorphic models, $$\cong_{T_1}\ \redur\ \cong_{T_2}\ \Leftrightarrow\ I(\kappa,T_1)\leq I(\kappa,T_2).$$
This complexity notion is not good for our purposes, since we want to compare the complexity of not classifiable shallow theories (by Theorem \ref{SHMGT}, these theories have $2^\kappa$ non-isomorphic models).
By demanding more on the reduction function (e.g. continuity), we obtain stronger complexity notions. 

A subset $D\s\beta^\kappa$ is said to be \emph{comeager} if  $D\supseteq\bigcap\mathcal D$ for some non-empty family $\mathcal D$ of at most $\kappa$-many dense  open subsets of $\beta^\kappa$. 
A subset of $\beta^\kappa$ is said to be \emph{meager} iff its complement is {comeager}. 

\begin{defn} [Reductions]\label{all_reduc}
Apart from a ``cardinality" reduction, $\redur$, we define the following notions which allow us to have a better spectrum of complexities.

\begin{itemize}
\item
{\bf Baire measurable reduction.}  A subset $B\s \beta^\kappa$ is said to have \emph{the Baire property} iff there is an open set $A\s \beta^\kappa$ for which the symmetric difference $A\symdiff B$ is meager.

A function $f\colon \beta^{\kappa} \rightarrow\theta^{\kappa}$ is said to be \emph{Baire measurable} if for any open set $A\s \theta^\kappa$, $f^{-1}[A]$ has the Baire property. 
The existence of a Baire measurable reduction of $E_0$ to $E_1$ is denoted by $E_0 \redum E_1$.

\item
{\bf Borel reduction.} 
A function $f\colon \beta^\kappa\rightarrow \theta^\kappa$ is said to be \emph{$\kappa$-Borel} if for any open set $A\subseteq \theta^\kappa$, $f^{-1}[A]$ is a $\kappa$-Borel set.
The existence of a $\kappa$-Borel reduction of $E_0$ to $E_1$ is denoted by $E_0 \redub E_1$\footnote{We use ``$\redub$" instead of ``$\leq_B$", because we will deal with the equivalence relations $=_S^\beta$ (Definition \ref{clubrel}) and the notation could become heavy for the reader.}.

\item
{\bf Continuous reduction.} 
The existence of a continuous reduction of $E_0$ to $E_1$ is denoted by $E_0 \reduc E_1$.

\item
{\bf Lipschitz reduction.} 
For all $\eta,\xi\in\beta^\kappa$, denote 
$$\Delta(\eta,\xi):=\min(\{\alpha<\kappa\mid \eta(\alpha)\neq\xi(\alpha)\}\cup\{\kappa\}).$$

A function $f\colon \beta^\kappa\rightarrow \theta^\kappa$ is said to be \emph{Lipschitz} if for all $\eta,\xi\in \beta^\kappa$,  $$\Delta(\eta,\xi)\le\Delta(f(\eta),f(\xi)).$$
The existence of a Lipschitz reduction of $E_0$ to $E_1$ is denoted by $E_0 \redul E_1$.
\end{itemize}
\end{defn}

Notice that the reductions are defined from weaker to stronger, e.g a Lipschitz reduction is a continuous reduction. Different notions of reduction provide us with a method to compare the complexities of two equivalence relations.  Two equivalence relations have ``very similar" complexities if they are  Lipschitz bireducible. On the other hand, The complexity of $E_0$ is ``far" from the complexity of $E_1$ if $E_0 \redul E_1$ and $E_1 \not\redum E_0$. Notice that $\not\redur$ is stronger than $\not\redum$.

It is clear that the isomorphism relation can be also defined in the GCS. Never the less the isomorphism relation of $T$ in GBS and in GCS are Lipschitz bireducible by the function $\mathcal{F}:\kappa^\kappa\rightarrow 2^\kappa$ given by 
$$\mathcal{F}(\eta)(\alpha)=\begin{cases} 1 &\mbox{if } \eta(\alpha)\neq 0\\
0 & \mbox{otherwise. } \end{cases}$$  

Since they are Lipschitz bireducible, we will use $\cong_T$ to denote both.

During the study of the interaction between classification theory and GDST, Friedman, Hyttinen, and Weinstein found some classification theory dividing lines in GDST. This supported the idea that there is a counterpart of the Main Gap, which ended in their conjecture.

\bigskip
\begin{fact}[Friedman-Hyttinen-Kulikov]\label{basics_FHK}
\mbox{}
\begin{enumerate}
    \item Let $\kappa^{<\kappa}=\kappa>2^\omega$. If  $T$ is classifiable and shallow, then $\cong_T$ is $\kappa$-Borel. \textup{\cite[Theorem 68]{FHK13}}
    \item Let $\kappa^{<\kappa}=\kappa>\omega$. If  $T$ is classifiable non-shallow, then $\cong_T$ is $\Delta_1^1(\kappa)$ not $\kappa$-Borel. \textup{\cite[Theorem 69 and 70]{FHK13}}
    \item Let $\kappa^{<\kappa}=\kappa>\omega$. If $T$ is unstable or stable with the OTOP, then $\cong_T$ is not $\Delta_1^1(\kappa)$. \textup{\cite[Theorem 71 (1) and (2)]{FHK13}}
    \item Let $\kappa^{<\kappa}=\kappa>\omega_1$. If $T$ is superstable with the DOP, then $\cong_T$ is not $\Delta_1^1(\kappa)$. \textup{\cite[Theorem 71 (3)]{FHK13}}    
    \item Let $\kappa^{<\kappa}=\kappa>\omega$ and for all $\lambda<\kappa$, $\lambda^\omega<\kappa$. If $T$ is stable unsuperstable, then $\cong_T$ is not $\kappa$-Borel.  \textup{\cite[Theorem 72]{FHK13}}
\end{enumerate}
\end{fact}

From Fact \ref{SHMGT} we conclude that if $\cong_T$ has less than $2^\kappa$ equivalence classes, then $\cong_T$ is $\kappa$-Borel.
When we turn our attention to classifiable shallow theories, we realize that these theories are classified by their depth. This raises the question whether the depth of a theory is related to the Borel rank of the isomorphism relation of the theory.

Let us denote by $\mathbf{B}(\kappa)$ the set of $\kappa$-Borel sets. We define the $\kappa$-Borel hierarchy by the following recursive definition.
$$\Sigma^0_1(\kappa):=\{A\subset \kappa^\kappa\mid A\textit{ is open}\},$$
$$\Pi^0_1(\kappa):=\{A\subset \kappa^\kappa\mid A\textit{ is closed}\},$$
$$\Sigma^0_\alpha(\kappa):=\left\{\bigcup_{\gamma<\kappa}A_\gamma\subset \kappa^\kappa\mid A_\gamma\in \bigcup_{1\leq\beta<\alpha}\Pi^0_\beta(\kappa)\right\},$$
$$\Pi^0_\alpha(\kappa):=\{\kappa^\kappa\backslash A\mid A\in \Sigma^0_\alpha(\kappa)\}.$$

The set of $\kappa$-Borel sets is $$\mathbf{B}(\kappa)=\bigcup_{1\leq\alpha<\kappa^+}\Sigma^0_\alpha(\kappa)=\bigcup_{1\leq\alpha<\kappa^+}\Pi^0_\alpha(\kappa)$$
If $A$ is a $\kappa$-Borel set, the smallest ordinal $1 \leq \alpha \leq \kappa^+$ such that $A \in \Sigma^0_\alpha(\kappa)\cup \Pi^0_\alpha(\kappa)$ is called the Borel rank of $A$ and it is denoted by
$rk_B(A)$.

Mangraviti and Motto Ros found a connection between the Borel rank of $\cong_T$ and the depth of $T$, when $T$ is a classifiable shallow theory  (see \cite[Definition X 4.2]{Sh90} for the definitions of depth and shallow). 

\begin{fact}[{Descriptive Main Gap; Mangraviti-Motto Ros, \cite[Theorem 1.9]{MM}}]
Let $\kappa$ be such that $\kappa^{<\kappa}=\kappa>2^{\omega}$.
 If $T$ is a classifiable shallow theory of depth $\alpha$, then $$rk_B(\cong_T)\leq 4\alpha.$$
\end{fact}

In the same article they studied the Borel reducibility between classifiable shallow theories and not classifiable shallow theories.

\begin{fact}[{Mangraviti-Motto Ros, \cite[Proposition 6.7]{MM}}]\label{Mangraviti_Motto}
Let $\kappa=\aleph_\gamma$ be such that $\beth_{\omega_1}(\mid\gamma\mid)\leq\kappa$ (CS assumption). Suppose $T_1$ is a classifiable shallow and $T_2$ not. Then there is an equivalence relation $E$ on $2^\kappa$ such that $$\cong_{T_1}\ \redub E\ \redub\ \cong_{T_2}$$ and $$\cong_{T_2}\ \not\redur E\ \not\redur\ \cong_{T_1}.$$
\end{fact}

What Mangraviti and Motto Ros proved was actually the following proposition, which is stronger. 

\begin{prop}[Mangraviti-Motto Ros, \cite{MM}]\label{ManMot}
    Let $E_1$ be a $\kappa$-Borel equivalence relation with $\gamma\leq\kappa$ equivalence classes and $E_2$ be an equivalence relation with $\theta$ equivalence classes. If $\gamma\leq\theta$, then $E_1\redub E_2$.
\end{prop}

This implies (under CS assumption, $T_1$, and $T_2$ classifiable shallow theories) that if $\cong_{T_1}$ and $\cong_{T_2}$ have the same number of equivalence classes, then they are Borel bireducible. 

For all cardinals $0<\varrho\leq\kappa$ let us define the \textit{counting $0$-classes} equivalence relation $0_\varrho\in \kappa^\kappa\times\kappa^\kappa$ as follows: $\eta\ 0_\varrho\ \xi$ if and only if one of the following holds:
\begin{itemize}
\item  $\varrho$ is finite:
    \begin{itemize}
        \item $\eta(0)=\xi(0)<\varrho-1$;
        \item $\eta(0),\xi(0)\ge\varrho-1$.
    \end{itemize}
\item $\varrho$ is infinite:

    \begin{itemize}
        \item $\eta(0)=\xi(0)<\varrho$;
        \item $\eta(0),\xi(0)\ge\varrho$.
    \end{itemize}
\end{itemize}

It is clear that for all $\varrho\leq\kappa$, $0_\varrho$ has $\varrho$ equivalence classes. By the previous proposition, if $\cong_T$ has $\varrho\leq\kappa$ equivalence classes, then $\cong_T$ and $0_\varrho$ are Borel bireducible.
We might think that when $\cong_T$ has $\varrho\leq\kappa$ equivalence classes, $\cong_T$ and $0_\varrho$ have the ``same" complexity. This is very counter-intuitive, since  $0_\varrho$ is a very simple equivalence relation.
In Lemma \ref{difcomplexity} we will show that it is not the case,  e.g. $0_\varrho\ \reduc\  \cong_T$ and $\cong_T\ \not\reduc\ 0_\varrho$.

Shelah developed stability theory and classification theory to answer Morley's conjecture. 
It is expected that the study of the Borel reducibility Main Gap can be used to study Morley's conjecture from the point of view of GDST. 

We ask whether $\cong_T^\mu$ is Borel reducible to $\cong_T$ for all $T$ first-order theory in a relational countable language, and $\omega<\mu<\kappa$?

By a simple observation we can answer this question for some theories using Fact \ref{SHMORL} and Proposition \ref{ManMot}.
Since $\kappa^{<\kappa}=\kappa$, for all $\mu<\kappa$ and for all first-order theories in a relational countable language $T$, the relation $\cong_T^\mu$ has at most $\kappa$ equivalence classes. Therefore by Fact \ref{SHMORL} and Proposition \ref{ManMot}, if $\cong_T^\mu$ is $\kappa$-Borel, then $\cong_T^\mu\redub \cong_T$. The following Fact and Fact \ref{basics_FHK} (1) tell us that $\cong_T^\mu\redub \cong_T$ holds for classifiable shallow theories (when $\kappa>2^\omega)$.
We generalize this in Section \ref{Section_main_gap}.

\begin{fact}[Folklore]\label{reduction_classify_descrpitive}
Suppose $E_0\ \redur\ E_1$. Then the following hold:
\begin{itemize}
\item If $E_1$ is $\kappa$-Borel and $E_0\ \redub\ E_1$, then $E_0$ is $\kappa$-Borel.
\item If $E_1$ is $\Delta^1_1(\kappa)$ and $E_0\ \redub\ E_1$, then $E_0$ is $\Delta^1_1(\kappa)$.
\item If $E_1$ is open and $E_0\ \reduc\ E_1$, then $E_0$ is open.
\end{itemize}
\end{fact}

A proof of Fact \ref{reduction_classify_descrpitive} can be found in \cite[Theorem 7]{HKM} . The proof is for $\kappa$-Borel$^*$ sets, a more general class of sets (defined by Blackwell  \cite{Bla}, see also \cite{MV}). 

As it was discussed previously, categoricity played a historical important role in the study of the spectrum function. It is not a surprise that $\kappa$-categoricity can be characterized in GDST.

\begin{fact}[{Mangraviti-Motto Ros, \cite[Theorem 3.3]{MM}}]\label{MM_categoric}
Let $T$ be a countable first-order theory in a countable vocabulary (not necessarily complete). $T$ is $\kappa$-categorical if and only if $rk_B(\cong_T)=0$, i.e. $\cong_T$ is clopen.
\end{fact}

Fact \ref{basics_FHK}, led Friedman, Hyttinen, and Weinstein to negative results about the Borel reducibility between non-classifiable theories and  classifiable theories. This was an important step to support their conjecture.

From Fact \ref{basics_FHK} and Fact \ref{reduction_classify_descrpitive}, Friedman, Hyttinen, and Weinstein deduced that  if $T_1$ is classifiable and $T_2$ is unstable or stable with the OTOP or superstable with the DOP and $\kappa>\omega_1$, then $\cong_{T_2}\ \not\redub\ \cong_{T_1}$.

They introduced the equivalence modulo different cofinalities to study the Borel reducibility of non-classifiable theories. This led them to positive results about their conjecture.

\begin{fact}[{Friedman-Hyttinen-Kulikov, \cite[Theorem 79]{FHK13}}]\label{previous_FHK}
Suppose that $\kappa=\lambda^+=2^\lambda$ such that $\lambda^{<\lambda}=\lambda$. If $T$ is unstable or superstable with the OTOP, then $=^2_\lambda\ \reduc\ \cong_T$. If additionally $\lambda\ge 2^\omega$, then $=^2_\lambda\ \reduc\ \cong_T$ holds also for superstable $T$ with the DOP.
\end{fact}

\begin{fact}[{Friedman-Hyttinen-Kulikov, \cite[Theorem 86]{FHK13}}]\label{previous_strictstb_FHK}
Suppose that $\kappa$ is such that for all $\lambda<\kappa$,  $\lambda^{\omega}<\kappa$. If $T$ is a stable unsuperstable theory, then $=^2_\omega\ \reduc\ \cong_T$.
\end{fact}

In \cite{HKM}, Hyttinen, Moreno, and Weinstein made notable progress by proving that the conjecture is consistent. The authors use the technique developed in \cite{HM}.

\begin{fact}[{Consistency: Borel reducibility Main Gap; Hyttinen-Kulikov-Moreno, \cite[Theorem 3.8]{HKM}}]\label{MGcons}
Suppose $\kappa=\kappa^{<\kappa}=\lambda^+$, $2^\lambda>2^\omega$ and $\lambda^{<\lambda}=\lambda$ hold in $V$. There is a forcing extension in which $\kappa=\kappa^{<\kappa}=\lambda^+$, $2^\lambda>2^\omega$ and $\lambda^{<\lambda}=\lambda$ hold, and the following statement also holds:
\textit{if $T_1$ is a classifiable theory and $T_2$ is not, then $\cong_{T_1}$ is Borel reducible to $\cong_{T_2}$ and there are $2^\kappa$ equivalent relations strictly between them.}
\end{fact}

In  the same article, using Fact \ref{previous_strictstb_FHK} with the technique developed in Fact \ref{MGcons},  it was proved the case of stable unsuperstable theories in ZFC.

\begin{fact}[{Strictly stable; Hyttinen-Kulikov-Moreno, \cite[Corollary 2]{HKM}}]\label{miangapstrictHNK}
Suppose that $\kappa=\lambda^+$ and $\lambda^{\omega}=\lambda$. If $T_1$ is a classifiable theory and $T_2$ is a stable unsuperstable theory, then $\cong_{T_1}\ \reduc\ \cong_{T_2}$ and $\cong_{T_2}\ \not\redub\ \cong_{T_1}$.
\end{fact}

Years later, this was improved by the author to cover the case of unstable theories.

\begin{fact}[{Unsuperstable; Moreno, \cite[Corollary 4.12]{Mor21}}]\label{unsup_old}
    Suppose $\kappa=\lambda^+=2^\lambda$ and $\lambda^\omega=\lambda$. If $T_1$ is a classifiable theory, and $T_2$ is an unsuperstable theory, then $\cong_{T_1}\ \reduc\ \cong_{T_2}$ and $\cong_{T_2}\ \not\redub\ \cong_{T_1}$.
\end{fact}

The proof of Fact \ref{unsup_old} uses a different approach from the one exhibited in Fact \ref{previous_FHK} and Fact \ref{previous_strictstb_FHK}. This is why Fact \ref{unsup_old} is a result about unsuperstable theories and not just about unstable theories.

Unfortunately the technique used in the proof of Fact \ref{unsup_old} was limited to unsuperstabel theories and could not cover superstable theories with the OTOP or the DOP. This is because Fact \ref{unsup_old} uses a tree construction, which is a construction for unsuperstable theories. In Theorem \ref{Main_Gap} we extend this construction to superstable theories with the OTOP or the DOP by using a density argument, coding trees by linear orders, and a homogeneity argument.

Fact \ref{MGcons} shows us that it is possible to force the complexities of classifiable theories and non-classifiable theories to be far apart, in the sense of having $2^{\kappa}$ equivalence relations strictly in between. Fact \ref{MGcons} doesn't tell us how different are the complexity of different non-classifiable theories. In \cite{HKM18} it was compared the complexity of different non-classifiable theories. Hyttinen, Moreno, and Weinstein proved that the following is consistent:

\textit{The isomorphism relations of two different non-classifiable theories are continuous bireducible.}

They show that this holds in the constructable model $L$, actually something stronger was proved.
An equivalence relation $E$ is $\Sigma_1^1(\kappa)$-complete if $E$ is a $\Sigma_1^1(\kappa)$ set and every $\Sigma_1^1(\kappa)$ equivalence relation $R$ is Borel reducible to $E$. 

\begin{fact}[{$L$-Main Gap Dichotomy; Hyttinen-Kulikov-Moreno, \cite[Theorem 4.11]{HKM18}}]\label{MGdichotomy}
($V=L$). Suppose $\kappa=\lambda^+$ and $\lambda$ is a regular uncountable cardinal. If T is a countable first-order theory in a countable
vocabulary, not necessarily complete, then one of the following holds:
\begin{itemize}
    \item $\cong_T$ is $\Delta^1_1(\kappa)$.
    \item $\cong_T$ is $\Sigma^1_1(\kappa)$-complete.
\end{itemize}
\end{fact}

From their work (\cite{FMR} and \cite{FMR20}), Fernandes, Moreno, and Rinot show that if $\kappa=\lambda^+=2^\lambda$, $\lambda^{<\lambda}=\lambda$, and $\lambda\ge 2^\omega$, then there is a $<\kappa$-closed $\kappa^+$-cc forcing extension in which the previous dichotomy holds. Even thought we now know that the dichotomy can be forced, we are still restricted to the cardinality assumptions, in particular ``$\kappa$ is the successor of $\lambda$ and $\lambda^{<\lambda}=\lambda$". We improve the dichotomy in the direction of these two assumptions.

\subsection{Outline}

In Section \ref{seccion1} We introduce the $F^\varphi_\omega$-isolation, the idea behind the construction of the $\kappa$-colorable linear orders. We generalize the notion of $(\kappa, bs,bs)$-nice, to construct $\kappa$-colorable $\varepsilon$-dense linear orders. We study the theory of $\kappa$-colorable linear orders from an abstract point of view and obtain an abstract generalization of the construction introduced in \cite{Mor21}.

In Section \ref{Sectio_trees} we present a detail study of the coloured ordered trees of any height. We use $\kappa$-colorable linear orders to construct coloured ordered trees. Coloured ordered trees will be used to construct skeletons of Ehrenfeucht-Mostowski models. 

In Section \ref{Section_modelsf} we use coloured ordered trees to construct Ehrenfeucht-Mostowski models of different non-classifiable theories. Each non-classifiable theory demands different assumptions on $\kappa$ to construct the models. Nevertheless, we show the isomorphism theorem, which is satisfied by all non-classifiable theories.

In Section \ref{Section_main_gap} we use the models from Section \ref{Section_modelsf} to  prove Theorem A. We study how big is the Borel reducibility gap and prove Theorem B, we use the idea of $\kappa$-Borel$^*$-sets to construct a game between models that captures the isomorphism. We introduce the $S$-recursive functions, most of the known reductions are $S$-recursive reductions. We introduce the counting $\alpha$-classes equivalent relation and give a detail picture of some gaps. We finish by studying Morley's conjecture from a GDST point of view and prove Theorem C.

\section{Linear Orders}\label{seccion1}
The notion of $\kappa$-colorable linear order was introduced in \cite{Mor21} to construct models of unsuperstable theories and prove Fact \ref{unsup_old}. $\kappa$-colorable linear order is a saturation notion that allows us to order colorable trees while preserving the isomorphism of trees, i.e. it allows us  to merge Shelah's ordered trees from \cite{Sh} and Hyttinen-Kulikov's coloured trees from \cite{HK}. 

In \cite{Mor21} a $\kappa$-colorable linear order was constructed in an inductive way. Due to the nature of the construction, the $\kappa$-colorable linear order presented in \cite{Mor21} has fundamental limitations, e.g. density arguments are not compatible with $\kappa$-colorable $(\kappa, bs,bs)$-nice linear order. To overcome these limitations we will make an abstract generalization of the inductive construction presented in \cite{Mor21}.

The goal of the construction of \cite{Mor21} was to obtain a linear order with the following properties (definitions below):
\begin{itemize}
    \item $(<\kappa, bs)$-stable
    \item $(\kappa, bs,bs)$-nice
    \item $\kappa$-colorable.
\end{itemize}
At first sight, it is not apparent that these three notions (a stability notion, a density notion, and a saturation notion respectively) should be compatible. Not only these notions are compatible, as it was showed in \cite{Mor21}, there is a dense linear order with these properties. To extend the results of \cite{Mor21} to other non-classifiable theories, we will need an $\varepsilon$-dense linear order. To construct such a linear order, we have to generalize the notion of $(\kappa, bs,bs)$-nice to $(\kappa, bs,bs,\varepsilon)$-nice.

Through this section, we will show the existence of such an order (under certain cardinality assumptions).

\textit{If there is a $\varepsilon$-dense model of DLO with size smaller than $\kappa$, then there is a model of DLO of size $\kappa$ such that it is
\begin{itemize}
    \item $(<\kappa, bs)$-stable (Definition \ref{stable2.4r} second item).
    \item $(\kappa, bs,bs,\varepsilon)$-nice (Definition \ref{epnice2.6r}).
    \item $\kappa$-colorable (Definition \ref{defcolorark}).
    \item $\varepsilon$-dense. (Definition \ref{epdense2.16r}).
\end{itemize}
   }

\subsection{The $F^{\varphi}_\omega$-saturation}

To generalize the construction made in \cite{Mor21}, we need to study the inductive construction. Let us recall the definitions and results that play a role in the construction of \cite{Mor21}. Let us start by $(<\kappa, bs)$-stable and $(\kappa, bs,bs)$-nice, which were introduced by Shelah in \cite{Sh}.

\begin{defn}[$\kappa$-representation]
Let $A$ be an arbitrary set of size at most $\kappa$. The sequence $\mathbb{A}=\langle A_\alpha\mid \alpha<\kappa  \rangle$ is a \textit{$\kappa$-representation} of $A$, if $\langle A_\alpha\mid \alpha<\kappa  \rangle$ is an increasing continuous sequence of subsets of $A$, for all $\alpha<\kappa$, $|A_\alpha|<\kappa$, and $\bigcup_{\alpha<\kappa}A_\alpha=A$.
\end{defn}

Notice that for any two representations $\mathbb A$ and $\mathbb A'$ of $A$, there is a club $C$ such that for all $\alpha\in C$, $A_\alpha=A'_\alpha$.

For any $\mathcal{L}$-structure $\mathcal{A}$ we denote by \textit{at} the set of atomic formulas of $\mathcal{L}$ and by \textit{bs} the set of basic formulas of $\mathcal{L}$ (atomic formulas and negation of atomic formulas). For all $\mathcal{L}$-structure $\mathcal{A}$, $a$ a finite tuple of $\mathcal{A}$, and $B\subseteq \mathcal{A}$ we define $$tp_\Delta(a,B,\mathcal{A})=\{\varphi(x,b)\mid \mathcal{A}\models \varphi(a,b),\varphi\in \Delta,b\text{  a finite tuple of } B\}$$
for $\Delta$ a set of formulas of $\mathcal{L}$. Example of theses types are $tp_{at}(a,B,\mathcal{A})$ and $tp_{bs}(a,B,\mathcal{A})$.

\begin{defn}
Let $\mathcal{A}$ be a model, $a$ a finite tuple of $\mathcal{A}$, $B,D\subseteq \mathcal{A}$.
We say that $tp_{bs}(a,B,\mathcal{A})$ (\textit{bs,bs})\textit{-splits} over $D\subseteq \mathcal{A}$ if there are $b_1$ and $b_2$ finite tuples of $\mathcal{B}$ such that $tp_{bs}(b_1,D,\mathcal{A})=tp_{bs}(b_2,D,\mathcal{A})$ but $tp_{bs}(a^\frown b_1,D,\mathcal{A})\neq tp_{bs}(a^\frown b_2,D,\mathcal{A})$.
\end{defn}

\begin{defn}
Let $A$ be an arbitrary set of size at most $\kappa$ and $\mathbb{A}$ a $\kappa$-representation of $A$. We define $Sp_{bs}(\mathbb{A})$ as the following set $$\{\delta<\kappa\mid \delta\text{ a limit ordinal, }\exists a\in A\ [\forall\beta<\delta \ (tp_{bs}(a,A_\delta,A)\text{ (bs,bs)-splits over }A_\beta)]\}.$$
\end{defn}

\begin{defn}\label{stable2.4r}
\begin{itemize}
\item Let $\mathcal{A}$ be a model of size at most $\kappa$. We say that $A$ is \textit{$(\kappa, bs,bs)$-nice} if $Sp_{bs}(\mathbb{A})$ is non-stationary for some (equivalently, any) $\kappa$-representation $\mathbb{A}$ of $A$ (i.e. $Sp_{bs}(\mathbb{A})\ =^2_{CUB}\ \emptyset$).
\item A linear order $A$ is \textit{$(<\kappa, bs)$-stable} if for every $B\subseteq A$ of size smaller than $\kappa$, $$\kappa>|\{tp_{bs}(a,B,A)\mid a\text{  a finite tuple of } A\}|.$$
\end{itemize}
\end{defn}

As it was mentioned before, $(<\kappa, bs)$-stable and $(\kappa, bs,bs)$-nice were used in \cite{Mor21} to construct continuous reductions from classifiable theories to unsuperstable theories. Unfortunately, this is not enough when we deal with superstable non-classifiable theories. We will need to generalize $(\kappa, bs,bs)$-nice to construct models of superstable non-classifiable theories. 

\begin{fact}[{Hyttinen-Tuuri, \cite[Lemma 8.12]{HT}}]\label{basic_nice}
Let $A$ be a linear order of size $\kappa$ and $\langle A_\alpha\mid\alpha<\kappa\rangle$ a $\kappa$-representation. Then the following are equivalent:
\begin{enumerate}
\item $A$ is $(\kappa, bs,bs)$-nice.
\item There is a club $C\subseteq \kappa$, such that for all limit $\delta\in C$, for all $x\in A$ there is $\beta<\delta$ such that one of the following holds:
\begin{itemize}
\item $\forall\sigma\in A_\delta [\sigma\ge x \Rightarrow \exists\sigma'\in A_\beta\ (\sigma\ge \sigma'\ge x)]$
\item $\forall\sigma\in A_\delta [\sigma\leq x \Rightarrow \exists\sigma'\in A_\beta\ (\sigma\leq \sigma'\leq x)]$
\end{itemize}
\end{enumerate}
\end{fact}

This characterization shows us that $(\kappa, bs,bs)$-nice is a notion of density.
This motivates the definition of $(\kappa, bs,bs,\varepsilon)$-nice, which generalizes $(\kappa, bs,bs)$-nice in the direction of density. This notion will allow us to abstract the construction from \cite{Mor21}.
\begin{defn}\label{epnice2.6r}
Let $\varepsilon<\kappa$ be a regular cardinal, $A$ be a linear order of size $\kappa$ and $\langle A_\alpha\mid\alpha<\kappa\rangle$ a $\kappa$-representation. Then $A$ is \textit{$(\kappa, bs,bs,\varepsilon)$-nice} if
there is a club $C\subseteq \kappa$, such that for all limit $\delta\in C$ with $cf(\delta)\ge\varepsilon$, for all $x\in A$ there is $\beta<\delta$ such that one of the following holds:
\begin{itemize}
\item $\forall\sigma\in A_\delta [\sigma\ge x \Rightarrow \exists\sigma'\in A_\beta\ (\sigma\ge \sigma'\ge x)]$
\item $\forall\sigma\in A_\delta [\sigma\leq x \Rightarrow \exists\sigma'\in A_\beta\ (\sigma\leq \sigma'\leq x)]$
\end{itemize}

\end{defn}

By Fact \ref{basic_nice}, $A$ is $(\kappa, bs,bs,\omega)$-nice if and only if it is $(\kappa, bs,bs)$-nice. 
Finally, we provide the definition of $\kappa$-colorable linear order, introduced in \cite{Mor21} to construct order coloured trees.

\begin{defn}\label{defcolorark}
Let $I$ be a linear order of size $\kappa$. We say that $I$ is \textit{$\kappa$-colorable} if there is a function $F: I\rightarrow\kappa$ such that for all $B\subseteq I$, $|B|<\kappa$, $b\in I\backslash B$, and $p=tp_{bs}(b,B,I)$ such that the following hold:
For all $\alpha\in \kappa$, $$|\{a\in I\mid a\models p\ \&\  F(a)=\alpha\}|=\kappa.$$
\end{defn}

The idea behind $\kappa$-colorable is that any realizable type over a small set, is realized by $\kappa$ many elements. 

\begin{fact}[{Moreno, \cite[Theorem 2.25]{Mor21}}]\label{mor21_order}
    There is a $(<\kappa, bs)$-stable $(\kappa, bs,bs,\omega)$-nice $\kappa$-colorable linear order of size $\kappa$.
\end{fact}

This linear order allows us to merge Shelah's ordered trees from \cite{Sh} and Hyttinen-Kulikov's coloured trees from \cite{HK} (i.e. construct ordered coloured trees). This was the main difficulty in the proof of Fact \ref{unsup_old}. This could be easily done by using a saturated dense linear order without end-point (when a saturated model exists). Unfortunately DLO (the theory of dense linear orderings without end-point) is an unstable theory, thus a saturated model of DLO is not $(<\kappa)$-stable. A saturated model of DLO is not useful for our purpose since $(<\kappa)$-stable plays an important role in the proof of the isomorphism theorem of Ehrenfeucht-Mostowski models (Theorem \ref{models_iso}). The notion of $\kappa$-colorable linear order was introduced to overcome this difficulty. $\kappa$-colorable give us enough saturation to merge Shelah's trees and Hyttinen-Kulikov's trees, and at the same time it behaves nice with $(<\kappa)$-stability.

The $\kappa$-colorable notion resembles the notion of $\bf F$-saturation from \cite{Sh90} Chapter 4. We may consider $\kappa$-colorable a saturation notion. Let us study the relation between $\kappa$-colorable and saturation.
We will use the axiomatic approach of $\bf F$-saturation from \cite{Sh90} Chapter 4.

Let $\Delta$ be a set of formulas of $\mathcal{L}$, we denote $S^m_{\Delta}(A)$ the set of all consistent types of $\Delta$-formulas over $A$ in $m$ variables and $S_{\Delta}(A)=\cup_{m<\omega}S^m_{\Delta}(A)$. We omit $\Delta$ when we refer to all the formulas of $\mathcal{L}$. In \cite{Sh90} Chapter 4 Shelah defined an isolation notion, an isolation notion provides us with a notion of saturation and construction. An $\bf F$-primary model has these two properties (saturation and construction). We are interested on these models, we have to chose the right saturation notion so we avoid the saturated model of DLO (in case it exists). We know from \cite{Mor21} that there are at least two non-isomorphic models of DLO that are $\kappa$-colorable. Thus $\kappa$-colorable is not an $\bf F$-saturation notion. Nevertheless, it comes from a weak version of isolation.

\begin{defn}
    Let $\varphi(x,y)\in bs$, we define $F^{\varphi}_\omega$ in the following way. Let $|B|<\kappa$ and $p\in S_{bs}(B)$, $(p,A)\in F^{\varphi}_\omega$ if and only if $A\subseteq B$, $A$ is finite, and there is $a\in A$ such that $$\{\varphi(x,a),x=a\}\cap p\neq\emptyset\ \&\ a\models p\restriction B\backslash\{a\}.$$ 
\end{defn}

It is easy to see that $F^{\varphi}_\omega$ satisfies the first axioms of $\bf F$-isolation (see \cite{Sh90} Chapter 4). But when $\varphi(x,y)$ is $x>y$, $F^{\varphi}_\omega$ does not satisfies Axiom V. 
It is clear that if $\varphi(x,a)\cup tp_{bs}(a,B,\mathcal{M})$ is consistent, then $(\varphi(x,a)\cup tp_{bs}(a,B,\mathcal{M}),\{a\})\in F^{\varphi}_\omega$. We define $F^{\varphi}_\omega$-constructible, $F^{\varphi}_\omega$-saturated, and $F^{\varphi}_\omega$-primary following the ideas of $\bf F$-isolation.

\begin{defn}
A sequence $(A,(a_i,B_i)_{i<\alpha})$ is an $F^{\varphi}_\omega$-construction over $A$ if for all $i<\alpha$, $(tp_{bs}(a_i,A_i),B_i)\in F^{\varphi}_\omega$ where $A_i=A\cup\bigcup_{j<i}a_j$.

$C$ is $F^{\varphi}_\omega$-constructible over $A$ if there is an $F^{\varphi}_\omega$-construction over $A$ such that $C=A\cup\bigcup_{j<\alpha}a_j$.
\end{defn}

\begin{defn}
    $C$ is $(F^{\varphi}_\omega, \varsigma)$-saturated if for all $B\subseteq C$ of size smaller than $\varsigma$, and $p\in S_{bs}(B)$, $(p,A)\in F^{\varphi}_\omega$ implies that $p$ is realized in $C$.
\end{defn}

\begin{defn}
    $C$ is $(F^{\varphi}_\omega,\varsigma)$-primary over $A$ if it is $F^{\varphi}_\omega$-constructible over $A$ and $(F^{\varphi}_\omega, \varsigma)$-saturated.
\end{defn}

Let us recall the order constructed in \cite{Mor21}.

\begin{defn}
Let $\mathbb{Q}$ be the linear order of the rational numbers. 
Let $\kappa\times\mathbb{Q}$ be ordered by the lexicographic order, $J^0$ be the set of functions $f:\omega\rightarrow\kappa\times\mathbb{Q}$ such that $f(\alpha)=(f_1(\alpha),f_2(\alpha))$, for which $|\{\alpha\in\omega\mid f_1(\alpha)\neq 0\}|$ is smaller than $\omega$. If $f,g\in J^0$, then $f<g$ if and only if $f(\alpha)<g(\alpha)$, where $\alpha$ is the least number such that $f(\alpha)\neq g(\alpha)$.
\end{defn}

Suppose $i<\kappa$ is such that $J^i$ has been defined. For all $\nu\in J^i$ let $\nu^{i+1}$ be such that $$\nu^{i+1}\models tp_{bs}(\nu, J^i\backslash \{\nu\}, J^i)\cup \{\nu>x\}.$$ Let $J^{i+1}=J^i\cup \{\nu^{i+1}\mid \nu\in J^i\}$. If $i\leq\kappa$ is a limit ordinal such that for all $j<i$, $J^j$ has been defined, then $J^{i}=\bigcup_{j<i}J^j$. Let $J=J^\kappa$. 

\begin{lemma}
     $J$ is $(F^{\varphi}_\omega,\kappa)$-primary over $J^0$, for $\varphi(x,y):= ``y>x"$. 
\end{lemma}
\begin{proof}
    It is easy to see from the construction that $J$ is $(F^{\varphi}_\omega, \varsigma)$-saturated. To show that it is $F^{\varphi}_\omega$-constructible over $J^0$, we need to find the right indexing for the elements of $J\backslash J^0$. For every $i<\kappa$, let $\{a^{i+1}_j\mid j<\kappa\}$ be an enumeration of $J^{i+1}\backslash J^i$. Let us order the set $\{\alpha<\kappa\mid \exists\beta<\kappa\ (\alpha=\beta+1)\}\times\kappa$ and let $\bar\kappa$ be the ordinal with the order type of $\{\alpha<\kappa\mid \exists\beta<\kappa\ (\alpha=\beta+1)\}\times\kappa$. We can index the elements of $J\backslash J^0$ by $\bar{\kappa}$, using the order type bijection between $\{\alpha<\kappa\mid \exists\beta<\kappa\ (\alpha=\beta+1)\}\times\kappa$ and $\bar\kappa$. The proof follows from the enumeration of the elements of $J^{i+1}\backslash J^i$ and the construction of $J$.
\end{proof}

As we can see, the construction in \cite{Mor21} was a $F^{\varphi}_\omega$-construction. Thus all the notions introduced in \cite{Mor21} (\textit{depth}, \textit{complexity}, \textit{road}, and \textit{generator}, see the following subsection) can be defined using $F^{\varphi}_\omega$-isolation and the right $F^{\varphi}_\omega$-construction. Therefore, these notions can be generalized to other structures (not only linear orders) by a different formula $\varphi$.

We will generalized this construction as follow: instead of adding one element in the $F^{\varphi}_\omega$-construction (i.e. $a_i=\nu^{i+1}$), we add an infinite number of elements ordered by a linear order (e.g. Hausdorff's $\eta$-sets \cite[VI 8]{Haus}).
To describe this generalization in an easy way, we will need a more abstract description of $J$. 

\begin{defn}[Generator Order] Let $Gen$ be the set of functions $f:\omega\rightarrow \kappa$ such that the following holds:
\begin{itemize}
    \item $f(0)=0$.
    \item For all $n<\omega$, $f(n)$ is either $0$ or a successor ordinal.
    \item There is $m<\omega$ such that for all $n>m$, $f(n)=0$.
    \item $f\restriction m+1\backslash\{0\}$ is strictly increasing.
\end{itemize}
    Let $f,g\in Gen$ and $i$ the least number such that $f(i)\neq g(i)$. Let us define $<_{Gen}$ as follows, $g<_{Gen}f$ if and only if one of the following holds:
    \begin{itemize}
        \item $f(i)=0$,
        \item $g(i)<f(i)$.
    \end{itemize}
\end{defn}

It is easy to check that $J$ is isomorphic to $J^0\times Gen$ with the lexicographic order (see \cite[Definition 2.15]{Mor21}).

\subsection{Constructing linear orders}

To construct models of non-classifiable superstable theories, we will need dense linear orders. Thus, we need to add a new density property to the linear order of Fact \ref{mor21_order}.

\begin{defn}\label{epdense2.16r}
Let $I$ be a linear order of size $\kappa$ and $\varepsilon$ a regular cardinal smaller than $\kappa$. We say that $I$ is \textit{$\varepsilon$-dense} if the following holds.

\textit{If $A,B\subseteq I$} are subsets of size less than $\varepsilon$ such that for all $a\in A$ and $b\in B$, $a<b$, then there is $c\in I$, such that for all $a\in A$ and $b\in B$, $a<c<b$.

\end{defn}

Notice that an $\varepsilon$-dense linear order $I$, is an $\eta_\alpha$-set where $\epsilon=\aleph_\alpha$ \cite[VI 8]{Haus}. It is known that there is an $\varepsilon$-saturated model of DLO of size $\varepsilon^{<\varepsilon}$, clearly this model is an $\varepsilon$-dense linear order. To avoid misunderstandings with Section \ref{Sectio_trees} notation $\eta_k$, we do not use Hausdorff notion $\eta$-sets. The $\varepsilon$-saturation notion is the usual saturation notion.

We want to construct a $(\kappa, bs,bs,\varepsilon)$-nice $\mu$-dense linear order. It is clear that the difficulty of the construction lies on the fact that both are density notions with different properties. It is clear that if $l$ is a $(\kappa, bs,bs,\varepsilon)$-nice linear order, then for all $\mu\ge \varepsilon$, $l$ is a $(\kappa, bs,bs,\mu)$-nice linear order. On the other hand if $l$ is a $\mu$-dense linear order, then for all $\varepsilon\leq \mu$, $l$ is a $\varepsilon$-dense linear order. Therefore, we will focus our study on the case of $(\kappa, bs,bs,\varepsilon)$-nice $\varepsilon$-dense linear order.

{\bf Notation.} Since we will only use basic formulas when dealing with these notions, we will denote by $(\kappa,\varepsilon)$-nice and $(<\kappa)$-stable the notions $(\kappa, bs,bs,\varepsilon)$-nice and $(<\kappa, bs)$-stable, respectively.

We can now proceed with the construction of the desired linear order (under some cardinal assumptions).
Let us fix a regular cardinal $\omega\leq\varepsilon$ such that $\varepsilon^{<\varepsilon}<\kappa$.
From now on let $\mathcal{Q}$ be an $\varepsilon$-saturated model of $DLO$ of size $\theta<\kappa$. It is clear that $\theta\ge\varepsilon^{<\varepsilon}$.

\begin{defn}
Let $\kappa\times\mathcal{Q}$ be ordered by the lexicographic order, $\mathcal{I}^0$ be the set of functions $f:\varepsilon\rightarrow\kappa\times\mathcal{Q}$ such that $f(\alpha)=(f_1(\alpha),f_2(\alpha))$, for which $|\{\alpha\in\varepsilon\mid f_1(\alpha)\neq 0\}|$ is smaller than $\varepsilon$. If $f,g\in \mathcal{I}^0$, then $f<g$ if and only if $f(\alpha)<g(\alpha)$, where $\alpha$ is the least number such that $f(\alpha)\neq g(\alpha)$.
\end{defn}

Now let us use the order $\mathcal{I}^0$ to construct an $\varepsilon$-dense $(<\kappa)$-stable $(\kappa,\varepsilon)$-nice $\kappa$-colorable linear order.

Let us fix $\tau\in \mathcal{Q}$. Let $I$ be the set of functions $f:\varepsilon\rightarrow(\{0\}\times \mathcal{I}^0)\cup (\kappa\times \mathcal{Q})$ such that the following hold
\begin{itemize}
    \item $f\restriction \{0\}:\{0\}\rightarrow\{0\}\times \mathcal{I}^0$.
    \item $f\restriction \varepsilon\backslash\{0\}:\varepsilon\backslash\{0\}\rightarrow \kappa\times \mathcal{Q}$;
    \item There is $\alpha<\varepsilon$ ordinal such that $\forall\beta>\alpha$, $f(\beta)=(0,\tau)$. We say that the least $\alpha$ with such property is the \textit{depth} of $f$ and we denote it by $dp(f)$.
    \item There are functions $f_1:\varepsilon\rightarrow \kappa$ and $f_2:\varepsilon\rightarrow \mathcal{I}^0\cup \mathcal{Q}$ such that $f(\beta)=(f_1(\beta),f_2(\beta))$ and $f_1\restriction dp(f)+1$ is strictly increasing.
\end{itemize}

Notice that since for all $f\in I$, $f(0)\in \{0\}\times \mathcal{I}^0$, so $f(0)\neq (0,\tau)$ and the depth of $f$ is well defined. Also, $f_1(\beta)=0$ if and only if either $\beta=0$, or $\beta>dp(f)$ and $f_2(\beta)=\tau$.
For all $f\in I$ with depth $\alpha$, define $o(f)=f_1(\alpha)$ the \textit{complexity} of $f$. Notice that for all $f\in I$, $f_1(dp(f)+1)=0$, and $f_1(dp(f))= 0$ if and only if $dp(f)=0$.
We say that $f<g$ if and only if one of the following holds:
\begin{itemize}
    \item $f(0)\neq g(0)$ and $f_2(0)<g_2(0)$.
    \item Let $\alpha=dp(g)$, $\forall\beta\leq \alpha$, $f(\beta)=g(\beta)$ and $f_1(\alpha+1)\neq 0$.
    \item Exists $\alpha>0$ such that $\forall\beta<\alpha$, $f(\beta)=g(\beta)$, and $f_1(\alpha),g_1(\alpha)\neq 0$ and $g(\alpha)>f(\alpha)$.
\end{itemize}

Notice that the set $$I^0=\{f\in I\mid f:\varepsilon\rightarrow (\{0\}\times \mathcal{I}^0)\cup (\{0\} \times \{\tau\})\}$$ with the induced order is isomorphic to $\mathcal{I}^0$. 

For every $i\leq\kappa$ let us define the order $I^{i+1}$ by $$I^{i+1}=\{f\in I\mid o(f)\leq i+1\}.$$ Suppose $i$ is a limit ordinal such that for all $j<i$, $I^j$ is defined, let $$I^i=\bigcup_{j<i}I^j.$$

Notice that $f\in I^{o(f)}$ if and only if $o(f)$ is a successor ordinal.

Let us proceed to define the $\kappa$-representations $\langle I^i_\alpha\mid \alpha<\kappa\rangle$ for every $i<\kappa$. 
Define $\langle \mathcal{I}^0_\alpha\mid\alpha<\kappa\rangle$ by $$\mathcal{I}^0_\alpha=\{\nu\in \mathcal{I}^0\mid \nu_1(n)<\alpha\textit{ for all } n<\varepsilon\}$$
it is clear that $\langle \mathcal{I}^0_\alpha\mid\alpha<\kappa\rangle$ is a $\kappa$-representation. Let us define $\langle I^0_\alpha\mid\alpha<\kappa\rangle$ in the canonical way following the definition of $I^0$.

Suppose $i<\kappa$ is such that $\langle I^i_\alpha\mid \alpha<\kappa\rangle$ has been defined. For all $\alpha<\kappa$ let $$I^{i+1}_\alpha=\{f\in I\mid o(f)\leq i+1\ \&\ f_2(0)\in \mathcal{I}^0_\alpha\},$$
notice that $f(0)\in \{0\}\times \mathcal{I}^0_\alpha$ holds if and only if $f\in I^0_\alpha$.

If $i<\kappa$ is a limit ordinal, then $$I^i_\alpha=\cup_{j<i}I^j_\alpha.$$ 
Notice that $I=\bigcup_{j<\kappa}I^i$.
Let us check that if for all $\beta<\kappa$, $\beta^{<\varepsilon}<\kappa$, then $\langle I^i_\alpha\mid \alpha<\kappa\rangle$ is a $\kappa$-representation. Notice that since $\langle I^0_\alpha\mid \alpha<\kappa\rangle$ is a $\kappa$-representation, for all $\beta<\alpha$, $I^0_\beta\subseteq I^0_\alpha$. Therefore, for all $i<\kappa$ and $\beta<\alpha<\kappa$, $I^i_\beta\subseteq I^i_\alpha$. On the other hand for all $f\in I$, $o(f)=f_1(\alpha)$, for $\alpha=dp(f)$. Since $f_1\restriction dp(f)+1$ is strictly increasing, $$|I^{o(f)}_\beta|\leq |o(f)^{< \varepsilon}\times I^0_\beta|<\kappa$$

Let us define the $\kappa$-representation $\langle I_\alpha\mid \alpha<\kappa\rangle$ by $$I_\alpha=I^\alpha_\alpha.$$
Now we can state the main result of this section.

\begin{thm}\label{main_order_prop}
    Suppose $\kappa$ is inaccessible, or $\kappa=\lambda^+$, $2^\theta\leq\lambda=\lambda^{<\varepsilon}$. Then $I$ is $\varepsilon$-dense, $(<\kappa)$-stable, $(\kappa, \varepsilon)$-nice, and $\kappa$-colorable.
\end{thm}

The proof of the previous theorem is divided in four lemmas, one per property. Before we prove these lemmas, we need to define more notions related to $I$. Generators and roads were notions that arose naturally in \cite{Mor21} from the inductive construction. To use these notions to study $I$, we will need to define them in a non-inductive way.

\begin{defn}[Generators]
    For all $f\in I$ with depth $\alpha$, define the \textit{generator} of $f$, $Gen(f)$, by $$Gen(f)=\{g\in I\mid f\restriction \alpha+1 =g\restriction \alpha+1\}.$$
\end{defn}

\begin{fact}\label{descending_gen}
    Let $f,g\in I$ be such that $f\neq g$ and $g\in Gen(f)$, then $f>g$.
\end{fact}

\begin{proof}
    Let $f,g\in I$ be such that $f\neq g$ and $g\in Gen(f)$. So $f\restriction \alpha+1 =g\restriction \alpha+1$, where $dp(f)=\alpha$. Therefore $f(\alpha+1)=(0,\tau)$ and $g(\alpha+1)\neq (0,\tau)$. We conclude that $f>g$.
\end{proof}

\begin{fact}\label{inside_the_gen} 
    Let $i,\delta$ and $\nu$ be such that $\nu\in I^i_\delta$. For all $f\in Gen(\nu)$, $f\in I^{o(f)+1}_\delta$ (in case $o(f)$ is a successor, $f\in I^{o(f)}_\delta$).
\end{fact}

\begin{proof}
    It follows from the construction of $\langle I^i_\alpha\mid \alpha<\kappa\rangle$.
\end{proof}

\begin{fact}
    Let $f,\nu\in I$ be such that $f\in Gen(\nu)$. If $g\notin Gen(\nu)$, then $g<\nu$ if and only if $g<f$.
\end{fact}

\begin{proof}
    Let $g,f,\nu\in I$ be such that $f\in Gen(\nu)$ and $g\notin Gen(\nu)$. Since $\nu\in Gen(\nu)$, $g\neq \nu$. 
    
    ($\Rightarrow$) Let us assume that $g<\nu$.
    By Fact \ref{descending_gen} $\nu\notin Gen(g)$. If $g_2(0)<\nu_2(0)$, then since $\nu_2(0)=f_2(0)$, $g_2(0)<f_2(0)$.

    Let us suppose that $g_2(0)=\nu_2(0)$. There is $\beta<dp(\nu)$ and $\nu'\in I$ such that $dp(\nu')=\beta$, $g,\nu\in Gen (\nu')$, and $g(\beta+1)\neq \nu(\beta+1)$. Since $g<\nu$, $g(\beta)<\nu(\beta)=f(\beta)$. We conclude that $g<f$.

    ($\Leftarrow$) Let us suppose $g>\nu$. Since $f\in Gen(\nu)$, $g>\nu>f$.
\end{proof}

\begin{cor}\label{coro1}
    For all $\nu,f\in I$ such that $f\in Gen(\nu)$
    $$f\models tp_{bs}(\nu,I\backslash Gen(\nu),I)\cup \{\nu>x\}.$$
\end{cor}

\begin{cor}\label{coro2}
    For all $\nu,f\in I$ such that $f\in Gen(\nu)$. If $\sigma\in I$ is such that $\nu\ge \sigma\ge f$, then $\sigma\in Gen(\nu)$.
    
\end{cor}

\begin{defn}[Roads]
We say that a sequence $\langle \nu_i\mid i\leq\alpha\rangle$ of elements of $I$ is a road from $\nu_0$ to $\nu_\alpha$ if for all $i<j$, $\nu_i\neq\nu_j$, $\nu_j\in Gen(\nu_i)$ and it is maximal according to this property, i.e. if $\nu$ is such that $\nu\in Gen(\nu_0)$ and $\nu_\alpha\in Gen(\nu)$, then there is $j\leq \alpha$ such that $\nu=\nu_j$.
\end{defn}

\begin{fact}
    For all $\nu\in I$ with $ dp(\nu)=\alpha$, there is a road $\langle \nu_i\mid i\leq\alpha\rangle$ from $\nu_0$ to $\nu_\alpha$ such that $\nu_0\in I^0$ and $\nu_\alpha=\nu$. This road is unique and we call it \textit{the road from $I^0$ to $\nu$}.
\end{fact}

\begin{proof}
    For all $i\leq \alpha$ let $\nu_i$ be such that $\nu_i\restriction i+1=\nu\restriction i+1$ and $\nu_i\restriction \varepsilon \backslash i+1$ is constant to $(0,\tau)$. Clearly $\nu_i\in I$ for all $i\leq\alpha$, in particular $\nu_o\in I^0$. Since $ dp(\nu)=\alpha$, $\nu_\alpha=\nu$. From the definition of generator and the way we constructed the sequence $\langle \nu_i\mid i\leq\alpha\rangle$, for all $i<j$, $\nu_j\in Gen(\nu_i)$. To show that it is a road let us suppose, towards contradiction, that it is not a maximal sequence. So there is $\nu'\in I$ such that $\nu'\in Gen(\nu_0)$, $\nu_\alpha\in Gen(\nu')$, and for all $j\leq \alpha$ $\nu'\neq \nu_j$. Since $\nu_\alpha\in Gen(\nu')$, $\nu'\restriction dp(\nu')+1 =\nu_\alpha\restriction dp(\nu')+1=\nu\restriction dp(\nu')+1$. So $dp(\nu')\leq dp(\nu)=\alpha$. Therefore $\nu'=\nu_{dp(\nu')}$, a contradiction.

    To show that it is unique, let us suppose, towards contradiction, that it is not unique and there is another road $\langle \nu'_i\mid i\leq\alpha\rangle$ from $\nu'_0$ to $\nu'_\alpha$ such that $\nu'_0\in I^0$, $\nu'_\alpha=\nu$. From the definition of road, $\nu=\nu'_\alpha\in Gen(\nu'_i)$ for all $i\leq \alpha$. So $\nu'_i\restriction dp(\nu'_i)+1 =\nu\restriction dp(\nu'_i)+1$ and $\nu'_i=\nu_{dp(\nu'_i)}$, a contradiction.
\end{proof}

For any $\nu\in I$ with $ dp(\nu)=\alpha$ and the road $\langle \nu_i\mid i\leq \alpha\rangle$ from $I^0$, and $\beta\leq \alpha$, we call the sub-sequence  $\langle \nu_i\mid \beta\leq i\leq \alpha\rangle$, \textit{the road from $\nu_\beta$ to $\nu$}.

\begin{fact}\label{road_type}
    Let $\langle \nu_j\mid j\leq \alpha\rangle$ be the road from $I^0$ to $\nu_\alpha$. For all $i<\alpha$ $$\nu_\alpha\models tp_{bs}(\nu_i,I^{o(\nu_{i+1})}\backslash (Gen(\nu_{i+1})\cup \{\nu_i\}), I)\cup \{\nu_i>x\}$$
\end{fact}
\begin{proof}
    Let $\langle \nu_j\mid j\leq \alpha\rangle$ be the road from $I^0$ to $\nu_\alpha$ and $i<\alpha$. By Corollary \ref{coro1}, we know that $$\nu_\alpha\models tp_{bs}(\nu_i,I\backslash Gen(\nu_i),I)\cup \{\nu_i>x\}.$$ 
    It is enough to show that $$\nu_\alpha\models tp_{bs}(\nu_i,(I^{o(\nu_{i+1})}\backslash Gen(\nu_{i+1}))\cap Gen(\nu_i),I).$$ 
    Let $\sigma\in (I^{o(\nu_{i+1})}\backslash Gen(\nu_{i+1}))\cap Gen(\nu_i)$. Since $\sigma\in Gen(\nu_i)$, $$\nu_i\restriction dp(\nu_i)+1=\sigma\restriction dp(\nu_i)+1$$ and $\nu_i>\sigma$. 
    On the other hand $\sigma\in I^{o(\nu_{i+1})}$, so $o(\sigma)\leq o(\nu_{i+1})$. 
    Thus $$0<\sigma_1(dp(\nu_i)+1)\leq o(\sigma)\leq o(\nu_{i+1})=(\nu_{i+1})_1(dp(\nu_{i+1})),$$ on the other hand since $dp(\nu_{i+1})=dp(\nu_i)+1$ and $\sigma\notin Gen(\nu_{i+1})$, $$0<\sigma_1(dp(\nu_i)+1)< o(\nu_{i+1}).$$ By the definition of the complexity, $$o(\nu_{i+1})=(\nu_{i+1})_1(dp(\nu_i)+1).$$ 
    We conclude that $\nu_{i+1}>\sigma$. Finally, by Corollary \ref{coro1}, $$\nu_\alpha\models tp_{bs}(\nu_{i+1},I\backslash Gen(\nu_{i+1}),I).$$ 
    We conclude that $\nu_\alpha>\sigma$, as we wanted.
\end{proof}

\subsection{$\varepsilon$-dense}

Recall that $\mathcal{Q}$ is $\varepsilon$-saturated, we will show that the orders $\kappa\times\mathcal{Q}$, $I^0$, and $I$ are $\varepsilon$-dense.

\begin{lemma}\label{dense_basic_order}
$\kappa\times\mathcal{Q}$ is $\varepsilon$-dense.
\end{lemma}

\begin{proof}

        Let $\langle a_i\mid i<\theta_1\rangle$ and $\langle b_i\mid i<\theta_2\rangle$ be sequences of $\kappa\times\mathcal{Q}$ of length smaller than $\varepsilon$ such that for all $i<h<\theta_1$ and $j<l<\theta_2$, $a_i<a_h$, $b_l<b_j$, and $a_i<b_j$.  For all $i<\theta_1$ and $j<\theta_2$, let us denote $a_i=(a^1_i,a^2_i)$ and $b_j=(b^1_j,b^2_j)$

Let us start by the case when the sequence $\langle b_i\mid i<\theta_2\rangle$ is empty. Since $\theta_1<\varepsilon<\kappa$, there is $\alpha<\kappa$ such that for all $i<\theta_1$, $a_i^1<\alpha$. Let us fix $\tau'\in \mathcal{Q}$, and define $a=(\alpha,\tau')$. Therefore for all $i<\theta_1$, $a_i<a$.
        
        Let us show the case when the sequence $\langle b_i\mid i<\theta_2\rangle$ is non-empty. Let $a^1=\bigcap_{j<\theta_2}b^1_j$. Let us show that for all $i<\theta_1$, $a_i^1\leq a^1$. Let us suppose, towards contradiction, that there is $i<\theta_1$ such that $a^1_i>a^1$. Since there are no infinite descending sequences of ordinals, there is $j<\theta_2$ such that $b_j^1=a^1$. Therefore $a^1_i>b_j^1$, a contradiction.

        Let us show that there is $a^2$ such that for all $i<\theta_1$ and $j<\theta_2$, $a_i<(a^1,a^2)<b_j$. Let us define $A=\{a^2_i\mid j<\theta_1\ \&\ a^1=a^1_i\}$ and $B=\{b^2_j\mid j<\theta_1\ \&\ a^1=b^1_j\}$. 

        Let us start with the case $A=\emptyset$. Clearly in this case, for all $i<\theta_1$ and $a\in \mathcal{Q}$, $a_i<(a^1,a)$. Thus it is enough to find $a^2\in \mathcal{Q}$ such that for all $y\in B$, $a^2<y$. Since $\mathcal{Q}$ is $\varepsilon$-saturated and $|B|<\varepsilon$, such an $a^2$ exists. 
        
        Let us show the case $A\neq\emptyset$. Since $A,B\subseteq \mathcal{Q}$ and $|A|,|B|<\varepsilon$, there is $a^2$ such that for all $x\in A$ and $y\in B$, $x<a^2<y$. Let $i<\theta_1$ and $j<\theta_2$. If $a^1_i<a^1$, then $a_i<(a^1,a^2)$. Also, if $a^1<b^1_j$, then $(a^1,a^2)<b_j$. Finally, if $a^1_i=a^1$, then by the way $a^2$ was defined, $a_i<(a^1,a^2)$. In the same way, if $b^1_j=a^1$, then by the way $a^2$ was defined, $(a^1,a^2)<b_j$.

\end{proof}

\begin{lemma}
$\mathcal{I}^0$ is $\varepsilon$-dense.
\end{lemma}
\begin{proof}
    Let $\langle f^i\mid i<\theta_1\rangle$ and $\langle g^i\mid i<\theta_2\rangle$ be sequences of $\mathcal{I}^0$ of length smaller than $\varepsilon$ such that for all $i<h<\theta_1$ and $j<l<\theta_2$, $f^i<f^h$, $g^l<g^j$, and $f^i<g^j$. 
    Let us start by constructing a sequence $\langle a^\alpha\mid \alpha<\varepsilon\rangle$ by induction. By the way $\mathcal{I}^0$ was constructed, we know that $\langle f^i(0)\mid i<\theta_1\rangle$ is a non-decreasing sequence of $\kappa\times\mathcal{Q}$ and $\langle g^i(0)\mid i<\theta_2\rangle$ is a non-increasing sequence of $\kappa\times\mathcal{Q}$, such that for all $i<\theta_1$ and $j<\theta_2$, $f^i(0)\leq g^j(0)$. By Lemma \ref{dense_basic_order}, there is $a^0\in \kappa\times \mathcal{Q}$, such that for all $i<\theta_1$ and $j<\theta_2$, $f^i(0)\leq a^0 \leq g^j(0)$.

    Let $\alpha<\varepsilon$ be such that for all $\beta<\alpha$, $a^\beta$ has been defined. Let $$X_\alpha^f=\{f^i(\alpha)\mid i<\theta_1\ \&\ \forall\beta<\alpha\ f^i(\beta)=a^\beta\}$$ and $$X_\alpha^g=\{g^i(\alpha)\mid i<\theta_2\ \&\ \forall\beta<\alpha\ g^i(\beta)=a^\beta\}.$$ By the induction hypothesis, if $X_\alpha^f,X_\alpha^g\neq\emptyset$, then for all $y\in X_\alpha^f$ and $z\in X_\alpha^g$, $y\leq z$. 
    
    Therefore, by Lemma \ref{dense_basic_order}, if $X_\alpha^f\neq\emptyset$ or $X_\alpha^g\neq\emptyset$, then there is $a\in \kappa\times \mathcal{Q}$, such that for all $y\in X_\alpha^f$ and $z\in X_\alpha^g$, $y\leq a \leq z$. So, if $X_\alpha^f\neq\emptyset$ or $X_\alpha^g\neq\emptyset$ we choose $a^\alpha\in \kappa\times \mathcal{Q}$, such that for all $y\in X_\alpha^f$ and $z\in X_\alpha^g$, $y\leq a^\alpha \leq z$. 
    
    Let us fix $\tau'\in \mathcal{Q}$. If $\alpha<\varepsilon$ is such that $X_\alpha^f=X_\alpha^g=\emptyset$, then we choose $a^\alpha=(0,\tau')$.

     Let $F:\varepsilon\rightarrow \kappa\times \mathcal{Q}$ be defined by $F(\alpha)=a^\alpha$.
     \begin{claim}
         $F\in \mathcal{I}^0$.
     \end{claim}
     \begin{proof}
         Since $\theta_1, \theta_2<\varepsilon$, and for all $f\in \{f^i\mid i<\theta_1\}\cup \{g^i\mid i<\theta_2\}$, $|\{\alpha\in\varepsilon\mid f_1(\alpha)\neq 0\}|$ is smaller than $\varepsilon$, $|\{\alpha\in\varepsilon\mid F_1(\alpha)\neq 0\}|$ is smaller than $\varepsilon$.
     \end{proof}
     \begin{claim}
         For all $i<\theta_1$, $f^i< F$.
     \end{claim}
     \begin{proof}
     By the way $F$ was constructed, we know that for all $i<\theta_1$ and $j<\theta_2$, $f^i\leq F \leq g^j$. Let us assume, towards contradiction, that there is $i<\theta_1$ such that $f^i\not < F$. Since $\mathcal{I}^0$ is a linear order, $f^i=F$ and for all $h\ge i$, $f^h=F$. So $\theta_1=i+1$. 
     
         Notice that for all $\alpha<\kappa$, $f^i(\alpha)\in X_\alpha^f$. By Lemma \ref{dense_basic_order}, for all $\alpha<\varepsilon$, $f^i(\alpha)=a^\alpha$ implies that there is $j_\alpha<\theta_2$ such that for all $\beta\leq \alpha$, $g^{j_\alpha}(\beta)=a^\beta$. Thus, there is a sequence $\{g^{j_\alpha}\}_{\alpha<\varepsilon}$ such that for all $\alpha<\kappa$ and $\beta\leq\alpha$, $g^{j_\alpha}(\beta)=a^\beta$.
         
         Since $\theta_2<\varepsilon$, there is $j<\theta_2$ such that for all $\alpha<\varepsilon$, $g^j(\alpha)=a^\alpha$. So $g^j=F=f^i$, this contradicts that $f^i<g^j$.
     \end{proof} 
     Using the same argument, we can show that for all $j<\theta_2$, $F<g^j$.
\end{proof}

\begin{lemma}
    $I$ is $\varepsilon$-dense.
\end{lemma}
\begin{proof}
    Let $\langle f^i\mid i<\theta_1\rangle$ and $\langle g^i\mid i<\theta_2\rangle$ sequences of $I$ of length smaller than $\varepsilon$ such that for all $i<h<\theta_1$ and $j<l<\theta_2$, $f^i<f^h$, $g^l<g^j$, and $f^i<g^j$. Let us construct the sequence $\langle a^\alpha\mid \alpha<\varepsilon\rangle$ by induction.
    Since $\mathcal{I}^0$ is $\varepsilon$-dense, there is $a^{-1}\in \mathcal{I}^0$ such that $a^0=(0,a^{-1})$, and for all $i<\theta_1$ and $j<\theta_2$, $f^i(0)\leq a^0\leq g^j(0)$. 
    
    Let $0<\alpha<\varepsilon$ be such that for all $\beta<\alpha$, $a^\beta$ has been defined.
    Define $$X^f_\alpha=\{f^i(\alpha)\mid i<\theta_1,\ \forall\beta<\alpha\ (f^i(\beta)=a^\beta)\ \&\ f^i_1(\alpha)\neq 0)\}$$ 
    and $$X^g_\alpha=\{g^i(\alpha)\mid i<\theta_2,\ \forall\beta<\alpha\ (g^i(\beta)=a^\beta)\ \&\ g^i_1(\alpha)\neq 0)\}.$$

    By the induction hypothesis, if $X_\alpha^f,X_\alpha^g\neq\emptyset$, then for all $x\in X_\alpha^f$ and $y\in X_\alpha^g$, $x\leq y$. Therefore by Lemma \ref{dense_basic_order}, for all $\alpha>0$ if $X_\alpha^f\neq\emptyset$ or $X_\alpha^g\neq\emptyset$, there is $a\in (\kappa\backslash\{0\})\times\mathcal{Q}$ such that for all $x\in X^f_\alpha$ and $y\in X^g_\alpha$, $x\leq a\leq y$.

    If $X_\alpha^g\neq\emptyset$ we choose $a^\alpha\in \kappa\times \mathcal{Q}$ as in Lemma \ref{dense_basic_order}, such that for all $y\in X_\alpha^f$ and $z\in X_\alpha^g$, $y\leq a^\alpha \leq z$. 
    
    For all $\beta<\kappa$, let us denote $a^\beta=(a^\beta_1,a^\beta_2)$.  If $X_\alpha^f\neq\emptyset$ and $X_\alpha^g=\emptyset$, we choose $a^\alpha=(a^\alpha_1,a_2^\alpha)$ by
\begin{itemize}
\item $a^\alpha_1=\bigcup(\{a^\beta_1+1\mid \beta<\alpha\}\cup \{f^i_1(\alpha)+1\mid f^i(\alpha)\in X^f_\alpha\})$;
\item $a^\alpha_2=\tau$.
\end{itemize}

    If $\alpha$ is such that $X_\alpha^f=X_\alpha^g=\emptyset$, then we choose $a^\alpha$ as follows:
    \begin{itemize}
        \item If for all $\beta<\alpha$, $X^f_\beta\neq\emptyset$ or $X^g_\beta\neq \emptyset$, then choose $a^\alpha=(\bar{a}^\alpha,\tau)$, where $\bar{a}^\alpha=\bigcup\{a^\beta_1+1\mid \beta<\alpha\}$;
        \item if there is $\beta<\alpha$ such that $a^\beta=(0,\tau)$, then choose $a^\alpha=(0,\tau)$.
    \end{itemize}
    Let us define $F:\varepsilon\rightarrow (\{0\}\times \mathcal{I}^0)\cup (\kappa\times \mathcal{Q})$ by $F(\alpha)=a^\alpha$.

\begin{claim}
    $F\in I$
\end{claim}
\begin{proof}
    Since $\theta_1,\theta_2<\varepsilon$, there is $\alpha<\kappa$ such that for all $\beta>\alpha$, $F(\beta)=(0,\tau)$. Notice that there is $\alpha$ such that $X^f_\alpha,X^g_\alpha=\emptyset$ and for all $\beta<\alpha$, $X^f_\beta\neq\emptyset$ or $X^g_\beta\neq \emptyset$. 
    
    We are missing to show that  $F_1\restriction \alpha+1$ is strictly increasing. Let us suppose, towards contradiction, there is $\beta<\alpha$ such that $F_1(\beta)\ge F_1(\beta+1)$. 
    
    {\bf Case $X^f_{\beta+1}\neq\emptyset$.} There is $i<\theta_1$ such that for all $\varsigma<\beta+1$, $f^i_1(\varsigma)=F_1(\varsigma)$ and $f^i_1(\beta+1)\in X^f_{\beta+1}$. So $f^i_1(\beta+1)\neq 0$. Since $f^i\in I$, $f^i_1(\beta+1)>f^i_1(\beta)=F_1(\beta)\ge F_1(\beta+1)$ a contradiction with the way we chose $a^{\beta+1}$.

    {\bf Case $X^g_{\beta+1}\neq\emptyset$} and {\bf $X^f_{\beta+1}=\emptyset$.} From the proof of Lemma \ref{dense_basic_order} there is $l<\theta_2$ such that $F_1(\beta+1)=g_1^l(\beta+1)\in X^g_{\beta+1}$. So, for all $\varsigma<\beta+1$, $g^l_1(\varsigma)=F_1(\varsigma)$. Since $g^l\in I$, $g^l_1(\beta+1)>g^l_1(\beta)=F_1(\beta)\ge F_1(\beta+1)=g^l_1(\beta+1)$, a contradiction.

    {\bf Case $X^g_{\beta+1},X^f_{\beta+1}=\emptyset$}, i.e. $\alpha=\beta+1${\bf .} By the way $a^\alpha$ was defined, $a_1^\alpha>a_1^\beta$. So $F_1(\beta)\ge F_1(\beta+1)=F_1(\alpha)>F_1(\beta)$ a contradiction.
    
\end{proof}
    
     \begin{claim}
         For all $i<\theta_1$, $f^i<F$.
     \end{claim}
    \begin{proof}
        Let us suppose, towards contradiction, that there is $i<\theta_1$ such that $f^i\ge F$.

        {\bf Case $f^i> F$.} By the way $F$ was constructed this could only happens when $\alpha=dp(f^i)$, $f^i\restriction \alpha+1=F\restriction \alpha+1$ and $F_1(\alpha+1)\neq 0$.
        Notice that since $\langle f^i\mid i<\theta_1\rangle$ is an increasing sequence, there is no $j<\theta_1$ such that $f^i<f^j$. Otherwise, since $\alpha=dp(f^i)$, there is $\beta\leq\alpha$ such that $f^i\restriction \beta=f^j\restriction \beta$ and $a^\beta=f^i(\beta)<f^j(\beta)$. A contradiction with the way $a^\beta$ was chosen. Thus $\theta_1=i+1$ and $X^f_{\alpha+1}=\emptyset$.

       On the other hand $X^g_{\alpha+1}=\emptyset$. Otherwise, there is $j<\theta_2$ such that $g^j\restriction \alpha+1=F\restriction \alpha+1=f^i\restriction \alpha+1$ and $f^i(\alpha+1)=(0,\tau)$; so $f^i\ge g^j$.

        Finally, by the way $a^{\alpha}$ was chosen, $f^i(\alpha)=F(\alpha)$ implies that there is $j<\theta_2$ such that $g^j\restriction \alpha+1=f^i\restriction \alpha+1$. Since $X^g_{\alpha+1}=\emptyset$, $dp(g^j)=\alpha=dp(f^i)$. We conclude that $f^i=g^j$ a contradiction.

        {\bf Case $f^i= F$.} Let $\alpha=dp(f^i)$. Notice that $f^i(\alpha)\in X^f_\alpha$, by the way $a^{\alpha+1}$ was defined, $F_1(\alpha+1)=a^{\alpha+1}_1>0$. But $F_1(\alpha+1)=f^i_1(\alpha+1)=0$, a contradiction.

    \end{proof}

    \begin{claim}
        For all $j<\theta_2$, $F<g^j$.
    \end{claim}

    \begin{proof}
        Let us suppose, towards contradiction, that there is $j<\theta_2$ such that $g^j\leq F$.

        {\bf Case $g^j< F$.} Notice that $dp(g^j)>0$, otherwise since $g^j(0)\ge F(0)$, $g^j(0)= F(0)$ and $g^j_1(1)=0\neq F(1)$; so $g^j> F$ a contradiction. 
        
        Therefore, there is $0<\alpha<\theta_2$ such that $g^j\restriction\alpha=F\restriction\alpha$, and $F(\alpha)=(0,\tau)$ and $g^j(\alpha)\neq (0,\tau)$. Therefore $dp(g^j)\ge\alpha$, $g^j_1(\alpha)>0$, and $g^j(\alpha)\in X^g_\alpha$. Thus $F_1(\alpha)>0$ a contradiction. 

        {\bf Case $g^j= F$.} Let $\alpha=dp(g^j)$. Then $g^j(\alpha)\in X^g_\alpha$, by the way $F$ was constructed, $F(\alpha+1)\neq (0,\tau)$. Since $F=g^j$, $g^j(\alpha+1)\neq (0,\tau)$ and $dp(g^j)>\alpha$ a contradiction.
            
    \end{proof}
\end{proof}

\subsection{$(\kappa, \varepsilon)$-nice}

To prove the $(\kappa, \varepsilon)$-nice property, we will show that for all $i<\kappa$, the order $I^i$ is $(\kappa, \varepsilon)$-nice. This will follows from the properties of the generators.

\begin{lemma}\label{stricti_minor}
For all limit $\delta<\kappa$ and $\nu\in I^0$ there is $\beta<\delta$ which satisfies the following:
$$\forall\sigma\in I^0_\delta [\sigma>\nu \Rightarrow \exists\sigma'\in I^0_\beta\ (\sigma\ge\sigma'\ge\nu)].$$
In particular. If $\nu\notin I^0_\delta$, then $\beta$ satisfies:
$$\forall\sigma\in I^0_\delta [\sigma>\nu \Rightarrow \exists\sigma'\in I^0_\beta\ (\sigma>\sigma'>\nu)].$$
\end{lemma}

\begin{proof}
It is enough to show that $\mathcal{I}^0$ satisfies the desire property.  
Suppose $\delta<\kappa$ is a limit and $\nu\in \mathcal{I}^0$. If $\nu\in \mathcal{I}^0_\delta$, then there is $\beta<\delta$ such that $\nu\in \mathcal{I}^0_\beta$ and the result follows.

Let us take care of the case $\nu\notin \mathcal{I}^0_\delta$. Let $\beta<\delta$ be the least ordinal such that for all $\alpha<\varepsilon$, $\nu_1(\alpha)<\delta$ implies $\nu_1(\alpha)<\beta$.

\begin{claim}
For all $\sigma\in \mathcal{I}^0_\delta$. If $\sigma> \nu$, then there is $\sigma'\in \mathcal{I}^0_\beta$ such that $\sigma\neq\sigma'$ and $\sigma>\sigma'>\nu$.
\end{claim}
\begin{proof}
Let us suppose $\sigma\in \mathcal{I}^0_\delta$ is such that $\sigma> \nu$. By the definition of $\mathcal{I}^0$, there is $\alpha<\varepsilon$ such that $\sigma(\alpha)>\nu(\alpha)$ and $\alpha$ is the minimum ordinal such that $\sigma(\alpha)\neq \nu(\alpha)$. Since $\sigma\in \mathcal{I}^0_\delta$, for all $\rho\leq \alpha$, $\nu_1(\rho)<\delta$. Thus for all $\rho\leq \alpha$, $\nu_1(\rho)<\beta$.

Let us divide the proof in two cases, $\sigma_1(\alpha)=\nu_1(\alpha)$ and $\sigma_1(\alpha)>\nu_1(\alpha)$.

{\bf Case 1.} $\sigma_1(\alpha)=\nu_1(\alpha)$.

By the density of $\mathcal{Q}$ there is $r$ such that $\sigma_2(\alpha)>r>\nu_2(\alpha)$. Let us define $\sigma'$ by:
$$\sigma'(\rho)=\begin{cases} \nu(\rho) &\mbox{if } \rho<\alpha\\
(\nu_1(\alpha), r) & \mbox{if } \rho=\alpha\\
0 & \mbox{otherwise. } \end{cases}$$  
Clearly $\sigma>\sigma'>\nu$. Since $\nu_1(\rho)<\beta$ for all $\rho\leq \alpha$, $\sigma'\in \mathcal{I}^0_\beta$.

{\bf Case 2.} $\sigma_1(\alpha)>\nu_1(\alpha)$.

Since $\mathcal{Q}$ is a model of DLO, there is $r$ such that $r>\nu_2(\alpha)$. Let us define $\sigma'$ by:
$$\sigma'(\rho)=\begin{cases} \nu(\rho) &\mbox{if } \rho<\alpha\\
(\nu_1(\alpha), r) & \mbox{if } \rho=\alpha\\
0 & \mbox{otherwise. } \end{cases}$$  
Clearly $\sigma>\sigma'>\nu$. Since $\nu_1(\rho)<\beta$ for all $\rho\leq \alpha$, $\sigma'\in \mathcal{I}^0_\beta$.
\end{proof}
\end{proof}

\begin{lemma}\label{nice_for_i}
For all $i<\kappa$, $\delta<\kappa$ a limit ordinal, and $\nu\in I^i$, there is $\beta<\delta$ that satisfies the following:
\begin{equation}
\forall\sigma\in I^i_\delta\ [\sigma>\nu \Rightarrow \exists\sigma'\in I^i_\beta\ (\sigma\ge\sigma'\ge\nu)].
\end{equation}
In particular. If $\nu\notin I^i_\delta$, then $\beta$ satisfies:
$$\forall\sigma\in I^i_\delta [\sigma>\nu \Rightarrow \exists\sigma'\in I^0_\beta\ (\sigma>\sigma'>\nu)]$$
\end{lemma}

\begin{proof}
Notice that if $\nu\in I^{i}_\delta$, then there is $\alpha<\delta$ such that $\nu\in I^{i}_\alpha$ and the result follows for $\beta=\alpha$. We only have to prove the lemma when $\nu\in I^i\backslash  I^{i}_\delta$:

\textit{For all $i<\kappa$, $\delta<\kappa$ a limit ordinal, and $\nu\in I^i\backslash I^{i}_\delta$, there is $\beta<\delta$ that satisfies the following:}
\begin{equation}\label{property_HT2}
\forall\sigma\in I^i_\delta\ [\sigma>\nu \Rightarrow \exists\sigma'\in I^0_\beta\ (\sigma>\sigma'>\nu)]
\end{equation}

Let $\nu\in I^i$ be such that $\nu\in I^i\backslash I^{i}_\delta$. Let $\nu'\in I^0$ be such that $\nu\in Gen(\nu')$, so by Fact \ref{inside_the_gen}, $\nu'\notin I^i_\delta$. By Lemma \ref{stricti_minor}, there is $\beta<\delta$ such that $$\forall\sigma\in I^0_\delta\ [\sigma>\nu' \Rightarrow \exists\sigma'\in I^0_\beta\ (\sigma\ge\sigma'\ge\nu')].$$

\begin{claim}
$\beta$ is as wanted.
\end{claim}
\begin{proof}
Let $\sigma\in I^i_\delta$ such that $\sigma>\nu$. Since $\nu'\notin I^i_\delta$, by Corollary \ref{coro1}, $\sigma>\nu'$. Since $\sigma\in I^i_\delta$, there is $\sigma'\in I^0_\delta$ such that $\sigma\in Gen(\sigma')$. By the way $\beta$ was chosen, there is $\sigma''\in I^0_\beta$ such that $\sigma'>\sigma''>\nu'$. It is clear that $\sigma''\notin Gen(\nu')\cup Gen(\sigma')$. We conclude from Corollary \ref{coro1} that $\sigma>\sigma''>\nu$. Since $\sigma''\in I^0_\beta\subseteq I^i_\beta$, $\sigma''$ is as wanted. 
\end{proof}
\end{proof}
As it can be seen in the proof of the previous lemma, the witness $\sigma''$ can be chosen in $I^0_\beta$ when $\nu\notin I^i_\delta$.

\begin{lemma}\label{nice_big_I}
For all $\delta<\kappa$ limit with $cf(\delta)\ge \varepsilon$, and $\nu\in I$, there is $\beta<\delta$ that satisfies the following:
\begin{equation}\label{property_HT3}
\forall\sigma\in I_\delta\ [\sigma>\nu \Rightarrow \exists\sigma'\in I_\beta\ (\sigma\ge\sigma'\ge\nu)]
\end{equation}
\end{lemma}

\begin{proof}
Let $\delta<\kappa$ be a limit ordinal with $cf(\delta)\ge \varepsilon$, and $\nu\in I$.
We have three different cases: $\nu\in I_\delta$, $\nu\in I^{o(\nu)}_\delta \backslash I_\delta$, and $\nu\notin I^{o(\nu)}_\delta$.

{\bf Case $\nu\in I_\delta$.} Since $\delta$ is a limit, $o(\nu)+1<\delta$ and there is $\rho<\delta$ such that $\nu\in I^{o(\nu)+1}_\rho$. Let $\beta=max\{o(\nu)+1,\rho\}$, it is clear that $\beta$ is as wanted.

{\bf Case $\nu\in I^{o(\nu)}_\delta \backslash I_\delta$.} 
Notice that $o(\nu)>\delta$. Let $\nu_0\in I^0$ be such that $\nu\in Gen(\nu_0)$. Let $\{\nu_i\}_{i\leq\alpha}$ be the road from $I^0$ to $\nu$, so $\nu=\nu_\alpha$ and $\alpha$ is the depth of $\nu$. The sequence $\langle o(\nu_i)\mid i\leq\alpha\rangle$ is an strictly increasing sequence with $o(\nu_0)=0$ and $o(\nu_\alpha)>\delta$. Therefore, there is $j\leq\alpha$ such that $\forall i<j$, $o(\nu_i)<\delta$ and $o(\nu_j)\ge \delta$. 

\begin{claim}
There is $\rho<\delta$ such that for all $i<j$, $o(\nu_i)+1<\rho$.
\end{claim}
\begin{proof}
Let us suppose, towards contradiction, that such $\rho$ doesn't exists. Since $\forall i<j$, $o(\nu_i)<\delta$ and $o(\nu_j)\ge \delta$, $\langle o(\nu_i)\mid i<j \rangle$ is cofinal to $\delta$. On the other hand $j<\varepsilon$, we conclude that $cf(\delta)<\varepsilon$ a contradiction.
\end{proof}

Recall that $\nu\in I^{o(\nu)}_\delta$, by Fact \ref{inside_the_gen}, for all $i<\alpha$, $\nu_i\in I^{o(\nu_i)+1}_\delta$. Let $\theta_1$ be $\rho$ in the previous claim. Notice that for all $i\leq\alpha$, $\nu_i\notin I^{\theta_1}_\delta$ implies $\nu_i\notin I^\delta_\delta$. Since $\delta$ is a limit ordinal, there is $\theta_2$ such that $\nu_0\in I^0_{\theta_2}$. Thus, for all $i\leq\alpha$, $\nu_i\in I^{o(\nu_i)+1}_{\theta_2}$. Let $\beta=max\{\theta_1,\theta_2\}$.

\begin{claim}
$\beta$ is as wanted.
\end{claim}
\begin{proof}
Let $\sigma\in I^\delta_\delta$ be such that $\sigma>\nu$. If $\sigma\in Gen(\nu_0)$, then $\sigma\in I^\beta_\beta$ and we are done.
Otherwise,  $\sigma\notin Gen(\nu_0)$. By Corollary \ref{coro1}, $\sigma>\nu_0>\nu$. So $\nu_0=\sigma'$ is as we wanted.
\end{proof}

{\bf Case $\nu\notin I^{o(\nu)}_\delta$.}  Let $\rho=max\{o(\nu)+1, \delta\}$, thus $\nu\in I^\rho$ and by Lemma \ref{nice_for_i} there is $\beta<\delta$ which satisfies the following:
$$\forall\sigma\in I^\rho_\delta [\sigma>\nu \Rightarrow \exists\sigma'\in I^0_\beta\ (\sigma>\sigma'>\nu)].$$

\begin{claim}
$\beta$ is as wanted.
\end{claim}
\begin{proof}
Let $\sigma\in I^\delta_\delta$ be such that $\sigma>\nu$. Since $\delta\leq\rho$, $\sigma\in I^\rho_\delta$. Therefore, there is $\sigma'\in I^0_\beta$ such that $\sigma>\sigma'>\nu$. The claim follows from $I^0_\beta\subseteq I^\beta_\beta=I_\beta$.
\end{proof}
\end{proof}

\begin{cor}\label{I_nice}
$I$ is $(\kappa, \varepsilon)$-nice.
\end{cor}

\subsection{$(<\kappa)$-stable}

Notice that if $\kappa$ is inaccessible, $I$ is $(<\kappa)$-stable. This can be generalize to $\kappa$ successors. Recall that $|\mathcal{Q}|=\theta<\kappa$ and $\theta\ge\varepsilon$.

\begin{lemma}\label{stable_I0_first_lemma}
Suppose $\kappa=\lambda^+$ and $2^\theta\leq\lambda=\lambda^{<\varepsilon}$.
$\mathcal{I}^0$ is $(<\kappa)$-stable.
\end{lemma}

\begin{proof}
Let $\mathcal{I}^{-1}$ be the set of functions $f:\varepsilon\rightarrow \kappa$ such that $|\{\alpha\in \varepsilon\mid f(\alpha)\neq 0\}|<\varepsilon$. We say that $f<g$ if $f(\alpha)<g(\alpha)$, where $\alpha$ is the least ordinal such that $f(\alpha)\neq g(\alpha)$.
\begin{claim}
$\mathcal{I}^{-1}$ is $(<\kappa)$-stable.
\end{claim}
\begin{proof}
Let $A\subseteq \mathcal{I}^{-1}$ be such that $|A|<\kappa$, and let $$\beta=sup\{f(\alpha)+1\mid f\in A\ \&\ \alpha<\varepsilon\}$$ so $\beta<\kappa$. Therefore, for all $f\in \mathcal{I}^{-1}$, $tp_{bs}(f,A,\mathcal{I}^{-1})$ is entirely determined by the coordinates of $f$ which are smaller than $\beta+1$. Since $\lambda=\lambda^{<\varepsilon}$ and $\beta<\kappa$, $|\{tp_{bs}(f,A,\mathcal{I}^{-1})\mid f\in \mathcal{I}^{-1}\}|< \kappa$.
\end{proof}

For all $A\subseteq \mathcal{I}^0$ define $Pr(A)$ as the set $\{f_1\mid f\in A\}$. Let $A\subseteq I^0$ be such that $|A|<\kappa$. 
Since $|\mathcal{Q}|=\theta$, $$|\{tp_{bs}(a,A,\mathcal{I}^0)\mid a\in \mathcal{I}^0\}|\leq |\{tp_{bs}(a,Pr(A),\mathcal{I}^{-1})\mid a\in \mathcal{I}^{-1}\}\times 2^\theta|.$$
By the previous claim and since $2^\theta\leq\lambda$, $|\{tp_{bs}(a,A,\mathcal{I}^0)\mid a\in \mathcal{I}^0\}|< \kappa$.
\end{proof}

This Lemma implies that under the assumptions
 $\kappa=\lambda^+$ and $2^\theta\leq\lambda=\lambda^{<\varepsilon}$,
$I^0$ is $(<\kappa)$-stable.

\begin{lemma}\label{stable_gen}
Suppose $\kappa=\lambda^+$ and $2^\theta\leq\lambda=\lambda^{<\varepsilon}$. For all $\nu\in I^0$, $Gen(\nu)$ with the induced order is  $(<\kappa)$-stable.
\end{lemma}

\begin{proof}
Let $\nu\in I^0$ and $A\subseteq Gen(\nu)$ be such that $|A|<\kappa$, and let $$\beta=sup\{f_1(\alpha)+1\mid f\in A\ \&\ \alpha<\varepsilon\}.$$ Since $f_1(\alpha)=0$ for all $\alpha> dp(f)$, $\beta<\kappa$. On the other hand, for all $f,g\in Gen(\nu)$, $f$ and $g$ eventually become constants to $(0,\tau)$, and the order $f<g$ (or $g<f$) is determined by the values of $f(\alpha)$ and $g(\alpha)$, where $\alpha$ is the least ordinal such that$f(\alpha)\neq g(\alpha)$.
Therefore, for all $f\in Gen(\nu)$, $tp_{bs}(f,A,Gen(\nu))$ is entirely determined by the coordinates $\alpha$ of $f$ in which $f_1(\alpha)$ is smaller than $\beta+1$, and the types of $\mathcal{Q}$. Since $\lambda^{<\varepsilon}=\lambda$, $\beta<\kappa$, and $2^\theta\leq \lambda$ $$|\{tp_{bs}(f,A,Gen(\nu))\mid f\in Gen(\nu)\}|\leq |\beta^{<\varepsilon}\times 2^\theta|\leq \lambda< \kappa.$$
\end{proof}

\begin{lemma}\label{I_stable}
Suppose $\kappa=\lambda^+$ and $2^\theta\leq\lambda=\lambda^{<\varepsilon}$.
$I$ is $(<\kappa)$-stable.
\end{lemma}

\begin{proof}
Let us fix $A\subset I$ such that $|A|<\kappa$. From Corollary \ref{coro1}, for all $a\in I$ and $\nu\in I^0$ such that $a\in Gen(\nu)$ the following holds:
$$b\models tp_{bs}(a,A,I) \Leftrightarrow b\models tp_{bs}(\nu,A \backslash Gen(\nu),I)\cup tp_{bs}(a,A\cap Gen(\nu), Gen(\nu)).$$
Thus for all $a\in I$ and $\nu\in I^0$ with $a\in Gen(\nu)$, the type of $a$ is determined by $tp_{bs}(\nu,A \backslash Gen(\nu),I)$ and $tp_{bs}(a,A \cap Gen(\nu), Gen(\nu))$.
Let $A'\subseteq I^0$ be such that the following hold:
\begin{itemize}
\item for all $x\in A$ there is $y\in A'$, $x\in Gen(y)$;
\item for all $y\in A'$ there is $x\in A$, $x\in Gen(y)$.
\end{itemize}

Clearly $|A'|\leq |A|$, and by Corollary \ref{coro1}, for all $\nu\in I^0$, $tp_{bs}(\nu,A\backslash Gen(\nu),I)$ is determined by $tp_{bs}(\nu,A'\backslash \{\nu\},I^0)$. So for all $a\in I$ and $\nu\in I^0$ with $a\in Gen(\nu)$, $tp_{bs}(a,A,I)$ is determined by $tp_{bs}(\nu,A'\backslash \{\nu\},I^0)$ and $tp_{bs}(a,A\cap Gen(\nu), Gen(\nu))$. Therefore $|\{tp_{bs}(a,A,I)\mid a\in I\}|$ is bounded by 
$$|\{tp_{bs}(\nu,A',I^0)\mid \nu\in I^0\}|\ \times\  Sup(\{B_\nu \mid \nu \in I^0\})$$ where 
$$B_\nu=|\{tp_{bs}(a,A \cap Gen(\nu), Gen(\nu))\mid a\in Gen(\nu)\}|.$$

From Lemma \ref{stable_gen}, we conclude that for all $\nu\in I^0$, $B_\nu<\kappa$. Since $\kappa=\lambda^+$, $Sup(\{B_\nu \mid \nu \in I^0\})\leq  \lambda$. From Lemma \ref{stable_I0_first_lemma} we know that $|\{tp_{bs}(\nu,A',I^0)\mid \nu\in I^0\}|<\kappa$, so $|\{tp_{bs}(\nu,A',I^0)\mid \nu\in I^0\}|\leq\lambda$. We conclude $|\{tp_{bs}(a,A,I)\mid a\in I\}|<\kappa$.
\end{proof}

\subsection{$\kappa$-colorable}

To finish the proof of Theorem \ref{main_order_prop} we will show that $I$ is $\kappa$-colorable. The depth of the elements will have a crucial role to define the $\kappa$-color function of $I$.

We say that $\Gamma=\langle \nu_j\mid j< \alpha\rangle$ is a pre-road if there is $\nu\in I$ such that $\langle \nu_j\mid j\leq \alpha\rangle$ is the road from $I^0$ to $\nu=\nu_\alpha$.
Notice that all roads are pre-roads.
For all pre-road $\Gamma=\langle \nu_j\mid j< \alpha\rangle$, let us define the successors of $\Gamma$, $Succ_I(\Gamma)$, and the complexities of $\Gamma$, $Comp(\Gamma)$, as follows:
$$Succ_I(\Gamma)=\{\sigma\in I\mid \langle \nu_j\mid j\leq \alpha\rangle \text{ is the road from } I^0\text{ to }\sigma=\nu_\alpha\}$$ and  $$Comp(\Gamma)=\{o(\nu)\mid \nu\in Succ_I(\Gamma) \}.$$

\begin{lemma}\label{A_desire_order}
$I$ is a $\kappa$-colorable linear order.
\end{lemma}

\begin{proof}
Let us fix a bijection $G:\kappa\rightarrow\kappa\times\kappa$, and $G_1$, $G_2$ be the functions such that $G(\alpha)=(G_1(\alpha),G_2(\alpha))$.  For all pre-road $\Gamma$ fix a bijection $g_\Gamma:Comp(\Gamma)\rightarrow \kappa$.
Let us define $F:I\rightarrow\kappa$ by 
$$F(\nu)=\begin{cases} 0 &\mbox{if } dp(\nu)=0\\
G_1(g_{\Gamma}(o(\nu))) & \mbox{where }\nu\in Succ_I(\Gamma). \end{cases}$$

\begin{claim}
$F$ is a $\kappa$-color function of $I$.
\end{claim}
\begin{proof}
Let $B\subseteq I$, $|B|<\kappa$, $b\in I\backslash B$, and $p= tp_{bs}(b,B,I)$. Since $|B|<\kappa$, there is $\alpha<\kappa$ such that $B\subset I^{\alpha}$. Let $\Gamma$ be the road from $I^0$ to $b$, $\beta=\max\{o(b), \alpha\}$, by Fact \ref{road_type}, for all $\nu\in \{a\in Succ_I(\Gamma)\mid o(a)>\beta\}$, $b$ and $\nu$ have the same type of basic formulas over $I^\beta\backslash \{b\}$. 
In particular for all $\nu\in \{a\in Succ_I(\Gamma)\mid o(a)>\beta\}$, $\nu\models p$. By the way $F$ was define, we conclude that for any $\rho<\kappa$, $|\{a\in Succ_I(\Gamma)\mid o(a)>\beta\ \&\ F(a)=\rho\}|=\kappa$. Which implies that for any $\rho<\kappa$, $|\{a\in Succ_I(\Gamma)\mid a\models p\ \&\ F(a)=\rho\}|=\kappa$.
\end{proof}
\end{proof}

The previous lemmas prove Theorem \ref{main_order_prop}, $I$ is $\varepsilon$-dense, $(<\kappa)$-stable, $(\kappa, \varepsilon)$-nice, and $\kappa$-colorable.

Notice that in the previous lemma, if $f,g\in I$ are such that $dp(f)=dp(g)=\alpha$, $o(f)=o(g)$,  and $f\restriction \alpha=g\restriction\alpha$, then $F(f)=F(g)$. 
\begin{remark}\label{small_order}
It is easy to see that $I\backslash I^0$ is a $\varepsilon$-dense $(<\kappa)$-stable $(\kappa, \varepsilon)$-nice $\kappa$-colorable linear order. For all $a\in I\backslash I^0$ and sequence $\langle a_i\rangle_{i<\alpha}$ coinicial to $a$, there is $\beta<\alpha$ such that for all $i>\beta$, $o(a)<o(a_i)$. Holes have a similar property. A hole is a pair $(\bar{b}, \bar{a})$, $\bar{b}=\langle b_i\rangle_{i<\theta_1}$ and $\bar{a}=\langle a_i\rangle_{i<\theta_2}$, such that $\theta_1,\theta_2<\kappa$, for all, $i<\theta_1$ and $j<\theta_2$, $b_i<a_j$, and there is no $a\in I\backslash I^0$ such that for all $i<\theta_1$ and $j<\theta_2$, $b_i<a<a_j$.
For any hole $(\bar{b}, \bar{a})$, there is a sequence $\bar{c}=\langle c_i\rangle_{i<\theta_3}$ such that 
\begin{itemize}
    \item $\theta_3\leq\theta_1$,
    \item for all $i<\theta_1$ there is $j<\theta_3$ such that $b_i<c_j$,
    \item for all $j<\theta_3$ there is $i<\theta_1$ such that $c_j<b_i$,
    \item there is $\beta<\theta_2$ such that for all $j<\theta_3$ and $i>\beta$, $o(c_j)<o(a_i)$.
\end{itemize}
We will use the order $I\backslash I^0$ in the following section.
\end{remark}

As it was mentioned before, $\varepsilon$-dense and $(\kappa, \mu)$-nice are two notions of density. We constructed linear orders when $\varepsilon=\mu$, this is enough for our purposes.

\section{Trees}\label{Sectio_trees}

\subsection{Coloured trees}

Coloured trees were introduced by Hyttinen and Weinstein  in \cite{HK} to construct models of stable unsuperstable theories. Variations of coloured trees have been used to construct models of other non-classifiable theories, by Hyttinen-Moreno in \cite{HM}, and by Moreno in \cite{Mor} and \cite{Mor21}.

Let $t$ be a set-theoretic tree, for every $x\in t$ we denote by $ht(x)$ the height of $x$, the order type of $\{y\in t | y\prec x\}$. Define $(t)_\alpha=\{x\in t|ht(x)=\alpha\}$ and $(t)_{<\alpha}=\cup_{\beta<\alpha}(t)_\beta$, denote by $x\restriction \alpha$ the unique $y\in t$ such that $y\in (t)_\alpha$ and $y\prec x$. If $x,y\in t$ and $\{z\in t|z\prec x\}=\{z\in t|z\prec y\}$, then we say that $x$ and $y$ are $\sim$-related, $x\sim y$, and we denote by $[x]$ the equivalence class of $x$ for $\sim$.

An \textit{$\alpha, \beta$-tree} is a tree $t$ with the following properties:
\begin{itemize}
\item $|[x]|<\alpha$ for every $x\in t$.
\item All the branches have order type less than $\beta$ in $t$.
\item $t$ has a unique root.
\item If $x,y\in t$, $x$ and $y$ have no immediate predecessors and $x\sim y$, then $x=y$.
\end{itemize}
\begin{defn}\label{D.2.1}
Let $\gamma$ be a regular cardinal smaller than $\kappa$, and $\beta$ a cardinal smaller or equal to $\kappa$. A $\beta$-\textit{coloured tree} is a pair $(t,c)$, where $t$ is a $\kappa^+$, $(\gamma+2)$-tree and $c$ is a map $c:(t)_\gamma\rightarrow \beta$. 
\end{defn}

Two $\beta$-coloured trees $(t,c)$ and $(t',c')$ are isomorphic, if there is a trees isomorphism $f:t\rightarrow t'$ such that for every $x\in (t)_\gamma$, $c(x)=c'(f(x))$.

From now on in this section $\gamma$ is a regular cardinal smaller than $\kappa$, and $\beta$ a cardinal smaller or equal to $\kappa$ (as in Definition \ref{D.2.1}).

\textbf{Intuition.} The idea behind the use of $\beta$-coloured trees is to use the color functions $c$ to code the equivalence classes of $=^\beta_\gamma$. It is advice to have the following intuition in mind when reading the rest of this section. Intuitively, the leaves of the constructed trees will have increasing functions with 5 parameters:
\begin{itemize}
    \item the first parameter will control whether the function is increases at each $\alpha<\gamma$;
    \item the second and third parameters control the range of the fifth parameter at easch $\alpha<\gamma$, being the second parameter the lower bound and the third parameter the upper bound;
    \item the fourth parameter ensures that the tree has many copies of each element, so that partial isomorphisms of trees could be extended;
    \item the fifth parameter is the one coding the values of $f$ by the color function at the leaves:
    \begin{itemize}
        \item at a leaf $l$, when this parameter is constant at an interval $[\alpha,\gamma)$, the color function at that leaf is decided by the fourth parameter to leave enough room to extend partial isomorphisms;
        \item otherwise, the supremum of this parameter, $\alpha$, determines the color function at that leaf, e.g. $c(l)=f(\alpha)$. 
    \end{itemize}
\end{itemize}

Let us start with the construction. Let us order the set $\gamma\times \kappa\times \kappa\times \kappa\times \kappa$ lexicographically as follows:

Given two tuples, $(\alpha_1,\alpha_2,\alpha_3,\alpha_4,\alpha_5)$ and $(\alpha'_1,\alpha'_2,\alpha'_3,\alpha'_4,\alpha'_5)$, we say that $$(\alpha_1,\alpha_2,\alpha_3,\alpha_4,\alpha_5)>(\alpha'_1,\alpha'_2,\alpha'_3,\alpha'_4,\alpha'_5)$$ if and only if there is $1\leq k \leq 5$ such that for all $i<k$, $\alpha_i=\alpha'_i$, and $\alpha_k>\alpha'_k$. Order the set $(\gamma\times \kappa\times \kappa\times \kappa\times \kappa)^{\leq \gamma}$ as a tree by initial segments.

For all $f\in \beta^\kappa$, define the tree $(R_f,r_f)$ as, $R_f$ the set of all strictly increasing functions from some $\alpha\leq \gamma$ to $\kappa$ and for each $\eta$ with domain $\gamma$, $r_f(\eta)=f(sup(rng(\eta)))$.

For every pair of ordinals $\alpha$ and $\varrho$, $\alpha<\varrho<\kappa$ and $i<\gamma$ define $$R(\alpha,\varrho,i)=\bigcup_{i< j\leq \gamma}\{\eta:[i,j)\rightarrow[\alpha,\varrho)\mid\eta \text{ strictly increasing}\}.$$

\begin{defn}\label{constants_trees}
If $\alpha<\varrho<\kappa$ and $\alpha,\varrho,\rho\neq 0$, let $\{Z^{\alpha,\varrho}_\rho|\rho<\kappa\}$ be an enumeration of all downward closed subtrees of $R(\alpha,\varrho,i)$ for all $i$, in such a way that each possible $\beta$-coloured tree appears cofinally often in the enumeration. Let $Z^{0,0}_0$ be the tree $(R_f,r_f)$.
\end{defn}

This enumeration is possible because there are at most $$|\bigcup_{i<\gamma}\mathcal{P}(R(\alpha,\varrho,i))|\leq \gamma\times\kappa=\kappa$$ downward closed coloured subtrees. Since for all $\alpha<\varrho<\kappa$, $|R(\alpha,\varrho,i)|<\kappa$ there are at most $\kappa\times \kappa^{<\kappa}=\kappa$ $\beta$-coloured trees.
Denote by $Q(Z^{\alpha,\varrho}_\rho)$ the unique ordinal $i$ such that $Z^{\alpha,\varrho}_\rho\subset R(\alpha,\varrho,i)$.

\begin{defn}\label{colorconst}
Define for each $f\in \beta^\kappa$ the $\beta$-coloured tree $(J_f,c_f)$ by the following construction.

For every $f\in \beta^\kappa$ we define $J_f=(J_f,c_f)$ recursively as the tree of all $\eta: s\rightarrow \gamma\times \kappa^4$, where $s\leq \gamma$, ordered by end extension, and such that the following conditions hold for all $i,j<s$:

Denote by $\eta_k$, $1\leq k\leq 5$, the functions from $s$ to $\kappa$ that satisfies, $$\eta(n)=(\eta_1(n),\eta_2(n),\eta_3(n),\eta_4(n),\eta_5(n))\text{ for all }n<s.$$
\begin{enumerate}
\item $\eta\restriction n\in J_f$ for all $n<s$.
\item $\eta$ is strictly increasing with respect to the lexicographical order on $\gamma\times \kappa^4$.
\item $\eta_1(i)\leq \eta_1(i+1)\leq \eta_1(i)+1$, when $i+1<s$.
\item $\eta_1(i)=0$ implies $\eta_2(i)=\eta_3(i)=\eta_4(i)=0$.
\item $\eta_2(i)\ge\eta_3(i)$ implies $\eta_2(i)=0$.
\item $\eta_1(i)<\eta_1(i+1)$ implies $\eta_2(i+1)\ge \eta_3(i)+\eta_4(i)$, when $i+1<s$.
\item For every limit ordinal $\alpha$, $\eta_k(\alpha)=sup_{\iota<\alpha}\{\eta_k(\iota)\}$ for $k\in \{1,2\}$.
\item $\eta_1(i)=\eta_1 (j)$ implies $\eta_k (i)=\eta_k (j)$ for $k\in \{2,3,4\}$.
\item If for some $\iota<\gamma$, $[i,j)=\eta_1^{-1}\{\iota\}$, then $$\eta_5\restriction {[i,j)}\in Z^{\eta_2(i),\eta_3(i)}_{\eta_4(i)}.$$
\noindent Note that 9 implies $Z^{\eta_2(i),\eta_3(i)}_{\eta_4(i)}\subset R(\alpha,\varrho,i)$ 
\item If $s=\gamma$, then either 
\begin{itemize}
\item [(a)] there exists an ordinal number $m$ such that for every $\iota<m$, $\eta_1(\iota)<\eta_1(m)$, for every $\iota' \ge m$, $\eta_1(\iota)=\eta_1(m)$, and the color of $\eta$ is determined by $Z^{\eta_2(m),\eta_3(m)}_{\eta_4(m)}$: $$c_f(\eta)=c(\eta_5\restriction {[m,\gamma)})$$ where $c$ is the colouring function of $Z^{\eta_2(m),\eta_3(m)}_{\eta_4(m)}$;

\end{itemize}

or
\begin{itemize}
\item [(b)] there is no such ordinal $m$ and then $c_f(\eta)=f(sup(rang(\eta_5)))$.
\end{itemize}
\end{enumerate}
\end{defn}

At first sight, these trees look very general. But actually, these trees are coding the $=^\beta_\gamma$-class of any $f\in \beta^\kappa$ . A simpler version of the following lemma, was stated in \cite{Mor} without proof due to the length of the article. The version $\gamma=\omega$ of the following lemma, can be found in \cite[Theorem 2.5]{HK} and \cite[Lemma 4.7]{HM}. 

\begin{lemma}\label{colorisom}
Suppose $\gamma$ is a regular cardinal smaller than $\kappa$, such that for all $\epsilon<\kappa$, $\epsilon^\gamma<\kappa$. Let $\beta$ be a cardinal, $2\leq\beta\leq\kappa$. For every $f,g\in \beta^\kappa$ the following holds $$f\ =^\beta_\gamma\ g \Leftrightarrow  J_f\cong_{ct} J_g$$
where $\cong_{ct}$ is the isomorphism of $\beta$-coloured trees.
\end{lemma}

 We will prove this lemma in the following two subsections.

\subsection{Filtrations}
To prove this lemma, we have to introduce the notion of filtration. Filtrations will be important in the last subsection, to understand the $\kappa$-representation of ordered $\beta$-coloured trees.
Given a $\beta$-coloured tree $(t,c)$, we say that a sequence $( I_\alpha)_{ \alpha<\kappa}$ is a \textit{filtration} of $t$ if the following hold:
\begin{itemize}
    \item it is an increasing sequence of downwards closed subsets of $t$;
    \item $\bigcup_{\alpha<\kappa}I_\alpha=t$;
    \item if $\rho<\kappa$ is a limit ordinal, then $I_\rho=\bigcup_{\alpha<\rho}I_\alpha$;
    \item for all $\alpha<\kappa$, $|I_\alpha|<\kappa$.
\end{itemize}

Notice that a filtration of $(t,c)$ is a $\kappa$-representation of $t$ made out of downwards closed subsets of $t$. 

\begin{defn}
    Let $(t,c)$ be a $\beta$-coloured tree and $\mathcal{I}=( I_\alpha)_{ \alpha<\kappa}$ a filtration of $t$. Let us define $H_{\mathcal{I},t}\in \kappa^\kappa$ as follows. 

    For every $\alpha<\kappa$ define $B_\alpha$ as the set of all $x\in t_\gamma$ that are not elements of $I_\alpha$, and for all $\varrho<\gamma$, $x\restriction \varrho\in I_\alpha$.
    \begin{itemize}
        \item If $B_\alpha$ is not empty and there is $\iota$ such that, for all $x\in B_\alpha$, $c(x)=\iota$, then let $H_{\mathcal{I},t}(\alpha)=\iota$;
        \item let $H_{\mathcal{I},t}(\alpha)=0$ otherwise.
    \end{itemize}
\end{defn}

Notice that for any two filtrations  $( I_\alpha)_{ \alpha<\kappa}$ and  $( J_\alpha)_{ \alpha<\kappa}$ of the same $\beta$-coloured tree, there is a club $C$ such that for all $\alpha\in C$, $I_\alpha=J_\alpha$. So, for any two filtrations $\mathcal{I}=( I_\alpha)_{ \alpha<\kappa}$ and  $\mathcal{J}=( J_\alpha)_{ \alpha<\kappa}$ of $(t,c)$, $H_{\mathcal{I},t}\ =^\beta_\gamma\  H_{\mathcal{J},t}$.
We say that a filtration is a \textit{good filtration} if for all $\alpha$, $B_\alpha\neq\emptyset$ implies that $c$ is constant on $B_\alpha$.

\begin{fact}[{Hyttinen-Kulikov, \cite[Lemma 2.4]{HK}}]\label{4.4}
Suppose $(t_0,c_0)$ and $(t_1,c_1)$ are isomorphic $\beta$-coloured trees, and $\mathcal{I}=(I_\alpha)_{\alpha<\kappa}$ and $\mathcal{J}=(J_\alpha)_{\alpha<\kappa}$ are good filtrations of $(t_0,c_0)$ and $(t_1,c_1)$ respectively. Then $H_{\mathcal{I},t_0}\ =^\beta_\gamma\  H_{\mathcal{J},t_1}$.
\end{fact}

For each $f\in \beta^\kappa$ let us proceed to define a good filtration $(J^\alpha_f)_{\alpha<\kappa}$ of $J_f$. For each $\alpha<\kappa$ define $J_f^\alpha$ as $$J_f^\alpha=\{\eta\in J_f| rang(\eta)\subset \gamma\times(\iota)^4\text{ for some }\iota<\alpha\}.$$

Notice that for any $k\in rang(\eta_1)$,  $[i,j)=\eta_1^{-1}(k)$ and if $i+1<j$, then $\eta_5\restriction {[i,j)}$ is strictly increasing. If $\eta_1(i)<\eta_1(i+1)$, by Definition \ref{colorconst} item 6, $\eta_2(i+1)\ge \eta_3(i)+\eta_4(i)$, so $\eta_5(i)<\eta_3(i)\leq \eta_2(i+1)\leq \eta_5(i+1)$. If $\alpha$ is a limit ordinal, by Definition \ref{colorconst} items 7 and 8, $\eta_5(\iota)<\eta_2(\iota+1)<\eta_2(\alpha)\leq \eta_5(\alpha)$ holds for every $\iota<\alpha$. Thus $\eta_5$ is strictly increasing. If $\eta\restriction n\in J_f$ for all $n<\dom (\eta)$, then $\eta\in J_f$. Clearly every maximal branch has order type $\gamma+1$, every chain $\eta\restriction 1\subset\eta\restriction 2\subset\eta\restriction 3\subseteq \cdots$ of any length, has a unique limit in the tree, and every element in $t_\varrho$, $\varrho<\gamma$, has an infinite number of successors (at most $\kappa$).

Suppose $rang(\eta_1)=\gamma$. Since $\eta_5$ is increasing and $$sup(rang(\eta_3))\ge sup(rang(\eta_5))\ge sup(rang(\eta_2)),$$ by Definition \ref{colorconst} item 6, $$sup(rang(\eta_2))\ge sup(rang(\eta_3))$$ and $$sup(rang(\eta_2))\ge sup(rang(\eta_4).)$$ This leads us to 
\begin{equation}
sup(rang(\eta_4))\leq sup(rang(\eta_3))=sup(rang(\eta_5))=sup(rang(\eta_2)).
\end{equation}
If $\eta\restriction k\in J_f^\alpha$ holds for every $k\in \gamma$ and $\eta\notin J_f^\alpha$, then 
\begin{equation}
sup(rang(\eta_5))= \alpha.
\end{equation}

\begin{fact}[{Hyttinen-Kulikov, \cite[Claim 2.7]{HK} and Hyttinen-Moreno, \cite[Claim 4.9]{HM}}]
Suppose $\gamma<\kappa$ is such that for all $\epsilon<\kappa$, $\epsilon^\gamma<\kappa$. For all $f\in \beta^\kappa$, $|J_f|=\kappa$, $\mathcal{J}=(J_f^\alpha)_{\alpha<\kappa}$ is a good filtration of $J_f$ and $H_{\mathcal{J},J_f}\ =^\beta_\gamma\ f$
\end{fact}

We conclude one of the directions of Lemma \ref{colorisom}.

\begin{lemma}\label{colorisom_easy_dir}
Suppose $\gamma<\kappa$ is such that for all $\epsilon<\kappa$, $\epsilon^\gamma<\kappa$. For all $f,g\in \beta^\kappa$ if  $ J_f\cong_{ct} J_g$, then $f\ =^\beta_\gamma\ g$.
\end{lemma}

\subsection{The isomorphism of coloured trees}
In this section we will prove the missing direction of Lemma \ref{colorisom}.

\begin{fact}[{Hyttinen-Kulikov, \cite[Claim 2.6]{HK} and Hyttinen-Moreno, \cite[Claim 4.8]{HM}}]\label{4.5.1}
Suppose $\xi\in J_f^\alpha$ and $\eta\in J_f$. If $dom(\xi)$ is a successor ordinal smaller than $\gamma$, $\xi\subsetneq \eta$ and for every $k$ in $dom(\eta)\backslash dom(\xi)$, $\eta_1(k)=\xi_1(max(dom(\xi)))$ and $\eta_1(k)>0$, then $\eta\in J_f^\alpha$.
\end{fact}

We say that a set $X$ is a \textit{$\gamma$-club} if $X$ is unbounded and it is closed under $\gamma$-limits. Notice that $\eta=_\gamma^\beta\xi$ holds if and only if $\{\alpha<\kappa\mid cf(\alpha)=\gamma\ \&\ \eta(\alpha)=\xi(\alpha)\}$ contains a $\gamma$-club.

\begin{lemma}\label{colorisom_long_dir}
Suppose $\gamma<\kappa$ is such that for all $\epsilon<\kappa$, $\epsilon^\gamma<\kappa$. For every $f,g\in \beta^\kappa$ if $f\ =^\beta_\gamma\ g$, then $ J_f\cong_{ct} J_g$.
\end{lemma}

\begin{proof} Let $C'\subseteq \{\alpha<\kappa|cf(\alpha)=\gamma\ \&\ f(\alpha)=g(\alpha)\}$ be a $\gamma$-club testifying $f\  =^\beta_\gamma\ g$, and let $C\supset C'$ be the closure of $C'$ under limits. We are going to construct an isomorphism between $J_f$ and $J_g$ by induction.

Let us define continuous increasing sequences $(\alpha_i)_{i<\kappa}$ of ordinals and $(F_{\alpha_i})_{i<\kappa}$ of partial color-preserving isomorphism from $J_f$ to $J_g$ such that:
\begin{itemize}
\item[a)] If $i$ is a successor, then $\alpha_i$ is a successor ordinal and there exists $\iota\in C$ such that $\alpha_{i-1}<\iota<\alpha_i$ and thus if $i$ is a limit, then $\alpha_i\in C$.
\item[b)] Suppose that $i=\rho+n$, where $\rho$ is a limit ordinal or $0$, and $n<\omega$ is even. Then $dom(F_{\alpha_i})=J_f^{\alpha_i}$.
\item[c)] Suppose that $i=\rho+n$, where $\rho$ is a limit ordinal or $0$, and $n<\omega$ is odd. Then $rang(F_{\alpha_i})=J_g^{\alpha_i}$.
\item[d)] If $dom(\xi)<\gamma$, $\xi\in dom (F_{\alpha_i})$, $\eta\restriction {dom(\xi)}=\xi$ and for every $k\ge dom(\xi)$ $$\eta_1(k)=\xi_1(sup (dom(\xi)))\text{ and } \eta_1(k)>0$$ then $\eta\in dom (F_{\alpha_i})$. Similar for $rang(F_{\alpha_i})$.
\item[e)] If $\xi\in dom(F_{\alpha_i})$ and $k<dom(\xi)$, then $\xi\restriction {k}\in dom(F_{\alpha_i})$.
\item[f)] For all $\eta\in dom (F_{\alpha_i})$, $dom(\eta)=dom(F_{\alpha_i}(\eta))$.
\end{itemize}
For every ordinal $\alpha$ denote by $M(\alpha)$ the ordinal that is order isomorphic to the lexicographic order of $\gamma\times \alpha^4$.

\textit{\bf First step (i=0).}

Let $\alpha_0=\iota+1$ for some $\iota\in C$. Let $\rho$ be an ordinal such that there is a $\beta$-coloured tree isomorphism $h: Z^{0,M(\iota)}_\rho\rightarrow J_f^{\alpha_0}$ and $Q(Z^{0,M(\iota)}_\rho)=0$. It is easy to see that such $\rho$ exists, by the way our enumeration was chosen.

Since $Z^{0,M(\iota)}_\rho$ and $J_f^{\alpha_0}$ are closed under initial segments, then $|dom(h^{-1}(\eta))|=|dom(\eta)|$. Also both domains are intervals containing zero, therefore $dom(h^{-1}(\eta))=dom(\eta)$.

Define $F_{\alpha_0}(\eta)$ for $\eta\in J^{\alpha_0}_f$ as follows, let $F_{\alpha_0}(\eta)$ be the function $\xi$ with $dom(\xi)=dom(\eta)$, and for all $\kappa<dom(\xi)$:
\begin{itemize}
\item $\xi_1(k)=1$
\item $\xi_2(k)=0$
\item $\xi_3(k)=M(\iota)$
\item $\xi_4(k)=\rho$
\item $\xi_5(k)=h^{-1}(\eta)(k)$
\end{itemize}
To check that $\xi\in J_g$, we will check every item of Definition \ref{colorconst}. Since $rang(F_{\alpha_0})=\{1\}\times\{0\}\times\{M(\iota)\}
\times\{\rho\}\times Z^{0,M(\iota)}_\rho$, $\xi$ satisfies 1. Also $\xi_5=h^{-1}(\eta)\in Z^{0,M(\iota)}_\rho$, by definition of $Z^{\alpha,\iota}_\rho$, we know that $\xi_5$ is strictly increasing with respect to the lexicographic order. Thus $\xi$ satisfies item 2. Notice that $\xi$ is constant in every component except for $\xi_5$, therefore $\xi$ satisfies the items 3, 6, 7, 8, 10 (a). Clearly $\xi_1(i)\neq 0$. So $\xi$ satisfies item 4. Since $\xi_2(k)=0$ for every $k$, $\xi$ satisfies 5. Notice that $[0,\gamma)=\xi_1^{-1}(1)$ and $Z^{\xi_2(k),\xi_3(k)}_{\xi_4(k)}=Z^{0,M(\iota)}_\rho$ for every $k$, therefore $\xi_5\in Z^{\xi_2(0),\xi_3(0)}_{\xi_4(0)}$ and $\xi$ satisfies 7.

Let us show that the conditions a)-f) are satisfied, the conditions a) and c) are clearly satisfied. By the way $F_{\alpha_0}$ was defined, $dom(F_{\alpha_0})=J^{\alpha_0}_f$ and  $dom(\eta)=dom(F_{\alpha_0}(\eta))$, these are the conditions b), e) and f). Since $dom(F_{\alpha_0})=J^{\alpha_0}_f$, Fact \ref{4.5.1} implies d) for $dom(F_{\alpha_0})$. For d) with $rang(F_{\alpha_0})$, suppose $\xi\in rang(F_{\alpha_0})$ and $\eta \in J_g$ are as in the assumption. Then $\eta_1(k)=\xi_1(k)=1$ for every $k<dom(\eta)$. By 8 in $J_g$, $\eta_2(k)=\xi_2(k)=0$, $\eta_3(k)=\xi_3(k)=M(\iota)$ and $\eta_4(k)=\xi_4(k)=\rho$ for every $k<dom(\eta)$. By 9 in $J_g$, $\eta_5\in Z^{0,M(\iota)}_\rho$ and since $rang(F_{\alpha_0})=\{1\}\times\{0\}\times\{M(\iota)\}
\times\{\rho\}\times Z^{0,M(\iota)}_\rho$, we can conclude that $\eta\in rang(F_{\alpha_0})$.

\textit{\bf Odd successor step.}

Suppose that $j<k$ is a successor ordinal such that $j=\iota_j+n_j$ for some limit ordinal (or 0) $\iota_j$ and an odd integer $n_j$. Assume  $\alpha_l$ and $F_{\alpha_l}$ are defined for every $l<j$ satisfying the conditions a)-f).

Let $\alpha_j=\iota+1$, where $\iota\in C$ is such that $\iota>\alpha_{j-1}$, and $rang(F_{\alpha_{j-1}})\subset J_g^\iota$, such a $\iota$ exists because $|rang(F_{\alpha_{j-1}})|\leq 2^{|\alpha_{j-1}|}$ and for all $\epsilon<\kappa$, $\epsilon^\gamma<\kappa$.

When $\eta\in rang(F_{\alpha_{j-1}})$ has domain $m<\gamma$, define $$W(\eta)=\{\zeta|dom(\zeta)=[m,s),m<s\leq \gamma, \eta^\frown \langle m,\zeta(m) \rangle\notin rang(F_{\alpha_{j-1}})\text{ and }\eta^\frown \zeta\in J_g^{\alpha_j}\}$$ with the color function $c_{W(\eta)}(\zeta)=c_g(\eta^\frown \zeta)$ for every $\zeta\in W(\eta)$ with $s=\gamma$. Denote $\xi'=F^{-1}_{\alpha_{j-1}}(\eta)$, $\alpha=\xi'_3(m-1)+\xi'_4(m-1)$ (if $m$ is a limit ordinal, then $\alpha=sup_{\varrho<m} \xi_2(\varrho)$) and $\varrho=\alpha+M(\alpha_j)$. Choose an ordinal $\rho_\eta$ such that $Q(Z_{\rho_\eta}^{\alpha,\varrho})=m$ and there is an isomorphism $h_\eta:Z_{\rho_\eta}^{\alpha,\varrho}\rightarrow W(\eta)$. We will define $F_{\alpha_{j}}$ by defining its inverse such that $rang(F_{\alpha_{j}})=J_g^{\alpha_j}$.

Each $\eta\in J_g^{\alpha_j}$ satisfies one of the followings:
\begin{itemize}
\item[(*)] $\eta\in rang(F_{\alpha_{j-1}})$.
\item[(**)] $\exists m<dom(\eta)(\eta\restriction {m}\in rang(F_{\alpha_{j-1}})\wedge \eta\restriction {(m+1)}\notin rang(F_{\alpha_{j-1}}))$.
\item[(***)] $\forall m<dom(\eta)(\eta\restriction {(m+1)}\in rang(F_{\alpha_{j-1}})\wedge \eta\notin rang(F_{\alpha_{j-1}}))$.
\end{itemize}
We define $\xi=F_{\alpha_{j}}^{-1}(\eta)$ as follows. There are three cases:

{\bf Case $\eta$ satisfies $(*)$.}

Define $\xi(n)=F_{\alpha_{j-1}}^{-1}(\eta)(n)$ for all $n<dom(\eta)$.

{\bf Case $\eta$ satisfies $(**)$.}

This case is divided in two subcases, when $m$ is a limit ordinal and when $m$ is a successor ordinal.
Let $m$ witnesses (**) for $\eta$ and suppose $m$ is a successor ordinal. For every $n<dom(\xi)$
\begin{itemize}
\item If $n<m$, then $\xi(n)=F_{\alpha_{j-1}}^{-1}(\eta\restriction m)(n)$.
\item For every $n\ge m$. Let
\begin{itemize}
\item $\xi_1(n)=\xi_1(m-1)+1$
\item $\xi_2(n)=\xi_3(m-1)+\xi_4(m-1)$
\item $\xi_3(n)=\xi_2(m)+M(\alpha_j)$
\item $\xi_4(n)=\rho_{\eta\restriction m}$
\item $\xi_5(n)=h^{-1}_{\eta\restriction m}(\eta\restriction {[m,dom(\eta))})(n)$
\end{itemize}
\end{itemize}
Notice that, $\eta\restriction {[m,dom(\eta))}$ is an element of $W({\eta\restriction m})$, this makes possible the definition of $\xi_5$.

Let us check the items of Definition \ref{colorconst} to see that $\xi\in J_f$. Clearly item 1 is satisfied. By the induction hypothesis, $\xi\restriction m$ is increasing, $\xi_1(m)=\xi_1(m-1)+1$ so $\xi(m-1)<\xi(m)$, and $\xi_k$ is constant on $[m, \gamma)$ for $k\in \{1,2,3,4\}$. Since $h^{-1}_{\eta\restriction m}(\eta)\in Z_{\rho_\eta}^{\alpha,\varrho}$, $\xi_5$ is increasing, $\xi$ is increasing with respect to the lexicographic order. So $\xi$ satisfies item 2. We conclude that $\xi_1(i)\leq \xi_1(i+1)\leq \xi_1(i)+1$, so $\xi$ satisfies item 3. For every $i<\gamma$, $\xi_1(i)=0$ implies $i<m$, so $\xi(i)=F_{\alpha_{j-1}}^{-1}(\eta\restriction m)(i)$ and by the induction hypothesis, $\xi$ satisfies item 4. By the induction hypothesis, $\xi\restriction m\in J_f$. Since $\xi_2(n)=\xi_3(m-1)+\xi_4(m-1)$ holds for every $n\ge m$, $\xi$  satisfies 5.
By the induction hypothesis, for every $i+1<m$, $\xi_1(i)<\xi_1(i+1)$ implies $\xi_2(i+1)\ge \xi_3(i)+\xi_4(i)$. On the other hand $\xi_1(i)=\xi_1 (j)$ implies $\xi_k (i)=\xi_k (j)$ for $k\in \{2,3,4\}$. Clearly $\xi_2(m)\ge \xi_3(m-1)+\xi_4(m-1)$ and $\xi_k (i)=\xi_k (i+1)$ for $i\ge m$ and $k\in\{2,3,4\}$, then $\xi$ satisfies items 6 and 8.

By the induction hypothesis, $\xi\restriction m\in J_f$. Since $\xi_1(n)=\xi_1(m-1)+1$ and  $\xi_2(n)=\xi_3(m-1)+\xi_4(m-1)$ hold for every $n\ge m$, $\xi$  satisfies 7.
Suppose $[i,j)=\xi_1^{-1}(k)$  for some $k$ in $rang(\xi)$. Either $j<m$ or $m=i$. If $j<m$, then by the induction hypothesis $\xi_5\restriction {[i,j)}\in Z_{\xi_4(i)}^{\xi_2(i),\xi_3(i)}$. If $[i,j)=[m,dom(\xi))$, then $\xi_5\restriction {[i,j)}=h^{-1}_{\eta\restriction m}(\eta\restriction {[m,dom(\xi))})\in Z_{\xi_4(m)}^{\xi_2(m),\xi_3(m)}$. Thus $\xi$ satisfies item 9. Since $\xi$ is constant on $[m,\gamma)$, $\xi$ satisfies 10 (a). Finally by item 10 (a) when $dom(\zeta)=\gamma$, we conclude $c_f(\xi)=c(\xi_5\restriction {[m,\gamma)})$, where $c$ is the color of $Z_{\xi_4(m)}^{\xi_2(m),\xi_3(m)}$. Since $\xi_5\restriction {[m,\gamma)}=h^{-1}_{\eta\restriction m}(\eta\restriction {[m,\gamma)})$, $c_f(\xi)=c(h^{-1}_{\eta\restriction m}(\eta\restriction {[m,\gamma)}))$. Since $h$ is an isomorphism, $c_f(\xi)=c_{W(\eta\restriction _m)}(\eta\restriction {[m,\gamma)})=c_g(\eta)$.

Let $m$ witnesses (**) for $\eta$ and suppose $m$ is a limit ordinal. For every $n<dom(\xi)$
\begin{itemize}
\item If $n<m$, then $\xi(n)=F_{\alpha_{j-1}}^{-1}(\eta\restriction m)(n)$.
\item For every $n\ge m$. Let
\begin{itemize}
\item $\xi_1(n)=sup_{\varrho<m}\xi_1(\varrho)$
\item $\xi_2(n)=sup_{\varrho<m}\xi_2(\varrho)$
\item $\xi_3(n)=\xi_2(m)+M(\alpha_j)$
\item $\xi_4(n)=\rho_{\eta\restriction m}$
\item $\xi_5(n)=h^{-1}_{\eta\restriction m}(\eta\restriction {[m,dom(\eta))})(n)$.
\end{itemize}
\end{itemize}
Notice that, $\eta\restriction {[m,dom(\eta))}$ is an element of $W({\eta\restriction m})$, this makes possible the definition of $\xi_5$.

Let us check the items of Definition \ref{colorconst} to see that $\xi\in J_f$. Clearly item 1 is satisfied. By induction hypothesis, $\xi\restriction m$ is increasing and $\xi_1(m)=sup_{\varrho<m}\xi_1(\varrho)$. So $\xi(\varrho)<\xi(m)$ holds for every $\varrho<m$, and $\xi_k$ is constant on $[m, \gamma)$ for $k\in \{1,2,3,4\}$. Since $h^{-1}_{\eta\restriction m}(\eta)\in Z_{\rho_\eta}^{\alpha,\varrho}$, $\xi_5$ is increasing. We conclude that $\xi$ is increasing with respect to the lexicographic order, so $\xi$ satisfies item 2. We conclude that $\xi_1(i)\leq \xi_1(i+1)\leq \xi_1(i)+1$, so $\xi$ satisfies item 3. For every $i<\gamma$, $\xi_1(i)=0$ implies $i<m$, so $\xi(i)=F_{\alpha_{j-1}}^{-1}(\eta\restriction m)(i)$ and by the induction hypothesis $\xi$ satisfies item 4. By the induction hypothesis, $\xi\restriction m\in J_f$. Since $\xi_2(n)=sup_{\varrho<m}\xi_2(\varrho)$ holds for every $n\ge m$, $\xi$  satisfies 5.
By the induction hypothesis, for every $i+1<m$, $\xi_1(i)<\xi_1(i+1)$ implies $\xi_2(i+1)\ge \xi_3(i)+\xi_4(i)$. On the other hand $\xi_1(i)=\xi_1 (j)$ implies $\xi_k (i)=\xi_k (j)$ for $k\in \{2,3,4\}$. Clearly $\xi_2(m)\ge sup_{\varrho<m}\xi_3(\varrho)$ and $\xi_k (i)=\xi_k (j)$ for $j,i\ge m$ and $k\in\{2,3,4\}$, then $\xi$ satisfies items 6 and 8.

By the induction hypothesis, $\xi\restriction m\in J_f$. Since $\xi_1(n)=sup_{\varrho<m}\xi_1(\varrho)$ and $\xi_2(n)=sup_{\varrho<m}\xi_2(\varrho)$ hold for every $n\ge m$, $\xi$ satisfies 7.
Suppose $[i,j)=\xi_1^{-1}(k)$  for some $k$ in $rang(\xi)$. Either $j<m$ or $m=i$, notice that if $i<m<j$, then $\eta\restriction {(m+1)}\in rang(F_{\alpha_{j-1}}))$. If $j<m$, then by the induction hypothesis $\xi_5\restriction {[i,j)}\in Z_{\xi_4(i)}^{\xi_2(i),\xi_3(i)}$. Thus $[i,j)=[m,dom(\xi))$, then $\xi_5\restriction {[i,j)}=h^{-1}_{\eta\restriction m}(\eta\restriction {[m,dom(\xi))})\in Z_{\xi_4(m)}^{\xi_2(m),\xi_3(m)}$. Thus $\xi$ satisfies item 9. Since $\xi$ is constant on $[m,\gamma)$, $\xi$ satisfies 10 (a). Finally by item 10 (a) when $dom(\zeta)=\gamma$, $c_f(\xi)=c(\xi_5\restriction {[m,\gamma)})$, where $c$ is the color of $P_{\xi_4(m)}^{\xi_2(m),\xi_3(m)}$. Since $\xi_5\restriction {[m,\gamma)}=h^{-1}_{\eta\restriction m}(\eta\restriction {[m,\gamma)})$, $c_f(\xi)=c(h^{-1}_{\eta\restriction m}(\eta\restriction {[m,\gamma)}))$ and since $h$ is an isomorphism, $c_f(\xi)=c_{W(\eta\restriction _m)}(\eta\restriction {[m,\gamma)})=c_g(\eta)$.

{\bf Case $\eta$ satisfies $(***)$.}

Clearly $dom(\eta)=\gamma$. By the induction hypothesis and condition d), $rang(\eta_1)=\gamma$, otherwise $\eta\in rang(F_{\alpha_{j-1}})$. Let $F^{-1}_{\alpha_j}(\eta)=\xi=\cup_{n<\gamma}F^{-1}_{\alpha_{j-1}}(\eta\restriction n)$, by the induction hypothesis, $\xi$ is well defined. Since $\xi\restriction n\in J_f$ for all $n<\gamma$, then $\xi\in J_f$. Let us check that $c_f(\xi)=c_g(\eta)$. Notice that $\xi\notin J_f^{\alpha_{j-1}}$, otherwise by the induction hypothesis f), 
\begin{equation*}
F_{\alpha_{j-1}}(\xi)=\bigcup_{n<\gamma}F_{\alpha_{j-1}}(\xi\restriction n)=\bigcup_{n<\gamma}\eta\restriction n=\eta 
\end{equation*}
giving us $\eta\in rang(F_{\alpha_{j-1}})$. By Equation (5), $sup(rang(\xi_5))=\alpha_{j-1}$ and $\xi$ satisfies item 10 b) in $J_f$. Therefore $c_f(\xi)=f(\alpha_{j-1})$. By the definition of $J_f^\alpha$ and since $\xi\restriction n\in J_f^{\alpha_{j-1}}$ holds for all $n<\gamma$, $\alpha_{j-1}$ is a limit ordinal. By condition a), $j-1$ is a limit ordinal and $\alpha_{j-1}\in C$. The conditions b) and c) ensure that $rang(F_{\alpha_{j-1}})=J_f^{\alpha_{j-1}}$. This implies, $\eta\notin J_f^{\alpha_{j-1}}$. 
Therefore $\alpha_{j-1}$ has cofinality $\gamma$, $\alpha_{j-1}\in C'$ and $f(\alpha_{j-1})=g(\alpha_{j-1})$. By item 10 b) in $J_g$, $c_g(\eta)=g(\alpha_{j-1})=f(\alpha_{j-1})=c_f(\xi)$.

Let us show that $F_{\alpha_i}$ is a color preserving partial isomorphism. We already showed that $F_{\alpha_i}$ preserve the colors, so we only need to show that 
\begin{equation}
\eta\subsetneq \xi \Leftrightarrow F_{\alpha_i}^{-1}(\eta)\subsetneq F^{-1}_{\alpha_i}(\xi).
\end{equation}
From left to right.

When $\eta,\xi\in rang(F_{\alpha_{i-1}})$, the induction hypothesis implies (6) from left to right. If $\eta\in rang(F_{\alpha_{i-1}})$ and $\xi\notin rang(F_{\alpha_{i-1}})$. Then the construction implies (6) from left to right. If $\eta,\xi\notin rang(F_{\alpha_{i-1}})$, then $\eta,\xi$ satisfy (**). Let $m_1$ and $m_2$ be the ordinals that witness (**) for $\eta$ and $\xi$, respectively. Notice that $m_2<dom(\eta)$, otherwise, $\eta\in rang(F_{\alpha_{i-1}})$. If $m_1<m_2$, then $\eta\in rang(F_{\alpha_{i-1}})$ which is not the case. A similar argument shows that $m_2<m_1$ cannot hold. We conclude that $m_1=m_2$. By the construction of $F_{\alpha_i}$, we cocnlude that $F_{\alpha_i}^{-1}(\eta)\subsetneq F^{-1}_{\alpha_i}(\xi)$.

From right to left.

When $\eta,\xi\in rang(F_{\alpha_{i-1}})$, the induction hypothesis implies (6) from right to left. If $\eta\in rang(F_{\alpha_{i-1}})$ and $\xi\notin rang(F_{\alpha_{i-1}})$, the construction implies (6) from right to left. If $\eta,\xi\notin rang(F_{\alpha_{i-1}})$, then $\eta,\xi$ satisfy (**). Let $m_1$ and $m_2$ be the respective ordinal numbers that witness (**) for $\eta$ and $\xi$, respectively. Notice that $m_2<dom(\eta)$, otherwise, $F_{\alpha_i}^{-1}(\eta)=F_{\alpha_{i-1}}^{-1}(\eta)$ and $\eta\in rang(F_{\alpha_{i-1}})$. Let us denote by $\vartheta$ the inverse map $F_{\alpha_i}^{-1}$ (e.g. $\vartheta(\zeta)=F_{\alpha_i}^{-1}(\zeta)$), and the first component by $\vartheta_1$ (e.g. $\vartheta_1(\zeta)=F_{\alpha_i}^{-1}(\zeta)_1$).

If $m_1<m_2$ and $m_2$ is a successor ordinal, then 
$$
\begin{array} {lcl} 
\vartheta_1(\eta)(m_2-1) & = & (\vartheta(\xi)\restriction _{m_2})_1(m_2-1)\\ & < & \vartheta_1(\xi\restriction _{m_2})(m_2-1)+1\\ & = & \vartheta_1(\eta)(m_2)\\ & = & \vartheta_1(\eta)(m_2-1).
\end{array}
$$
If $m_1<m_2$ and $m_2$ is a limit ordinal, then 
$$
\begin{array} {lcl} 
\forall \rho\in [m_1,m_2)\ \ \vartheta_1(\eta)(\rho) & = & (\vartheta(\xi)\restriction _{m_2})_1(\rho)\\ & < & sup_{n<m_2}\vartheta_1(\xi\restriction _{m_2})(n)\\ & = & \vartheta_1(\eta)(m_2)\\ & = & \vartheta_1(\eta)(\rho).
\end{array}
$$
This cannot hold. A similar argument shows that $m_2<m_1$ cannot hold. We conclude that $m_1=m_2$.

By the induction hypothesis, $F_{\alpha_{i-1}}^{-1}(\eta\restriction {m_1})=F_{\alpha_{i-1}}^{-1}(\xi\restriction {m_2})$ implies $\eta\restriction {m_1}=\xi\restriction {m_2}$ (it also implies $h_{\eta\restriction {m_1}}=h_{\xi\restriction {m_2}}$). Since $F_{\alpha_{i-1}}^{-1}(\eta\restriction {m_1})(n)=F_{\alpha_i}^{-1}(\eta)(n)$ for all $n<m_1$, we only need to prove that $\eta\restriction {[m_1,dom(\eta))}\subsetneq \xi\restriction {[m_2,dom(\xi))}$. But $h_{\eta\restriction {m_1}}$ is an isomorphism and $F_{\alpha_i}^{-1}(\eta)_5(n) = F_{\alpha_{i}}^{-1}(\xi)_5(n)$ holds for every $n\ge m_1$, so $h^{-1}_{\eta\restriction {m_1}}(\eta\restriction {[m_1,dom(\eta))})(n)=h^{-1}_{\xi\restriction {m_2}}(\xi\restriction {[m_2,dom(\xi))})(n)$. Therefore $\eta\restriction {[m_1,dom(\eta))}\subsetneq \xi\restriction {[m_2,dom(\xi))}$.

Let us check that this three constructions satisfy the conditions a)-f).

When $i$ is a successor, we have that $\alpha_{i-1}<\iota<\alpha_i=\iota+1$ for some $\iota\in C$, this is the condition a). Clearly the three cases satisfy b). We defined $F_{\alpha_i}^{-1}$ according to (*), (**), or (***). Since every $\eta\in J_g^{\alpha_j}$ satisfies one of these, $rang(F_{\alpha_i})=J_g^{\alpha_j}$ which is the condition c).

Let us show that $F_{\alpha_i}$ satisfies condition d). Let $\xi$ and $\eta$ be as in the assumptions of condition d) for domain. Notice that if $\xi\in dom(F_{\alpha_{i-1}})$, then the induction hypothesis ensure that $\eta\in dom(F_{\alpha_i})$. Suppose $\xi\notin dom(F_{\alpha_{i-1}})$, then $F_{\alpha_{i}}(\xi)\notin rang(F_{\alpha_{i-1}})$. Since $dom(\xi)<\gamma$, $F_{\alpha_{i}}(\xi)$ satisfies (**). Let $m$ be the number witnessing it.
If $m$ is a limit ordinal, then $dom(\xi)\ge m+1$. Therefore $\xi\restriction {m+1}\in J_f^{\alpha_i}$ and by Fact \ref{4.5.1}, $\eta\in J_f^{\alpha_i}$. If $m$ is a successor ordinal, then $\xi\in J_f^{\alpha_i}$ and by Fact \ref{4.5.1}, $\eta\in J_f^{\alpha_i}$. By item 8 in $J_f^{\alpha_i}$, $\eta_k$ is constant on $[m,dom(\eta))$ for $k\in \{2,3,4\}$. By Definition \ref{colorconst} item 9 in $J_f^{\alpha_i}$, $\eta_5\restriction {[m,dom(\eta))}\in Z^{\alpha,\iota}_{\rho_{\xi\restriction _m}}$. Let $\zeta=h_{\xi\restriction m}(\eta_{[m,dom(\eta))})$, then $\eta=F^{-1}_{\alpha_{i}}(F_{\alpha_i}(\xi\restriction m)^\frown \zeta)$ and $\eta\in dom(F_{\alpha_i})$.

Using the same argument, the condition d) can be proved.

For the conditions e) and f), notice that $\xi$ was constructed such that $dom(\xi)=dom(\eta)$ and $\xi\restriction k\in dom(F_{\alpha_i})$, which are these conditions.

\textit{\bf Even successor step.}

The construction of $F_{\alpha_j}$ such that $dom(F_{\alpha_j})=J_f^{\alpha_i}$ follows as in the odd successor step, with the equivalent definitions.

\textit{\bf Limit step.}

Assume $j$ is a limit ordinal. Let $\alpha_j=\cup_{i<j}\alpha_i$ and $F_{\alpha_j}=\cup_{i<j}F_{\alpha_i}$. Clearly $F_{\alpha_j}:J_f^{\alpha_j}\rightarrow J_g$ and satisfies condition c). Since for $i$ successor, $\alpha_i$ is the successor of an ordinal in $C$, then $\alpha_j\in C$ and satisfies condition a). Also $F_{\alpha_j}$ is a partial isomorphism. Recall that $\cup_{i<j}J_f^{\alpha_i}=J_f^{\alpha_j}$, equivalent for $J_g$. By the induction hypothesis and conditions b) and c) for $i<j$, we have $dom(F_{\alpha_j})=J_f^{\alpha_j}$ (this is the condition b)) and $rang(F_{\alpha_j})=J_g^{\alpha_j}$. This and Fact \ref{4.5.1} ensure that condition d) is satisfied. By the induction hypothesis, for every $i<j$, $F_{\alpha_i}$ satisfies conditions e) and f). Thus $F_{\alpha_j}$ satisfies conditions e) and f).

Define $F=\cup_{i<\kappa}F_{\alpha_i}$, clearly, it is an isomorphism between $J_f$ and $J_g$.

\end{proof}

\subsection{Ordered trees}

 By applying similar ideas to the ones used by Abraham in \cite{Abr}, it was possible to construct highly saturated ordered coloured trees in \cite{Mor21}.
 Following the construction presented in \cite{Mor21}, we will use the $\kappa$-colorable linear order $I$ to construct ordered trees with $\gamma+1$ levels, $A^f$, for every $f\in \beta^\kappa$ with the property $A^f \cong A^g$ if and only if $f\ =^\beta_{\gamma}\ g$. 

\begin{defn}\label{Shtree}
Let $K_{tr}^\gamma$ be the class of models $(A,\prec, (P_n)_{n\leq \gamma},<, \wedge)$, where:
\begin{enumerate}
\item There is a linear order $(\mathcal{I},<_\mathcal{I})$ such that $A\subseteq \mathcal{I}^{\leq\gamma}$.
\item $A$ is closed under initial segment.
\item $\prec$ is the initial segment relation.
\item $\wedge(\eta,\xi)$ is the maximal common initial segment of $\eta$ and $\xi$.
\item Let $lg(\eta)$ be the length of $\eta$ (i.e. the domain of $\eta$) and $P_n=\{\eta\in A\mid lg(\eta)=n\}$ for $n\leq \gamma$.
\item For every $\eta\in A$ with $lg(\eta)<\gamma$, define $Suc_A(\eta)$ as $\{\xi\in A\mid \eta\prec \xi\ \&\ lg(\xi)=lg(\eta)+1\}$. For all $\eta\in A\backslash P_\gamma$, $<\restriction Suc_A(\eta)$ is the induced linear order from $\mathcal{I}$, i.e. $$\eta^\frown \langle x\rangle < \eta^\frown \langle y\rangle \Leftrightarrow x<_\mathcal{I} y.$$
\item $<$ is $\bigcup_{\eta\in A}(<\restriction Suc_A(\eta))$, i.e. if $\xi<\zeta$, then there is $\eta\in A$ such that $\xi,\zeta\in Suc_A(\eta)$.
\item If $\eta$ and $\xi$ have no immediate predecessor  and $\{\zeta\in A\mid \zeta\prec\eta\}=\{\zeta\in A\mid \zeta\prec\xi\}$, then $\eta=\xi$.
\end{enumerate}
\end{defn}

The elements of $K_{tr}^\gamma$ are called \textit{ordered trees}.
For each $f\in \beta^{\kappa}$ we will use the $\beta$-coloured trees $J_f$ to construct new $\beta$-coloured trees that will be ordered later. An \textit{ordered $\beta$-coloured tree} is a model $(A,\prec, (P_n)_{n\leq \gamma},<, \wedge, c)$ where $(A,\prec, (P_n)_{n\leq \gamma},<, \wedge)$ is an ordered tree, and $(A,\prec,c)$ is a $\beta$-coloured tree.

For every $f\in \beta^\kappa$, we have constructed the $\beta$-coloured tree $J_f$ and the filtration $(J_f^\alpha)_{\alpha<\kappa}$. 
Notice that $J_f^0=\{\emptyset\}$ and $dom(\emptyset)=0$.
Let us denote by $\acc(\kappa)=\{\alpha<\kappa\mid \alpha=0~ \text{or}~ \alpha \text{ is a limit ordinal}\}$. For all $\alpha\in \acc(\kappa)$ and  $\eta\in J_f^\alpha$ with $dom(\eta)=m<\gamma$ define 
$$W_\eta^\alpha=\{\zeta\mid dom(\zeta)=[m,s), m\leq s\leq \gamma, \eta^\frown\zeta\in J_f^{\alpha+\omega}, \eta^\frown \langle m,\zeta(m)\rangle\notin J_f^\alpha\}.$$ 
Notice that by the way $J_f$ was constructed, for every $\eta\in J_f$ with domain smaller than $\gamma$ and $\alpha<\kappa$, the set $$\{(\vartheta_1,\vartheta_2,\vartheta_3,\vartheta_4,\vartheta_5)\in (\gamma\times
\kappa^4)\backslash (\gamma\times\alpha^4)\mid \eta^\frown (\vartheta_1,\vartheta_2,\vartheta_3,\vartheta_4,\vartheta_5)\in J_f^{\alpha+\omega}\}$$ is either empty or has size $\omega$.
Let $\sigma_\eta^\alpha$ be an enumeration of this set, when this set is not empty.

Let us denote by $\mathcal{T}=(\kappa\times \omega\times \acc(\kappa)\times \gamma\times
\kappa\times\kappa\times
\kappa\times\kappa)^{\leq\gamma}$. 
For every $\xi \in \mathcal{T}$ there are functions $\{\xi_i\in \kappa^{\leq \omega}\mid 0<i\leq 8\}$ such that for all $i\leq 8$, $dom(\xi_i)=dom(\xi)$ and for all $n\in dom(\xi)$, $\xi(n)=(\xi_1(n),\xi_2(n),\xi_3(n),\xi_4(n),\xi_5(n),\xi_6(n),\xi_7(n),\xi_8(n))$. 
For every $\xi\in \mathcal{T}$ let us denote $(\xi_4,\xi_5,\xi_6,\xi_7,\xi_8)$ by $\overline{\xi}$.

\begin{defn}\label{Gammas}
For all $\alpha\in \acc(\kappa)$ and  $\eta\in \mathcal{T}$ with $\overline{\eta}\in J_f$, $dom(\eta)=m<\gamma$ define
$\Gamma_\eta^\alpha$ as follows:

If $\overline{\eta}\in J_f^\alpha$, then $\Gamma_\eta^\alpha$ is the set of elements $\xi$ of $\mathcal{T}$ such that:

\begin{enumerate}
\item $\xi\restriction m=\eta,$
\item $\overline{\xi}\restriction dom(\xi)\backslash m \in W^\alpha_\eta,$
\item $\xi_3$ is constant on $dom(\xi)\backslash m,$
\item $\xi_3(m)=\alpha$,
\item for all $n\in dom(\xi)\backslash m$, let $\xi_2(n)$ be the unique $r<\omega$ such that $\sigma_{\zeta}^\alpha(r)=\overline{\xi}(n)$, where $\zeta=\overline{\xi}\restriction n$.
\end{enumerate}

If $\overline{\eta}\notin J_f^\alpha$, then $\Gamma_\eta^\alpha=\emptyset$.

\end{defn}

For $\eta\in \mathcal{T}$ with $\overline{\eta}\in J_f$, $dom(\eta)=m<\gamma$ define $$\Gamma(\eta)=\bigcup_{\alpha\in \acc(\kappa)}\Gamma_\eta^\alpha .$$
Finally we can define $A^f$ by induction. Let $T_f(0)=\{\emptyset\}$ and for all $n<\gamma$, $$T_f(n+1)=T_f(n)\cup\bigcup_{\eta\in T_f(n)~dom(\eta)=n}\Gamma(\eta).$$ For $n\leq \gamma$ a limit ordinal, 
$$\bar{T}_f(n)=\bigcup_{m<n}T_f(m)$$
and
$$T_f(n)=\bar{T}_f(n)\cup \{\eta\in \mathcal{T}\mid dom(\eta)=n\ \&\ \forall m<n\ (\eta\restriction m\in \bar{T}_f(n))\}.$$

For  $0<i\leq 8$ let us denote by $s_i(\eta)=sup\{\eta_i(n)\mid n<\gamma\}$ and $s_\gamma(\eta)=max\{s_i(\eta)\mid i\leq 8\}$. Finally $$A^f=T_f(\gamma).$$ 

Define the color function $d_f$ by 
\begin{equation*}
d_f(\eta)=\begin{cases} c_f(\overline{\eta}) &\mbox{if } s_1(\eta)< s_\gamma(\eta)\\
f(s_1(\eta)) & \mbox{if } s_1(\eta)= s_\gamma(\eta). \end{cases}
\end{equation*}
It is clear that $A^f$ is closed under initial segments, indeed the relations $\prec$, $(P_n)_{n\leq\gamma}$, and $\wedge$ of Definition \ref{Shtree} have a canonical interpretation in $A^f$. 

Now we finish the construction of $A^f$ by using the $\kappa$-colorable linear order $I$ of Remark \ref{small_order}. 
We have to define $<\restriction Suc_{A^f}(\eta)$ for all $\eta\in A^f$ with domain smaller than $\gamma$. 
Properly speaking, $A^f$ will not be an ordered $\beta$-coloured tree as in Definition \ref{Shtree}, but it will be isomorphic to an ordered $\beta$-coloured tree as in Definition \ref{Shtree}. 

Let us proceed to define $<\restriction Suc_{A^f}(\eta)$.  Let $F:I\rightarrow\kappa$ be a $\kappa$-color function of $I$.
 
For any $\eta\in A^f$ with domain $m<\gamma$, we will define the order $<\restriction Suc_{A^f}(\eta)$ such that it is isomorphic to $I$ and satisfies the following:

$(\ast)$
\textit{
Let $sup(rng(\eta_3))=\vartheta$.
 For any set $B\subset Suc_{A^f}(\eta)$ of size less than $\kappa$, $p(x)$ a type of basic formulas over $B$ in the variable $x$ 
 , and any tuple $(\vartheta_2,\vartheta_3)\in \omega\times \acc(\kappa)$ with $\vartheta_3\ge \vartheta$, if $p(x)$ is realized in $Suc_{A^f}(\eta)$, then
 there are $\kappa$ many $\alpha<\kappa$ such that $\eta^\frown (\alpha,\vartheta_2,\vartheta_3,\sigma^{\vartheta_3}_{\overline{\eta}}(\vartheta_2))\models p$.
}

By the construction of $A^f$, an isomorphism between $\{(\vartheta_1,\vartheta_2,\vartheta_3)\in \kappa\times \omega\times \acc(\kappa)\mid \vartheta_3\ge \vartheta \}$ and $I$, induces an order in $Suc_{A^f}(\eta)$.

\begin{defn}

Let $F$ be a $\kappa$-color function of $I$ (see Theorem \ref{main_order_prop}). For all $\vartheta, \alpha<\kappa$, let us fix a bijections $\tilde{G}_\vartheta:\{(\vartheta_2,\vartheta_3)\in \omega\times \acc(\kappa) \mid \vartheta_3\ge \vartheta\}\rightarrow \kappa$ and $\tilde{H}_\alpha:F^{-1}[\alpha]\rightarrow\kappa$. 
Notice that these functions exist because $F$ is a $\kappa$-color function of $I$ and there are $\kappa$ tuples $(\vartheta_2,\vartheta_3)$of this form.

Let us define $\tilde{\mathcal{G}}_\vartheta:\{(\vartheta_1,\vartheta_2,\vartheta_3)\in \kappa\times \omega\times \acc(\kappa) \mid \vartheta_3\ge \vartheta\}\rightarrow I$, by: $$\tilde{\mathcal{G}}_\vartheta((\vartheta_1,\vartheta_2,\vartheta_3))=a,$$ where $a$ and $\alpha$ are the unique elements that satisfy:
\begin{itemize}
\item $\tilde{G}_\vartheta((\vartheta_2,\vartheta_3))=\alpha$;
\item $\tilde{H}_\alpha(a)=\vartheta_1$.
\end{itemize} 

\end{defn}

For any $\eta\in A^f$ with $sup(rng(\eta_3))=\vartheta$, the isomorphism $\tilde{\mathcal{G}}_\vartheta$ and $I$ induce an order in $Suc_{A^f}(\eta)$. Let us define $<\restriction Suc_{A^f}(\eta)$ as the induced order given by  $\tilde{\mathcal{G}}_\vartheta$ and $I$.

\begin{fact}
Suppose $\eta\in A^f$ with $sup(rng(\eta_3))=\vartheta$. Then
$<\restriction Suc_{A^f}(\eta)$ satisfies ($*$). 
\end{fact}
\begin{proof}
Let $b\in Suc_{A^f}(\eta)$, $(\vartheta_2,\vartheta_3)\in \omega\times \acc(\kappa)$ such that $\vartheta_3\ge \vartheta$, and $B\subseteq Suc_{A^f}(\eta)$ have size less than $\kappa$. Let $b(dom(\eta))=(b_1,b_2,b_3,b_4,b_5,b_6,b_7,b_8)$ and denote by $pr_{123}(B(dom(\eta)))$ the set $$\{(a_1,a_2,a_3)\in \kappa\times\omega\times \acc(\kappa)\mid \exists \xi\in B\ (\xi(dom(\eta))=(a_1,a_2,a_3,a_4,a_5,a_6,a_7,a_8))\}.$$ 
Let us denote by $q$ the type $$tp_{bs}(\tilde{\mathcal{G}}_\vartheta(b_1,b_2,b_3), \tilde{\mathcal{G}}_\vartheta (pr_{123}(B(dom(\eta)))), I).$$ By the construction of $\tilde{G}_\vartheta$ and since $F$ is a $\kappa$-color function of $I$, $$|\{a\in I\mid a\models q \ \&\ F(a)=\tilde{G}_\vartheta(\vartheta_2,\vartheta_3)\}|=\kappa.$$
Therefore for all $a$ such that $a\models q$ and $F(a)=\tilde{G}_\vartheta(\vartheta_2,\vartheta_3)$, $$\eta^\frown (\tilde{H}_{\tilde{G}_\vartheta((\vartheta_2,\vartheta_3))}(a),\vartheta_2,\vartheta_3,\sigma^{\vartheta_3}_{\overline{\eta}}(\vartheta_2))\models p$$
\end{proof}

It is clear that $(A^f,\prec,(P_n)_{n\leq\gamma},<, \wedge)$ is isomorphic to a subtree of $I^{\leq \gamma}$ in the sense of Definition \ref{Shtree}.

\begin{remark}\label{equal_orders}
Notice that for any $\eta\in A^f$, $<\restriction Suc_{A^f}(\eta)$ is isomorphic to $I$. Therefore for any $\zeta,\eta\in A^f$, $<\restriction Suc_{A^f}(\zeta)$ and $<\restriction Suc_{A^f}(\eta)$ are isomorphic. Even more, the construction of $<\restriction Suc_{A^f}(\eta)$ only depends on $sup(rng(\eta_3))=\vartheta$.
\end{remark}

\begin{thm}\label{Afcong}
Suppose $\gamma<\kappa$ is such that for all $\epsilon<\kappa$, $\epsilon^\gamma<\kappa$.  For all $f,g\in \beta^\kappa$, $f\ =^\beta_\gamma\ g$ if and only if $A^f\cong A^g$ (as ordered $\beta$-coloured trees).
\end{thm}

\begin{proof}
For every $f\in \beta^\kappa$ let us define the $\kappa$-representation $\mathbb{A}^f=\langle A^f_\alpha \mid \alpha<\kappa\rangle$ of $A^f$, 
$$
A^f_\alpha=\{\eta\in A^f\mid rng(\eta)\subseteq \vartheta\times\omega\times\vartheta\times\gamma\times\vartheta^4\text{ for some }\vartheta<\alpha\}.
$$

Let $f$ and $g$ be such that $f\ =^\beta_\gamma\ g$, there is $G$ a $\beta$-coloured trees isomorphism between $J_f$ and $J_g$. Let $C\subseteq \kappa$ be a club such that $\{\alpha\in C\mid cf(\alpha)=\gamma\}\subseteq \{\alpha<\kappa\mid f(\alpha)=g(\alpha)\}$. 
We will show that there are sequences $\{\alpha_i\}_{i<\kappa}$ and $\{F_i\}_{i<\kappa}$ with the following properties:
\begin{itemize}
\item $\{\alpha_i\}_{i<\kappa}$ is a club.
\item If $i$ is a successor, then there is $\vartheta\in C$ such that $\alpha_{i-1}<\vartheta<\alpha_i$.
\item Suppose $i=\iota+n$, where $\iota$ is limit or 0 and $n$ is odd. Then $F_i$ is a partial isomorphism between $A^f$ and $A^g$, and $A_{\alpha_i}^f\subseteq dom(F_i)$.
\item Suppose $i=\iota+n$, where $\iota$ is limit or 0 and $n$ is even. Then $F_i$ is a partial isomorphism between $A^f$ and $A^g$, and $A_{\alpha_i}^g\subseteq rng(F_i)$.
\item If $i$ is limit, then $F_i:A_{\alpha_i}^f\rightarrow A_{\alpha_i}^g$.
\item If $i<j$, then $F_i\subseteq F_j$.
\item For all $\eta\in dom(F_i)$, $G(\overline{\eta})=\overline{F_i(\eta)}$.
\end{itemize}

We will proceed by induction over $i$.

{\bf Case $i=0$}. Let $\alpha_0=0$ and $F_0(\emptyset)=\emptyset$. 

{\bf Case $i$ is successor}. Suppose $i=\iota+n$, with $\iota$ limit or 0 and $n$ even, is such that:
\begin{itemize}
    \item $F_i$ is a partial isomorphism.
    \item $A_{\alpha_i}^g\subseteq rng(F_i)$.
    \item For all $j<i$, $F_j\subseteq F_i$.
    \item For all $\eta\in dom(F_i)$, $G(\overline{\eta})=\overline{F_i(\eta)}$.
\end{itemize}

Let us choose $\alpha_{i+1}$ to be a successor ordinal such that $\alpha_{i}<\vartheta<\alpha_{i+1}$ holds for some $\vartheta\in C$ and enumerate $A^f_{\alpha_{i}}$ by $\{\eta_j\mid j<\Omega\}$ for some $\Omega<\kappa$. Denote by $B_j$ the set $\{x\in A^f_{\alpha_{i+1}}\backslash dom(F_i)\mid \eta_j\prec x\}$. 

By the induction hypothesis, we know that for all $j<\Omega$, $x\in B_j$, $\overline{F_i(\eta_j)}\prec G(\overline{x})$. By Remark \ref{equal_orders}, for all $\eta\in A^f$ and $\xi\in A^g$, $<\restriction Suc_{A^f}(\eta)$ and $<\restriction Suc_{A^g}(\xi)$ are isomorphic. Thus, since $|A^f_{\alpha_{i}}|, |B_0|<\kappa$, by $(\ast)$ there is an embedding $F^0_i$ from $(A^f_{\alpha_i}\cup B_0,\prec,<)$ to $(A^g,\prec, <)$ that extends $F_i$ and for all $\eta\in dom(F^0_i)$, $\overline{F^0_i(\eta)}=G(\overline{\eta})$. 

Suppose that $0<t<\Omega$ is such that the following hold:
\begin{itemize}
\item There is a sequence of embeddings $\{F_i^j\mid j<t\}$, where $F^j_i$ is an embedding from $(A^f_{\alpha_i}\cup \bigcup_{l\leq j}B_l,\prec,<)$ into $A^g$.
\item $F_i^l\subseteq F_i^j$ holds for all $l<j<t$.
\item For all $\eta\in dom(F^j_i)$, $\overline{F^j_i(\eta)}=G(\overline{\eta})$.
\end{itemize}

Since $|A^f_{\alpha_{i}}\cup \bigcup_{j< t}B_j|, |B_t|<\kappa$, by $(\ast)$ there is an embedding $F^t_i$ from $(A^f_{\alpha_i}\cup \bigcup_{j\leq t}B_j,\prec,<)$ to $(A^g,\prec, <)$ that extends $\bigcup_{j<t} F_i^j$ and for all $\eta\in dom(F^t_i)$, $\overline{F^t_i(\eta)}=G(\overline{\eta})$. 

Finally $F_{i+1}=\bigcup_{j<\Omega} F_i^j$ is as wanted.

The case $i=\iota+n$ with $n$ odd is similar. For $i$ limit, we define $\alpha_i=\bigcup_{j<i} \alpha_j$ and $F_{\alpha_{i}}=\bigcup_{j<i} F_j$.

It is clear that $F=\bigcup_{j<\kappa} F_j$ witnesses that $A^f$ and $A^g$ are isomorphic as ordered trees. Let us show that $d_f(\eta)=d_g(F(\eta))$, suppose $\eta\in A^f$ is a leaf. Let $l$ be the least ordinal such that $\eta\in A^f_{\alpha_l}$. If there is $n<\gamma$ such that for all $j<l$, $\eta\restriction n\notin A^f_{\alpha_j}$, then by the way $F$ was constructed, $d_f(\eta)=d_g(F(\eta))$. On the other hand, if for all $n<\gamma$ there is $j<l$ such that $\eta\restriction n\in A^f_{\alpha_j}$, then there is a $\gamma$-cofinal ordinal $i$ such that $s_\gamma(\eta)=\alpha_i$ and $i+1=l$. 
By the construction of $A^f$ we know that
$$
d_f(\eta)=\begin{cases} c_f(\overline{\eta}) &\mbox{if } s_1(\eta)< s_\gamma(\eta)\\
f(s_1(\eta)) & \mbox{if } s_1(\eta)= s_\gamma(\eta). \end{cases}
$$
Since $s_\gamma(\eta)=\alpha_i$, either $d_f(\eta)=f(s_1(\eta))$ (if $s_1(\eta)=\alpha_i$) or $d_f(\eta)=c_f(\overline{\eta})$ (if $s_1(\eta)<\alpha_i$). 

Therefore, if $s_1(\eta)=\alpha_i$, then $d_f(\eta)=f(\alpha_i)$. 

Let us calculate $d_f(\eta)$, when $s_1(\eta)< s_\gamma(\eta)$.
By Fact \ref{4.4},
$$s_7(\eta)\leq s_5(\eta)=s_6(\eta)=s_8(\eta)=sup(rng(\eta_8).$$ 
It is easy to see that $s_2(\eta),s_3(\eta),s_4(\eta)\leq s_5(\eta)$.

We conclude that $s_\gamma(\eta)=s_8(\eta)=sup(rng(\eta_8))$ and $\alpha_i=sup(rng(\eta_8))$.  From Definition \ref{colorconst} (8), $$c_f(\overline{\eta})=f(sup(rng(\eta_8)))=f(\alpha_i).$$

Therefore $d_f(\eta)=f(\alpha_i)$ in both cases ($s_1(\eta)=s_\gamma(\eta)$ and $s_1(\eta)<s_\gamma(\eta)$). By the same argument and using the definition of $F$, we can conclude that $d_g(F(\eta))=g(\alpha_i)$. Finally since $i$ is a limit ordinal with cofinality $\gamma$, $\alpha_i$ is an $\gamma$-limit of $C$. Thus $d_f(\eta)=f(\alpha_i)=g(\alpha_i)=d_g(F(\eta))$ and $F$ is a $\beta$-coloured tree isomorphism.

Now let us prove that if $A^f$ and $A^g$ are isomorphic ordered $\beta$-coloured trees, then $f~ =^\beta_\gamma~ g$.

Let us start by defining  the following function $H_f\in \beta^\kappa$. For every $\alpha\in \kappa$ with cofinality $\gamma$, define $B_\alpha=\{\eta\in A^f\backslash A^f_\alpha\mid dom(\eta)=\gamma~\&~\forall n<\gamma ~(\eta\restriction n\in A^f_\alpha)\}$. Notice that by the construction of $A^f$ and the definition of $A^f_\alpha$, for all $\eta\in B_\alpha$ we have $d_f(\eta)=f(s_\gamma(\eta))=f(\alpha)$. Therefore, the value of $f(\alpha)$ can be obtained from $B_\alpha$ and $d_f$, and we can define the function $H_f\in \beta^\kappa$ as:

$$H_f(\alpha)=\begin{cases} f(\alpha) &\mbox{if } cf(\alpha)=\gamma\\
0 & \mbox{otherwise. } \end{cases}$$
This function can be obtained from the $\kappa$-representation $\{A^f_\alpha\}_{\alpha<\kappa}$ and $d_f$. It is clear that $f~ =^\beta_\gamma~ H_f$.

\begin{claim}
If $A^f$ and $A^g$ are isomorphic ordered $\beta$-coloured trees, then $H_f~ =^\beta_\gamma~ H_g$.
\end{claim}
\begin{proof}
Let $F$ be an ordered $\beta$-coloured tree isomorphism. It is easy to see that $\{F[A^f_{\alpha}]\}_{\alpha<\kappa}$ is a $\kappa$-representation. Define $C=\{\alpha<\kappa\mid F[A^f_\alpha]=A^g_\alpha\}$. Since $F$ is an isomorphism, for all $\alpha\in C$, $H_f(\alpha)=H_g(\alpha)$. Therefore it is enough to show that $C$ is a $\gamma$-club. By the definition of $\kappa$-representation, if $(\alpha_n)_{n<\gamma}$ is a sequence of elements of $C$ cofinal to some $\vartheta$, then $A^g_\vartheta=\bigcup_{n<\gamma}A^g_{\alpha_n}=\bigcup_{n<\gamma}F[A^f_{\alpha_n}]=F[A^f_\vartheta]$. We conclude that $C$ is $\gamma$-closed. 

Let us finish by showing that $C$ is unbounded. Fix an ordinal $\alpha<\kappa$, let us construct a sequence $(\alpha_n)_{n\leq\omega}$ such that $\alpha_\omega\in C$ and $\alpha_\omega>\alpha$. Define $\alpha_0=\alpha$. For every odd $n$, define $\alpha_{n+1}$ to be the least ordinal bigger than $\alpha_n$ such that $F[A^f_{\alpha_n}]\subseteq A^g_{\alpha+1}$.  For every even $n$, define $\alpha_{n+1}$ to be the least ordinal bigger than $\alpha_n$ such that $A^g_{\alpha_n}\subseteq F[A^f_{\alpha+1}]$. Define $\alpha_\omega=\bigcup_{n<\omega}\alpha_n$. Clearly $\bigcup_{i<\omega}F[A^f_{\alpha_{2i}}]=\bigcup_{i<\omega}A^g_{\alpha_{2i+1}}$. We conclude that $\alpha_\omega\in C$
\end{proof}
\end{proof}

\begin{remark}\label{equal_trees}
Same as in the construction of the $\beta$-coloured trees $J_f$, the function$f\in \beta^\kappa$ is only used to define the color function in the construction of $A^f$. 
So if $f,g\in \beta^\kappa$ and $\alpha$ are such that $f\restriction\alpha=g\restriction\alpha$, then $J_f^\alpha=J_g^\alpha$. As a consequence  $f\restriction\alpha=g\restriction\alpha$ implies that $A^f_\alpha= A^g_\alpha$.
\end{remark}

Notice that the only property we used from $I$ to construct the ordered $\beta$-coloured trees was that it is $\kappa$-colorable. Therefore the construction can be done with any $\kappa$-colorable linear order. We will need the other properties of $I$ in the next section, when we construct the Ehrenfeucht-Mostowski models.

\section{Ehrenfeucht-Mostowski models}\label{Section_modelsf}

\subsection{Index models}

We will use ordered $\beta$-coloured trees to construct the models of non-classifiable theories, we will construct Ehrenfeucht-Mostowski models (see \cite{EM}). These models require an skeleton for the construction.

{\bf Notation.} We will use all the objects we have studied so far in the previous sections. Let us recall some notation, before we do the construction of the models.
\begin{itemize}
    \item $\kappa$ is the cardinal associated to the Baire space, it is an uncountable regular cardinal that satisfies $\kappa^{<\kappa}=\kappa$.
    \item $\kappa$ is either an inaccessible cardinal or a successor cardinal. If $\kappa$ is a successor cardinal, then $\lambda$ is the cardinal such that $2^\lambda=\lambda^+=\kappa$.
    \item $\varepsilon<\kappa$ is the cardinal associated to the density, i.e. $\varepsilon$-dense linear orders.
    \item $\theta<\kappa$ is a cardinal of an $\varepsilon$-saturated model of DLO, i.e. it is a cardinal such that there is a model of $DLO$ of size $\theta$ that is $\varepsilon$-dense. For simplicity, let $\theta$ be the least cardinal with such property. 
    \item $\gamma<\kappa$ is the height of the ordered $\beta$-coloured trees.
    \item $\beta\leq \kappa$ is the amount of colors of the ordered $\beta$-coloured trees.
    \item $\mathcal{Q}$ is a model of DLO with cardinality $\theta$ that is $\varepsilon$-dense.
    \item $I$ is the linear order constructed in Section \ref{seccion1}.
\end{itemize}

\noindent {\bf Assumptions.} To use the results of the the previous sections, we will need to make some cardinal assumptions.
\begin{itemize}
    \item If $\kappa=\lambda^+$, then $2^\theta\leq\lambda=\lambda^{<\varepsilon}$.
    \item $\varepsilon$ is regular.
    \item $\gamma$ is regular and satisfies $\forall\alpha<\kappa$, $\alpha^\gamma<\kappa$.
    \item $\varepsilon\leq\gamma$.
    \item $\beta=2$.
\end{itemize}

In this section we will construct a model $\mathcal{M}^f$ for each $f\in 2^\kappa$, and study the isomorphism between these models. The following is the main result of this section.

\begin{thm}\label{main_teo_4}
    Let $\kappa^{<\kappa}=\kappa$ an uncountable cardinal, be either an inaccessible cardinal or successor cardinal such that $2^\lambda=\lambda^+=\kappa$. For all non-classifiable theory $T$ there are regular cardinals $\varepsilon\leq \gamma<\kappa$ satisfying $2^\theta\leq\lambda=\lambda^{<\varepsilon}$ ($\theta=\varepsilon^{<\varepsilon}$) and $\forall\alpha<\kappa$, $\alpha^\gamma<\kappa$; such that the following holds:
    
     For any function $f\in 2^\kappa$, we can construct a model $\mathcal{M}^f$, such that for any $f,g\in 2^\kappa$, $$f\ =^2_\gamma\ g \text{ iff } \mathcal{M}^f \cong \mathcal{M}^g.$$
\end{thm}

It is clear that the construction of the models depends on which kind of non-classifiable theory we are dealing with (e.g. unstable or superstable with the OTOP). In particular, each case requires a different values for $\varepsilon$ and $\gamma$. 

For the constructions we will use different kinds of Ehrenfeucht-Mostowski models. The sets of indiscernibles (below) will allow us to use the types over trees to study the types over models.

\begin{defn}
Let $\Delta$ be a set of formulas. Let $A$ and $\mathcal{M}$ be models, and $X=\{\bar{a}_s\mid s\in A\}$ an indexed set of finite tuples of elements of $\mathcal{M}$. Given $\bar{s}=(s_0,\ldots ,s_n)$ a $n$-tuple of elements of $A$, we denote $\bar{a}_{s_0}^\frown \cdots ^\frown\bar{a}_{s_n}$ by $\bar{a}_{\bar{s}}$ (an $n$-tuple of $X$).

We say that \textit{$X$ is a set of indiscernibles in $\mathcal{M}$ relative to $\Delta$}, if the following holds:

If $\bar{s},\bar{s}'$ are $n$-tuples of elements of $A$ and $tp_{at}(\bar{s},\emptyset,A)=tp_{at}(\bar{s}',\emptyset,A)$, then $$tp_\Delta(\bar{a}_{\bar{s}},\emptyset,\mathcal{M})=tp_\Delta(\bar{a}_{\bar{s}'},\emptyset,\mathcal{M}).$$
\end{defn}

Notice that there are no restrictions over $\Delta$. The existence of a model with a set of indiscernibles relative to an infinitary logic has been studied by Eklof in \cite{Ekl} and Makkai in \cite{Mak}.
The use of infinitary logics will be useful when we construct the models, for some theories we will deal with linear orders definable by a formula in an infinitary logic (Fact \ref{otop1} and Fact \ref{DOP1}).

\begin{defn}[Ehrenfeucht-Mostowski models]\label{EM_def}
    Let $T$ be a $L_{\omega\omega}$-theory of vocabulary $\tau$, $l$ a dense linear order, $\mathcal{M}$ a model of vocabulary $\tau^1$, and $\varphi(\bar{u}, \bar{v})$ a formula with vocabulary $\tau$ in some logic $\mathcal{L}$.

    We say that \textit{$\mathcal{M}$ is an Ehrenfeucht-Mostowski model of $T$ for $l$}, if $\mathcal{M}\models T$, $\tau\subseteq \tau^1$, and there is a natural number $n$ and $n$-tuples of elements $\bar{a}_x\in \mathcal{M}$, $x\in l$, such that the following hold:
    \begin{enumerate}
        \item Every element of $\mathcal{M}$ is of the form $\mu(\bar{a}_{x_1},\ldots , \bar{a}_{x_m})$, where $\mu$ is a $\tau^1$-term and $x_1<\cdots <x_m$.
        \item If $x,y\in l$, then $\mathcal{M}\models \varphi(\bar{a}_x,\bar{a}_y)$ if and only if $x<y$.
        \item If $\psi(\bar{u}_1\ldots , \bar{u}_m)$ is an atomic $\tau^1$-formula, $x_1<\cdots <x_m$ and $y_1<\cdots <y_m$, then $$\mathcal{M}\models \psi(\bar{a}_{x_1},\ldots , \bar{a}_{x_m})\text{ iff }\mathcal{M}\models \psi(\bar{a}_{y_1},\ldots , \bar{a}_{y_m}).$$
    \end{enumerate}
\end{defn}

Suppose $T$ is a theory such that for each dense linear order $l$, $T$ has an Ehrenfeucht-Mostowski model where the order is definable by an $L_{\infty\omega}$-formula. We will only consider linear orders of a fixed set $B$ of linear orders. Let $l_B$ be a dense linear order such that every linear order of $B$ is a submodel of $l_B$. Let $EM_1(l_B)$, $\tau^1$, $\varphi$, $n$, $(\bar{a}_x)_{x\in l_B}$ be such that the conditions of Definition \ref{EM_def} are satisfied for $l_B$, where $\tau^1$ is a Skolemization of $\tau$ and $T^1$ the Skolemization of $T$ by $\tau^1$.

Notice that in the case $\mathcal{L}=L_{\infty\omega}$, if $l\subseteq l_B$ is dense, then $\varphi$ is absolute between a model and its submodels. Thus, if $l\subseteq l_B$ is dense, then we define $EM_1(l)$ as the submodel of $EM_1(l_B)$ generated by $\bar{a}_x$, $x\in l$, and the Skolem functions. Notice that $EM_1(l)$, $\tau^1$, $\varphi$, $n$, $(\bar{a}_x)_{x\in l}$ satisfy the conditions of Definition \ref{EM_def} for $l$.

We call the linear order \textit{$l$ the index model} of $EM_1(l)$. The indexed set \textit{$(\bar{a}_x)_{x\in l}$ is the skeleton of $EM_1(l)$}, and the tuples $\bar{a}_x$, $x\in l$, are the generating elements of $EM_1(l)$. Let us denote $EM(l)=EM_1(l)\restriction \tau$.

Suppose $T$ is a theory such that for each dense linear order $l$, $T$ has an Ehrenfeucht-Mostowski model where the order is definable by an $L_{\infty\omega_1}$-formula, and $B$ contains only $\omega_1$-dense linear orders. Then we can define $EM_1(l)$ and $EM(l)$ for all $l\in B$ as above.

\begin{defn}[Generalized Ehrenfeucht-Mostowski models]\label{EM-models_def}
We say that a function \textit{$\Phi$ is proper for $K_{tr}^\gamma$}, if there is a vocabulary $\tau^1$ and for each $A\in K_{tr}^\gamma$, there is a model $\mathcal{M}_1$ and tuples $\bar{a}_s$, $s\in A$, of elements of $\mathcal{M}_1$ such that the following two hold:
\begin{itemize}
\item Every element of $\mathcal{M}_1$ is an interpretation of some $\mu(\bar{a}_{\bar{s}})$, where $\mu$ is a $\tau^1$-term.
\item $tp_{at}(\bar{a}_{\bar{s}}, \emptyset, \mathcal{M}_1)=\Phi(tp_{at}(\bar{s}, \emptyset, A))$.
\end{itemize}
\end{defn}
Notice that for each $A$, the previous conditions determined $\mathcal{M}_1$ up to isomorphism. We may assume $\mathcal{M}_1$, $\bar{a}_{s}$, $s\in A$, are unique for each $A$. We denote $\mathcal{M}_1$ by $EM^1(A,\Phi)$. 

We call \textit{$A$ the index model of $EM^1(A,\Phi)$}. The indexed set \textit{$(\bar{a}_s)_{s\in A}$ is the skeleton of $EM^1(A,\Phi)$}, and the tuples $\bar{a}_s$, $s\in A$, are the generating elements of $EM^1(A,\Phi)$.

Suppose $T$ is a countable complete theory in a countable vocabulary $\tau$, $\tau^1$ a Skolemization of $\tau$, and $T^1$ the Skolemization of $T$ by $\tau^1$. If there is a proper function $\Phi$ for $K_{tr}^\gamma$, then for every $A\in K_{tr}^\gamma$, we will denote by $EM(A,\Phi)=EM^1(A,\Phi)\restriction \tau$. 

We call $EM_1(l)$ and $EM^1(A,\Phi)$ \textit{Ehrenfeucht-Mostowski models}. As it was mentioned before, the sets of indiscernibles play an important role in the Ehrenfeucht-Mostowski models. For more on indiscernibles and types see \cite{Sco} and \cite{TT}, where indiscernible trees and generalized indiscernible sets are studied.

\begin{defn} Let $\varepsilon,\gamma<\kappa$ be regular cardinals.
\begin{itemize} 
\item $A\in K^\gamma_{tr}$ of size at most $\kappa$, is \textit{locally $(\kappa,\varepsilon)$-nice} if for every $\eta\in A\backslash P^A_\gamma$, $(Suc_A(\eta),<)$ is $(\kappa,\varepsilon)$-nice, $Suc_A(\eta)$ is infinite, and there is $\xi\in P_\gamma^A$ such that $\eta\prec\xi$.
\item $A\in K^\gamma_{tr}$ is \textit{$(<\kappa)$-stable} if for every $B\subseteq A$ of size smaller than $\kappa$, $$\kappa>|\{tp_{bs}(a,B,A)\mid a\text{  a finite tuple of } A\}|.$$
\end{itemize}
\end{defn}

By Theorem \ref{main_order_prop}, we know that there is a linear order that is $\varepsilon$-dense, $(<\kappa)$-stable, $(\kappa, \varepsilon)$-nice, and $\kappa$-colorable.

Even though  in this section we are working under the assumption $\beta=2$, the following two results are true for any $\beta<\kappa$.

\begin{lemma}\label{nonzero}
For every $f\in \beta^\kappa$, $\rho<\beta$, and $\eta\in (J_f)_{<\gamma}$, there is $\xi\in (J_f)_\gamma$ such that $\eta<\xi$ and $c_f(\xi)=\rho$.
\end{lemma}

\begin{proof}

Let $f\in \beta^\kappa$, such that $\eta\in (J_g)_{<\gamma}$, and $r=dom(\eta)$. By The construction of $J_f$, there is $\eta'\in (J_g)_{<\gamma}$ such that $\eta\prec\eta'$ and $n=dom(\eta')$ is a successor ordinal.

Let us construct $\xi$, such that $\eta\prec\xi$ and $c_f(\xi)=\rho$.
\begin{itemize}
\item $\xi\restriction n=\eta'$.
\item if $n\leq m<\gamma$, 
\begin{itemize}
\item $\xi_1(m)=\xi_1(n-1)+1$.
\item $\xi_2(m)=\xi_3(n-1)+\xi_4(n-1)$.
\item $\xi_3(m)=\xi_2(n)+\gamma$.
\item let $\varrho$ and $\zeta$ be such that $dom(\zeta)=[n,\gamma)$, $\zeta\in Z^{\xi_2(n),\xi_3(n)}_{\varrho}$ with $c(\zeta)=\rho$. Such $\varrho$ and $\zeta$ exist by Definition \ref{constants_trees}.
\item $\xi_4(m)=\varrho$.
\item  $\xi_5\restriction_{[n,\gamma)}=\zeta$.
\end{itemize}
\end{itemize}

By the way we defined $\xi$, we know that $\xi\in J_f$ and $\eta\prec\xi$. By the item (8) (a) on the construction of $J_f$, we know that $c_f(\xi)=c(\xi_5\restriction_{[n,\gamma)})=\rho$.

\end{proof}

\begin{defn}\label{Def_final_trees_and_repr}
For every $f\in 2^\kappa$, define the tree $A_f\subseteq A^f$ by: $x\in A_f$ if and only if $x$ is not a leaf of $A^f$ or $x$ is a  leaf such that $d_f(x)=1$.
For every $f\in 2^\kappa$, define the $\kappa$-representation $\mathbb{A}_f=\langle (A_f)_\alpha\mid\alpha<\kappa\rangle$ of $A_f$ by $$(A_f)_\alpha=\{\eta\in A_f\mid \eta\in A^f_\alpha\},$$ where $\mathbb{A}^f=\langle A^f_\alpha\mid\alpha<\kappa\rangle$ is the $\kappa$-representation introduced in the proof of Theorem \ref{Afcong}.
\end{defn}

\begin{thm}\label{the_final_trees}
For any $f\in 2^\kappa$, $A_f$ is a locally $(\kappa, \varepsilon)$-nice and $(<\kappa)$-stable ordered tree and satisfies: For all $f,g\in 2^\kappa$,
$$f\ =^2_\gamma\ g\ \Leftrightarrow A_f\cong A_g.$$
\end{thm}

\begin{proof}
  Recall that $\gamma$ is such that for all $\alpha<\kappa$, $\alpha^\gamma<\kappa$, from Fact \ref{Afcong}, we know that for all $f,g\in 2^\kappa$,
$$f\ =^2_\gamma\ g\ \Leftrightarrow A_f\cong A_g.$$
From the way that $A_f$ was constructed, we know that for all $\eta\in A_f\backslash P^{A_f}_\gamma$, $(Suc_{A_f}(\eta),<)$ is $(\kappa, \varepsilon)$-nice and $Suc_{A_f}(\eta)$ is infinite. From Lemma \ref{nonzero} and the way $A_f$ was constructed, for all $\eta\in A_f\backslash P^{A_f}_\gamma$ there is $\xi\in P_\gamma^A$ such that $\eta\prec\xi$. Since the branches of the trees $A_f$ have length at most $\gamma+1$ and $I$ is $(<\kappa)$-stable, then the trees $A_f$ are $(<\kappa)$-stable.
\end{proof}

The trees $A_f$ are not the index models, we will use these trees to construct the index models (depending on the theory). Clearly, the properties of the trees $A_f$ will be important. A very important property that we will need is \textit{homogenicity}, it is related to indiscernible sets. Let $\upsilon\leq\kappa$ be a regular cardinal, a tree $A$ is \textit{$\upsilon$-homogeneous with respect to quantifier free formulas} if the following holds:

\textit{For every partial isomorphisms $F: X \rightarrow A$, where $|X|<\varepsilon$ is a subset of $A$, and $a$ in $A$; there is a partial isomorphisms $g: X \cup {a} \rightarrow A$ that extends $F$.}

See \cite{Kei} Chapter 19, for more on homogeneous models.

\begin{fact}\label{DOP2}
For all $f\in 2^\kappa$, $A_f$ is $\varepsilon$-homogeneous with respect to quantifier free formulas.

\end{fact}

\begin{proof}
    Since $I$ is $\varepsilon$-dense and by the way $A_f$ was constructed, we know that for all $f\in 2^\kappa$, $A_f$ is closed and for all $\eta\in A_f$, $\neg P_\gamma(\eta)$, $<\restriction Suc_{A_f}(\eta)$ is $\varepsilon$-dense. Therefore, for all $f\in 2^\kappa$, $X\subseteq A_f$ with $|X|<\varepsilon$, partial isomorphisms $F: X \rightarrow A_f$, and $a\in A_f\backslash X$; there is $b\in A_f\backslash F[X]$ such that $b\models F(tp_{qf}(a,X,A_f))$.
\end{proof}

We have finished the general preparations. The trees $A_f$ will be used to construct the index models, we can proceed to construct the models (depending on the theory).

\subsection{Stable unsuperstable theories}

This case was studied in \cite{Mor21} under the assumption $\gamma=\varepsilon=\theta=\omega$, thus $A_f\in K^\omega_{tr}$. In particular, in [\cite{Mor21}, Definition 4.5] the models were constructed by $\mathcal{M}^f=EM(A_f,\Phi)$. Notice that this construction requires the existence of a proper function $\Phi$ for $K^\omega_{tr}$, this can be found in \cite[Theorem 1.3]{Sh}, with proof in \cite[Chapter VII 3]{Sh90}. 

Even more, in \cite{Mor21} the following was proved.

\begin{fact}[{Moreno, \cite[Lemma 4.8]{Mor21}}]\label{unsuperstable_iso_case}
If $T$ is a countable complete unsuperstable theory over a countable vocabulary, then for all $f,g\in 2^\kappa$, $f\ =^2_\omega\ g$ if and only if EM$(A_f,\Phi)$ and EM$(A_g,\Phi)$ are isomorphic.
\end{fact}

\subsection{Superstable theories with the OTOP}

For the OTOP case we will use $\theta=\varepsilon=\omega$, and $\mathcal{Q}=\mathbb{Q}$.

The following facts follow from the proof of \cite[Theorem 2.5]{Shotop} and the fact that a theory $T$ with the OTOP is weakly unstable as an $L_{\omega_1 \omega}$-theory, \cite[Definition 6.4 and Definition 6.5]{HT}; see \cite[Theorem 6.6]{HT}.
\begin{fact}\label{otop1}
    Suppose $T$ is a theory with the OTOP in a countable vocabulary $\tau$. Then for each dense linear order $l$ we can find a model $\mathcal{N}$ of a countable vocabulary $\tau^1\supseteq \tau$ such that $\mathcal{N}$ is an Ehrenfeucht-Mostowski model of $T$ for $l$, where the order is definable by an $L_{\infty \omega}$-formula.
\end{fact}

\begin{fact}\label{otop2}
Let $\mathcal{M}$ be an Ehrenfeucht-Mostowski model of $T$ for $l$, and $A=(\bar{a}_x)_{x\in l}$.
If $l$ is dense, then $A$ is a set of indiscernibles relative to $L_{\infty\omega}$. 
\end{fact}

\begin{defn}
    For every $f\in 2^\kappa$ let us define the order $K^O(f)$ by:
    \begin{enumerate}
        \item [I.] $dom \ K^O(f)=(dom\ A_f\times \{0\})\cup \{(\eta,1)\mid\eta\in A_f\ \&\ A_f\not\models P_\gamma(\eta)\}$.
        \item [II.] For all $\eta\in A_f$ such that $A_f\not\models P_\gamma(\eta)$, $(\eta,0)<_{K^O(f)}(\eta,1)$.
        \item [III.] If $\eta,\xi\in A_f$ such that $A_f\not\models P_\gamma(\xi)\ \vee\ P_\gamma(\eta)$, then $\eta\prec \xi$ if and only if $$(\eta,0)<_{K^O(f)}(\xi,0)<_{K^O(f)}(\xi,1)<_{K^O(f)}(\eta,1).$$ 
        \item [IV.] If $\eta,\xi\in A_f$ such that $A_f\models P_\gamma(\xi)$ and $A_f\not\models P_\gamma(\eta)$, then $\eta\prec \xi$ if and only if $$(\eta,0)<_{K^O(f)}(\xi,0)<_{K^O(f)}(\eta,1).$$
        \item [V.] If $\eta,\xi\in A_f$, then $\eta< \xi$ if and only if $(\eta,1)<_{K^O(f)}(\xi,0)$.
    \end{enumerate}
\end{defn}

Notice that $K^O(f)$ is dense.

\begin{lemma}\label{construction_otop}
Suppose $T$ is superstable with the OTOP in a countable relational vocabulary $\tau$. Let $\tau^1$ be a Skolemization of $\tau$, and $T^1$ be a complete theory in $\tau^1$ extending $T$ and with Skolem-functions in $\tau$. 
Then for every $f\in 2^\kappa$ there is $\mathcal{M}_1^f\models T^1$ with the following properties.

\begin{enumerate}
    \item There is a map $\mathcal{H}: A_f\rightarrow (dom \ \mathcal{M}_1^f)^n$ for some $n<\omega$, $\eta\mapsto a_\eta$, such that $\mathcal{M}_1^f$ is the Skolem hull of $\{a_\eta\mid \eta\in A_f\}$. Let us denote $\{a_\eta\mid \eta\in A_f\}$ by $Sk(\mathcal{M}_1^f)$.
    \item $\mathcal{M}^f=\mathcal{M}_1^f\restriction \tau$ is a model of $T$.
    \item $Sk(\mathcal{M}_1^f)$ is indiscernible in $\mathcal{M}_1^f$ relative to $L_{\infty \omega}$.
    \item There is a formula $\varphi\in L_{\infty \omega}(\tau)$ such that for all $\eta,\nu\in A_f$ and $m<\gamma$, if $A_f\models P_m(\eta) \wedge P_\gamma(\nu)$, then $\mathcal{M}^f\models \varphi (a_\nu,a_\eta)$ if and only if $A_f\models \eta\prec \nu$.
\end{enumerate}
\end{lemma}

\begin{proof}
By Fact \ref{otop1}, and Fact \ref{otop2}, there is an Ehrenfeucht-Mostowski model $\mathcal{M}^f_1$ for the linear order $K^O(f)$ where the order is definable by a $L_{\infty \omega}$-formula $\psi$. Suppose $\bar{\eta}=(\eta_0,\ldots , \eta_n)$ and $\bar{\xi}=(\xi_0,\ldots , \xi_n)$ are sequences in $A_f$ that have the same quantifier free type. The sequences $$\langle (\eta_0,0), (\eta_0,1), (\eta_1,0),\ldots, (\eta_n,0), (\eta_n,1) \rangle$$ and $$\langle (\xi_0,0), (\xi_0,1), (\xi_1,0),\ldots, (\xi_n,0), (\xi_n,1) \rangle$$ have the same quantifier free type in $K^O(f)$. Let the skeleton of $\mathcal{M}^f_1$ be $\{a_x\mid x\in K^O(f)\}$. 

Let us define the $A_f$-skeleton of $\mathcal{M}^f_1$  to be the set $$\{{a_{(\eta,0)}} ^\frown a_{(\eta,1)}\mid \eta\in A_f\ \&\ A_f\not\models P_\gamma(\eta)\}\cup \{{a_{(\eta,0)}} ^\frown a_{(\eta,0)}\mid \eta\in A_f\ \&\ A_f\models P_\gamma(\eta)\}.$$

For all $\eta\in A_f$ such that $A_f\not\models P_\gamma(\eta)$, let us denote $b_\eta={a_{(\eta,0)}}^\frown a_{(\eta,1)}$. For all $\eta\in A_f$ such that $A_f\models P_\gamma(\eta)$, let us denote $b_\eta={a_{(\eta,0)}}^\frown a_{(\eta,0)}$.
It is clear that (1) and (2) follow from the construction. Fact \ref{otop2} implies (3). Let us show that (4) holds. From the construction, $K^O(f)$ is definable in $\mathcal{M}^f$ by the formula $\psi(\bar{u},\bar{c})$, i.e. $\mathcal{M}^f\models \psi(a_x,a_y)$ if and only if $K^O(f)\models x<y$. Let $\varphi(x_0,x_1,y_0,y_1)$ be the formula $$\psi(x_0,y_0)\wedge \psi(y_1,x_1).$$
Therefore, for all $\eta,\nu\in A_f$ and $m<\gamma$ such that $A_f\models P_m(\eta) \wedge P_\gamma(\nu)$, $$\varphi((a_\eta,0), (a_\eta,1),(a_\nu,0), (a_\nu,1))$$ holds in $\mathcal{M}^f$ if and only if $A_f\models\eta\prec\nu$.
 \end{proof}

\subsection{Unstable theories}

For the unstable case we will use $\theta=\varepsilon=\omega$, and $\mathcal{Q}=\mathbb{Q}$. Thus the construction of the models follow as in the OTOP case.

\begin{fact}[{Shelah, \cite{Shunst}; Hyttinen-Tuuri, \cite[Lemma 4.7]{HT}}]\label{unst1}
    Suppose $T$ is a complete unstable theory in a countable vocabulary $\tau$. Then for each linear order $l$ we can find a model $\mathcal{N}$ of a countable vocabulary $\tau^1\supseteq \tau$ such that $\mathcal{N}$ is an Ehrenfeucht-Mostowski model of $T$ for $l$, where the order is definable by a first order formula $\psi$. Even more, $A=(\bar{a}_x)_{x\in l}$ is a set of indiscernibles relative to $L_{\omega \omega}$.
\end{fact}

Since we do not need a dense linear order, we can define $(\eta,1)$ also for elements at the level $\gamma$.
\begin{defn}
    For every $f\in 2^\kappa$ let us define the order $K^U(f)$ by:
    \begin{enumerate}
        \item [I.] $dom \ K^U(f)=(dom\ A_f\times \{0\})\cup (dom\ A_f\times \{1\})$.
        \item [II.] For all $\eta\in A_f$, $(\eta,0)<_{K^U(f)}(\eta,1)$.
        \item [III.] If $\eta,\xi\in A_f$, then $\eta\prec \xi$ if and only if $$(\eta,0)<_{K^U(f)}(\xi,0)<_{K^U(f)}(\xi,1)<_{K^U(f)}(\eta,1).$$ 
        \item [IV.] If $\eta,\xi\in A_f$, then $\eta< \xi$ if and only if $(\eta,1)<_{K^U(f)}(\xi,0)$.
    \end{enumerate}
\end{defn}

\begin{lemma}\label{construction_unst}
Suppose $T$ is a complete unstable theory in a countable vocabulary $\tau$. Let $\tau^1$ be a Skolemization of $\tau$, and $T^1$ be a complete theory in $\tau^1$ extending $T$ and with Skolem-functions in $\tau$. 
Then for every $f\in 2^\kappa$ there is $\mathcal{M}_1^f\models T^1$ with the following properties.

\begin{enumerate}
    \item There is a map $\mathcal{H}: A_f\rightarrow (dom \ \mathcal{M}_1^f)^n$ for some $n<\omega$, $\eta\mapsto a_\eta$, such that $\mathcal{M}_1^f$ is the Skolem hull of $\{a_\eta\mid \eta\in A_f\}$. Let us denote $\{a_\eta\mid \eta\in A_f\}$ by $Sk(\mathcal{M}_1^f)$.
    \item $\mathcal{M}^f=\mathcal{M}_1^f\restriction \tau$ is a model of $T$.
    \item $Sk(\mathcal{M}_1^f)$ is indiscernible in $\mathcal{M}_1^f$ relative to $L_{\omega \omega}$.
    \item There is a formula $\varphi\in L_{\omega \omega}(\tau)$ such that for all $\eta,\nu\in A_f$ and $m<\gamma$, if $A_f\models P_m(\eta) \wedge P_\gamma(\nu)$, then $\mathcal{M}^f\models \varphi (a_\nu,a_\eta)$ if and only if $A_f\models \eta\prec \nu$.
\end{enumerate}
\end{lemma}

\begin{proof}
Similar to Lemma \ref{construction_otop}. Instead of using Fact \ref{otop1} and Fact \ref{otop2},  use Fact \ref{unst1}.
 \end{proof}

\subsection{Superstable theories with the DOP}

For the DOP case we will use $\varepsilon=\omega_1$, thus $\theta=2^{\omega}=\mathfrak c$. Let us denote by $SH(X)$ the Skolem-hull of $X$, i.e. $\{\mu(a)\mid a\in X, \mu\text{ an }\tau^1 \text{-term}\}$. 

\begin{fact}[{Shelah, \cite[Fact 2.5B]{Shdop}; Hyttinen-Tuuri, \cite[Theorem 6.1]{HT}}]\label{DOP1}
    Suppose $T$ is a countable superstable theory with the DOP in a countable vocabulary $\tau$. Then there exists a vocabulary $\tau^1\supseteq \tau$, $|\tau^1|=\omega_1$, such that for every linear order $l$ we can find a $\tau^1$-model $\mathcal{N}$ which is an Ehrenfeucht-Mostowski model of $T$ for $l$, where the order is definable by an $L_{\omega_1 \omega_1}$-formula.
\end{fact}

\begin{defn}
    For every $f\in 2^\kappa$ let us define the order $K^D(f)$ by:
    \begin{enumerate}
        \item [I.] $dom \ K^D(f)=(dom\ A_f\times \{0\})\cup (dom\ A_f\times \{1\})$.
        \item [II.] For all $\eta\in A_f$, $(\eta,0)<_{K^D(f)}(\eta,1)$.
        \item [III.] If $\eta,\xi\in A_f$, then $\eta\prec \xi$ if and only if $$(\eta,0)<_{K^D(f)}(\xi,0)<_{K^D(f)}(\xi,1)<_{K^D(f)}(\eta,1).$$ 
        \item [IV.] If $\eta,\xi\in A_f$, then $\eta< \xi$ if and only if $(\eta,1)<_{K^D(f)}(\xi,0)$.
    \end{enumerate}
\end{defn}

\begin{lemma}\label{construction_dop}
Suppose $T$ is superstable with the DOP in a countable relational vocabulary $\tau$. Let $ \tau^1$ be a Skolemization of $\tau$, and $T^1$ be a complete theory in $\tau^1$ extending $T$ and with Skolem-functions in $\tau$. 
Then for every $f\in 2^\kappa$ there is $\mathcal{M}_1^f\models T^1$ with the following properties.

\begin{enumerate}
    \item There is a map $\mathcal{H}: A_f\rightarrow (dom \ \mathcal{M}_1^f)^n$ for some $n<\omega$, $\eta\mapsto a_\eta$, such that $\mathcal{M}_1^f$ is the Skolem hull of $\{a_\eta\mid \eta\in A_f\}$. Let us denote $\{a_\eta\mid \eta\in A_f\}$ by $Sk(\mathcal{M}_1^f)$.
    \item $\mathcal{M}^f=\mathcal{M}_1^f\restriction \tau$ is a model of $T$.
    \item $Sk(\mathcal{M}_1^f)$ is indiscernible in $\mathcal{M}_1^f$ relative to $L_{\omega_1 \omega_1}$.
    \item There is a formula $\varphi\in L_{\omega_1 \omega_1}(\tau)$ such that for all $\eta,\nu\in A_f$ and $m<\gamma$, if $A_f\models P_m(\eta) \wedge P_\gamma(\nu)$, then $\mathcal{M}^f\models \varphi (a_\nu,a_\eta)$ if and only if $A_f\models \eta\prec \nu$.
\end{enumerate}
\end{lemma}

\begin{proof}
Similar to Lemma \ref{construction_otop}. Instead of use Fact \ref{otop1} and Fact \ref{otop2},  use Fact \ref{DOP1}. Only property (3) requires a proof. 
Let $\bar s$ and $\bar t$ be $n$-tuples of elements of $A^f$ such that $tp_{at}(\bar{s},\emptyset,A^f)=tp_{at}(\bar{t},\emptyset,A^f)$. Since $tp_{at}(\bar{s},\emptyset,A^f)=tp_{at}(\bar{t},\emptyset,A^f)$, there is a partial isomorphism $F: \mathcal{S}\rightarrow A^f$ such that $F[\mathcal{S}]=\mathcal{T}$, where $\mathcal{S}$ and $\mathcal{T}$ are the sets of the elements of $\bar s$ and $\bar t$, respectively.

\begin{claim}
For all $\varphi(x)\in L_{\omega_1 \omega_1}$ with $n$ free variables, $\mathcal{M}_1^f\models \varphi(\bar{a}_{\bar s})$ holds if and only if $\mathcal{M}_1^f\models \varphi(\bar{a}_{\bar t})$
.\end{claim}
\begin{proof}
    We will proceed by induction on the complexity of $\varphi(x)$. The non-trivial cases are the quantifier cases. 

    {\bf Case $\exists$.} Let $\varphi$ be of the form $\exists(u)\psi(x)$, where $\psi$ satisfies the property of the claim. Suppose $\mathcal{M}_1^f\models \varphi(\bar{a}_{\bar s})$. There is a countable tuple $\bar \nu\in \mathcal{M}_1^f$ such that $\mathcal{M}_1^f\models \psi[\bar{a}_{\bar s}, \bar \nu]$. 
    By the construction of $\mathcal{M}_1^f$, there is a countable set $B\subseteq A^f$ such that for all $i<\omega$ there is $\mu_i$, $m(i)<\omega$, and a tuple $(z^i_0,\ldots, z^i_{m(i)})$ of elements of $S\cup B$ such that $\nu_i=\mu_i(\bar{a}_{z^i_0},\ldots, \bar{a}_{z^i_{m(i)}})$. Since $A^f$ is $\omega_1$-homogeneous, there is a partial isomorphism $\mathcal{F}:S\cup B\rightarrow A^f$ that extends $F$. 
    By the way $\mathcal{M}_1^f$ was constructed, there is a partial isomorphism $\bar{\mathcal{F}}:SH(\{a_x\mid x\in s\cup B\})\rightarrow \mathcal{M}_1^f$ such that $\bar{\mathcal{F}}(a_x)=a_{\mathcal{F}(x)}$. For all $i<\omega$, let $\varpi_i=\bar{\mathcal{F}}(\nu_i)$, $\bar \varpi=(\varpi_0,\ldots)$. 
    So $\mathcal{M}_1^f\models \psi(\bar{a}_{\bar t}, \bar \varpi)$, we conclude that $\mathcal{M}_1^f\models \varphi(\bar{a}_{\bar t})$.

    The case $\mathcal{M}^f\models \varphi(\bar{a}_{\bar t})$ is similar.

    {\bf Case $\forall$.} Let $\varphi$ be of the form $\forall(u)\psi(x)$, where $\psi$ satisfies the property of the claim. Let us suppose, towards contradiction, that $\mathcal{M}_1^f\models \varphi(\bar{a}_{\bar s})$ and $\mathcal{M}_1^f\not\models \varphi(\bar{a}_{\bar t})$. 
    Therefore, there is a countable tuple $\bar \varpi \in \mathcal{M}_1^f$ such that $\mathcal{M}_1^f\models \neg\psi[\bar{a}_{\bar t}, \bar \varpi]$. 
    By a similar argument as in the existential case, we can conclude that there is $\bar \nu\in \mathcal{M}_1^f$ such that $\mathcal{M}_1^f\models \neg\psi[\bar{a}_{\bar s}, \bar \nu]$. On the other hand, $\mathcal{M}_1^f\models \varphi(\bar{a}_{\bar s})$ implies $\mathcal{M}_1^f\models \psi[\bar{a}_{\bar s}, \bar \nu]$, a contradiction.
    
\end{proof}

Property (3) follows from the claim.

 \end{proof}

\subsection{The isomorphism theorem}

\begin{defn}
	Let $A\in K^\gamma_{tr}$, $B\subseteq A$ and $\eta\in A$.
	\begin{itemize}
	\item We say that $B$ is \textit{downward closed} if for all $\eta\in B$, $\eta\restriction m\in B$ if $m<lg(\eta)$.
	\item Let $d(\eta)=(\eta\restriction \alpha)_{\alpha\leq lg(\eta)}$.
	\end{itemize}
\end{defn}

\begin{fact}[{Hyttinen-Tuuri, \cite[Lemma 8.11]{HT}}]\label{closed_lema}
	Let $A\in K^\gamma_{tr}$ and $B\subseteq A$, $B$ downward closed. Let $\bar{\eta}_1=(\eta^0_1,\ldots,\eta^n_1)$ and $\bar{\eta}_2=(\eta^0_2,\ldots,\eta^n_2)$ be sequences of elements of $A$. Suppose:
	\begin{itemize}
	\item $tp(\bar{\eta}_1,\emptyset,A)=tp(\bar{\eta}_2,\emptyset,A)$;
	\item for all $i\leq n$, $tp_{bs}(d(\eta_1^i),B,(A,\prec,<))=tp_{bs}(d(\eta_2^i),B,(A,\prec,<))$.
	\end{itemize}
	Then $tp_{bs}(\bar{\eta}_1,B,A)=tp_{bs}(\bar{\eta}_2,B,A)$.
\end{fact}

\begin{thm}\label{models_iso}
Suppose $T$ is a non-classifiable first order theory in a countable relational vocabulary $\tau$.
\begin{enumerate}
    \item If $T$ is stable unsuperstable, $\gamma=\omega$, and for all $\alpha<\kappa$, $\alpha^\omega<\kappa$, then for all $f,g\in 2^\kappa$ $$f\ =^2_\gamma\ g \text{ iff } \mathcal{M}^f \cong \mathcal{M}^g.$$
    \item If $T$ is unstable or superstable with the OTOP, $\omega\leq \gamma<\kappa$ is such that for all $\alpha<\kappa$, $\alpha^\gamma<\kappa$, then for all $f,g\in 2^\kappa$ $$f\ =^2_\gamma\ g \text{ iff } \mathcal{M}^f \cong \mathcal{M}^g.$$
    \item If $T$ is superstable with the DOP, $\kappa$ is inaccessible or $\kappa=\lambda^+$ and $2^\mathfrak{c}\leq \lambda$, and $\omega_1\leq \gamma<\kappa$ is such that for all $\alpha<\kappa$, $\alpha^\gamma<\kappa$, then for all $f,g\in 2^\kappa$, $$f\ =^2_\gamma\ g \text{ iff } \mathcal{M}^f \cong \mathcal{M}^g.$$
\end{enumerate}

\end{thm}

\begin{proof}
The case $T$ stable unsuperstable, follows from Fact \ref{unsuperstable_iso_case}. Let us prove the case when $T$ is either unstable or superstable. Recall the assumptions of each case
\begin{itemize}
    \item If $T$ is unstable, then $\varepsilon=\omega$ and $L=L_{\omega \omega}$;
    \item If $T$ is superstable with the OTOP, then $\varepsilon=\omega$ and $L=L_{\infty \omega}$;
    \item If $T$ is superstable with the DOP, then $\varepsilon=\omega_1$ and $L=L_{\omega_1 \omega_1}$;
\end{itemize}

By Theorem \ref{the_final_trees} $f\ =^2_\gamma\ g$ holds if and only if $A_f\cong A_g$. It is enough to show that $$A_f\cong A_g \text{ iff }  \mathcal{M}^f \cong \mathcal{M}^g.$$

$\Rightarrow)$ It is clear that if $A_f\cong A_g$, then $ \mathcal{M}^f \cong \mathcal{M}^g$.

$\Leftarrow)$ Recall the filtration $\mathbb{A}_f$ from Definition \ref{Def_final_trees_and_repr}. For all $f\in 2^\kappa$, $\eta\in A_f$, $A_f\not\models P_\gamma (\eta)$, and $\alpha\leq\kappa$ let $$B^f(\eta,\alpha)=Suc_{A_f}(\eta)\cap (A_f)_\alpha.$$ 
It is clear that $\langle B^f(\eta,\alpha) \mid \alpha<\kappa\rangle$ is a $\kappa$-representation of $Suc_{A_f}(\eta)$. 
By Theorem \ref{the_final_trees} $A_f$ is $(\kappa, \varepsilon)$-nice, in particular $Suc_{A_f}(\eta)$ is isomorphic to $I$. Since any two representations coincide in a club, for any $\eta\in A_f$ there is a club $C_\eta$ such that for all $\delta\in C_\eta$ with $cf(\delta)\ge\varepsilon$ and $\nu\in Suc_{A_f}(\eta)$ there is $\beta<\delta$ such that $$\forall\sigma\in B^f(\eta,\delta)\ [\sigma>\nu \Rightarrow \exists\sigma '\in B^f(\eta,\beta)\ (\sigma\ge\sigma'\ge\nu)].$$
Let $$\bar{C}^f=\{\delta<\kappa\mid cf(\delta)\ge\varepsilon\text{ and for all }\eta\in (A_f)_\delta,\  \delta\in C_\eta\}$$ and $C^f$ be $\bar{C}^f$ closed under $\alpha$-limits for $\alpha<\varepsilon$.
Notice that $C^f$ is a club that satisfies that for all $\delta\in C^f$ with $cf(\delta)\ge\varepsilon$, $\eta\in A_f$, $A_f\not\models P_\gamma (\eta)$, and $\nu\in Suc_{A_f}(\eta)$,
$$\forall\sigma\in B^f(\eta,\delta)\ [\sigma>\nu \Rightarrow \exists\sigma '\in B^f(\eta,\beta)\ (\sigma\ge\sigma'\ge\nu)].$$

Let us define $C^g$ in a similar way. Let $C_0=C^f\cap C^g$.

Assume, for sake of contradiction, that $f,g\in 2^\kappa$ are such that $f\ \neq^2_\gamma\ g$, and $\mathcal{M}^f $ are $\mathcal{M}^g$ isomorphic. 
Since $f\ \neq^2_\gamma\ g$, there is a stationary set $S\subseteq S^\kappa_\gamma$ such that for all $\alpha\in S$, $f(\alpha)\neq g(\alpha)$. Assume, without loss of generality, that $S\subseteq S^\kappa_\gamma$ is a stationary set such that for all $\alpha\in S$, $f(\alpha)=1$ and $g(\alpha)=0$. Let $F$ be an isomorphism from $\mathcal{M}^f $ to $\mathcal{M}^g$.

Let us denote by $\bar a_\eta$ and $\bar b_\xi$ the elements of $Sk(\mathcal{M}_1^f)$ and $Sk(\mathcal{M}_1^g)$. For a sequence $\bar{a}=(a^0,\ldots, a^m)$ from $\mathcal{M}^f$ we denote $F(\bar{a})=(F(a^0),\ldots,F(a^m))$ and for a sequence $\bar{v}=(v^0,\ldots,v^0)$ from $A_g$ we denote $\bar{b}_{\bar{v}}=\bar b_{v^0}^\frown\cdots^\frown \bar b_{v_m}$. For each $\eta\in A_f$ let $$F(a_\eta)=(\mu_\eta^0(\bar b_{\bar{v}_\eta}),\ldots,\mu_\eta^m(\bar b_{\bar{v}_\eta}))=\bar{\mu}_\eta(\bar{b}_{\bar{v}_\eta}),$$ where $m=lg(\bar a_\eta)-1$, $\mu_\eta^i$ are $\tau^1$-terms and $\bar{v}_\eta$ is a finite sequence of elements of $A_g$.

Let $\bar{v}_\eta=(v_\eta^i)_{i<lg(\bar{v}_\eta)}$. Let 
\begin{itemize}
\item $C_1=\{\delta\in C_0\mid \forall \eta\in A_f\  (\eta\in (A_f)_\delta \text{ implies }\bar{v}_\eta\subseteq (A_g)_\delta)\}$;
\item $C'_2=\{\delta\in C_1\mid \forall\alpha<\delta\ \forall\eta\in (A_f)_\delta\ \forall\sigma_1\in B^f(\eta,\kappa)\ \exists\sigma_2\in B^f(\eta,\delta)$ $$[\bar{v}_{\sigma_1},\bar{v}_{\sigma_2}\text{ realizes the same atomic type over }(A_g)_\alpha\text{ and }\bar{\mu}_{\sigma_1}=\bar{\mu}_{\sigma_2}] \}$$
\item $C_2=\{\delta\in C'_2\mid cf(\delta)\ge\gamma\}$
\item $C=\{\delta\in C_2\mid \delta\in C_2\ \&\ \delta\text{ is a limit point of }C_2\}.$ 
\end{itemize}

It is clear that $C_1$ is a club. Since $A_g$  is $(<\kappa)$-stable ordered tree, there are less than $\kappa$ possible $bs$-types of $\bar{v}_{\sigma_1}$ over $(A_g)_\alpha$, and since $|\tau^1|<\kappa$, there are less than $\kappa$ possible terms $\bar{\mu}_{\sigma_1}$, so $C'_2$ is a club. Thus $C'_2\cap S\neq\emptyset$. Since $S\subseteq S^\kappa_\gamma$, $C_2\cap S\neq\emptyset$. So $S\cap C\neq \emptyset$ and it is a subset of $S^\kappa_\gamma$.

Let $\delta\in S\cap C$, so there is $\eta\in A_f$, such that:
\begin{enumerate}
\item  $A_f\models P_\gamma(\eta)$.
\item For all $n<\gamma$, $\eta\restriction n\in (A_f)_\delta$.
\item For all $\alpha<\delta$, there is $m<\gamma$ such that $\eta\restriction m\notin (A_f)_\alpha$.
\end{enumerate}
For each $n<lg (\bar{v}_\eta)$ there is $\alpha_n\in C_2\cap\delta$ such that one of the following holds
\begin{enumerate}
\item[I.] $v_\eta^n\in (A_g)_{\alpha_n}$.
\item[II.] There is $m_n<lg(v_\eta^n)$ such that for $w^0=v_\eta^n\restriction m_n$ and $w^1=v_\eta^n\restriction (m_n +1)$ the following hold
\begin{itemize}
\item $w^0\in (A_g)_{\alpha_n}$ and $w^1\notin (A_g)_{\delta}$.
\item  $\forall\sigma\in B^g(w^0,\delta)\ [\sigma>w^1 \Rightarrow \exists\sigma '\in B^g(w^0,\alpha_n)\ (\sigma\ge\sigma'\ge w^1)]$.
\end{itemize}
\end{enumerate}

By the construction of $A_g$ and since $g(\delta)=0$, there is no $\nu\notin (A_g)_\delta$ with $A_g\models P_\gamma(\nu)$ such that for all $n<\gamma$, $\nu\restriction n\in (A_g)_\delta$.

Let $\alpha=max\{\alpha_n\mid n<lg(\bar{v}_\eta)\}$, clearly $\alpha\in C_2\cap\delta$. Since $\delta\in C\cap S$, there is $\chi\in C_2$  such that $\alpha<\chi<\delta$.

Notice that $cf(\chi)\ge\gamma$, therefore for all $n<\gamma$ limit, either $\eta\restriction n\in (A_f)_\chi$ or there is $n'<n$ such that $\eta\restriction n'\in (A_f)_\chi$ and $\eta\restriction n'+1\in (A_f)_\chi$. Let $n<\gamma$ be maximal such that $\eta\restriction n\in (A_f)_\chi$. Let $\zeta_1=\eta\restriction (n+1)$, so $\zeta_1\notin (A_f)_\chi$. Since $\chi\in C_2$, there is $\zeta_2\in B^f(\eta\restriction n\ ,\chi)$, $\bar{\mu}_{\zeta_1}=\bar{\mu}_{\zeta_2}$, and $\bar{v}_{\zeta_1}$ and $\bar{v}_{\zeta_2}$ have the same atomic type over $(A_g)_\alpha$. Notice that $\zeta_1,\zeta_2\in (A_f)_\delta$, so $\bar{v}_{\zeta_1}, \bar{v}_{\zeta_2}\subseteq (A_g)_\delta$.

\begin{claim}
$tp_{at}(\bar{v}_{\zeta_1}^\frown \bar{v}_{\eta},\emptyset,A_g)=tp_{at}(\bar{v}_{\zeta_2}^\frown \bar{v}_{\eta},\emptyset,A_g)$
\end{claim}
\begin{proof}
By Fact \ref{closed_lema} it is enough to show that for each $n<lg(\bar{v}_{\eta_1})$, $$tp_{bs}(d(v^n_{\zeta_1}),d(\bar{v}_\eta),(A_g,\prec,<))=tp_{bs}(d(v^n_{\zeta_2}),d(\bar{v}_\eta),(A_g,\prec,<)).$$

Let $v^n_{\zeta_1}\restriction r_1\in d(v^n_{\zeta_1})$ and $v^k_{\eta}\restriction r_2\in d(\bar{v}_{\eta})$.

{\bf Case $v^n_{\zeta_1}\restriction r_1=v^k_{\eta}\restriction r_2$.} Notice that $r_1=r_2$. Since $v^n_{\zeta_1}\in (A_g)_\delta$, $v^k_{\eta}\restriction r_1\in (A_g)_{\delta}$ and $v^k_{\eta}\restriction r_1\in (A_g)_{\alpha_k}$. Therefore $v^n_{\zeta_1}\restriction r_1\in (A_g)_{\alpha}$. Since $v^n_{\zeta_1}$ and $v^n_{\zeta_2}$ have the same type over $(A_g)_\alpha$, $v^n_{\zeta_1}\restriction r_1\prec v^n_{\zeta_2}$. Thus $$v^n_{\zeta_1}\restriction r_1=v^n_{\zeta_2}\restriction r_1=v^k_{\eta}\restriction r_2.$$

{\bf Case $v^n_{\zeta_1}\restriction r_1\prec v^k_{\eta}\restriction r_2$.} As in the previous case, we can conclude that $v^n_{\zeta_1}\restriction r_1\in (A_g)_{\alpha}$. Since $v^n_{\zeta_1}$ and $v^n_{\zeta_2}$ have the same type over $(A_g)_\alpha$, $v^n_{\zeta_2}\restriction r_1\prec v^k_{\eta}\restriction r_2$.

{\bf Case $v^k_{\eta}\restriction r_2 \prec v^n_{\zeta_1}\restriction r_1$.} As in the previous case, we can conclude that $v^k_{\eta}\restriction r_2\in (A_g)_{\alpha}$. Since $v^n_{\zeta_1}$ and $v^n_{\zeta_2}$ have the same type over $(A_g)_\alpha$, $v^k_{\eta}\restriction r_2 \prec v^n_{\zeta_2}\restriction r_1$.

{\bf Case $v^n_{\zeta_1}\restriction r_1< v^k_{\eta}\restriction r_2$.} Clearly $r_1=r_2$. If $v^k_{\eta}\restriction r_2\in (A_g)_\alpha$, then $v^n_{\zeta_2}\restriction r_1< v^k_{\eta}\restriction r_2$ (since $v^n_{\zeta_1}$ and $v^n_{\zeta_2}$ have the same type over $(A_g)_\alpha$). 

Let us take care of the case $v^k_{\eta}\restriction r_2\notin (A_g)_\alpha$. So $r_2>m_k$ and $v^k_{\eta}\restriction r_2\notin (A_g)_\delta$. Since $v^k_{\eta}\restriction r_2-1\prec v^n_{\zeta_1}\restriction r_1$, and $v^n_{\zeta_1}\restriction r_1\in (A_g)_\delta$, $r_2=m_k+1$. By II above, $v^n_{\zeta_1}\restriction r_1\in (A_g)_\delta$ implies that there is $\sigma'\in (A_g)_\alpha$ such that $v^n_{\zeta_1}\restriction r_1\leq \sigma'< v^k_{\eta}\restriction r_2$. Since $v^n_{\zeta_1}$ and $v^n_{\zeta_2}$ have the same type over $(A_g)_\alpha$, $v^n_{\zeta_2}\restriction r_1\leq \sigma'< v^k_{\eta}\restriction r_2$

{\bf Case $v^n_{\zeta_1}\restriction r_1> v^k_{\eta}\restriction r_2$.} Following the proof of the previous case, we only have to take care of the case $v^k_{\eta}\restriction r_2\notin (A_g)_\delta$ and $r_2=m_k+1$. Since $v^n_{\zeta_2}\restriction r_1\in (A_g)_\delta$, $v^n_{\zeta_1}\restriction r_1\neq v^k_{\eta}\restriction r_2$. Therefore either $v^n_{\zeta_2}\restriction r_1< v^k_{\eta}\restriction r_2$ or $ v^k_{\eta}\restriction r_2<v^n_{\zeta_2}\restriction r_1$. Let us suppose, towards contradiction, $ v^k_{\eta}\restriction r_2>v^n_{\zeta_2}\restriction r_1$. By II above, $v^n_{\zeta_2}\restriction r_1\in (A_g)_\delta$ implies that there is $\sigma'\in (A_g)_\alpha$ such that $v^n_{\zeta_2}\restriction r_1\leq \sigma'< v^k_{\eta}\restriction r_2<v^n_{\zeta_1}\restriction r_1$. A contradiction, since $v^n_{\zeta_1}$ and $v^n_{\zeta_2}$ have the same type over $(A_g)_\alpha$.

\end{proof}
From the previous claim and the way $\mathcal{M}^g_1$ was constructed (see Theorem \ref{construction_otop} (3), Theorem \ref{construction_unst} (3), and Theorem \ref{construction_dop} (3)), we know that $$tp_L(\bar{b}_{\bar{v}_{\zeta_1}^\frown \bar{v}_{\eta}},\emptyset,\mathcal{M})=tp_L(\bar{b}_{\bar{v}_{\zeta_2}^\frown \bar{v}_{\eta}},\emptyset,\mathcal{M}).$$ Since $\bar{\mu}_{\zeta_1}=\bar{\mu}_{\zeta_2}$, $$\mathcal{M}^g_1\models \varphi(\bar{\mu}_{\eta}(\bar{b}_{\bar{v}_\eta}), \bar{\mu}_{\zeta_1}(\bar{b}_{\bar{v}_{\zeta_1}})) \Leftrightarrow \varphi(\bar{\mu}_{\eta}(\bar{b}_{\bar{v}_\eta}), \bar{\mu}_{\zeta_2}(\bar{b}_{\bar{v}_{\zeta_2}}))$$ so $$\mathcal{M}^f_1\models \varphi(\bar{a}_{\eta}, \bar{a}_{\zeta_1}) \Leftrightarrow \varphi(\bar{a}_{\eta}, \bar{a}_{\zeta_2}).$$
On the other hand, since $\zeta_1\prec \eta$ and $\zeta_2\not\prec \eta$, $$\mathcal{M}^f\models \varphi(\bar{a}_{\eta}, \bar{a}_{\zeta_1}) \wedge \neg\varphi(\bar{a}_{\eta}, \bar{a}_{\zeta_2}),$$ a contradiction, since $\mathcal{M}^f=\mathcal{M}_1^f\restriction \tau$ and $\varphi\in L(\tau)$.

\end{proof}

\section{Generalized Borel reducibility}\label{Section_main_gap}

Now that we have constructed the models in detail, taking care of every possible constant and variable, we have all we need to prove Theorem A and give an answer to Friedman-Hyttinen-Kulikov's conjecture. We can also prove Theorem B, which gives an answer to ``How big can the gap be?"

\subsection{A Borel reducibility Main Gap in ZFC}

The first step to prove Theorem A, is to construct the continuous reduction $$=^2_\gamma\ \reduc\ \cong_T$$ for each kind of non-classifiable theory.

\begin{thm}\label{maincor}
Let $\kappa$ be inaccessible or $\kappa=\lambda^+=2^\lambda$.  Suppose $T$ is a countable complete non-classifiable theory over a countable vocabulary, $\tau$. 
\begin{enumerate}
	\item If $T$ is stable unsuperstable, then let $\theta=\gamma=\omega$.
    \item If $T$ is unstable, or superstable with the OTOP, then let $\theta=\omega$ and $\omega\leq\gamma<\kappa$.
    \item If $T$ is superstable with the DOP, then let $\theta=2^\omega=\mathfrak{c}$ and $\omega_1\leq\gamma<\kappa$.
    \end{enumerate}
If $\theta$, $\gamma$, and $\kappa$ satisfy that $\forall\alpha<\kappa$, $\alpha^\gamma<\kappa$, and $(2^{\theta})^+\leq\kappa$,
then $$=^2_\gamma\ \reduc\ \cong_T.$$ 
\end{thm}

\begin{proof}
For every $f\in 2^\kappa$, we will construct a model $\mathcal{M}_f$ isomorphic to $\mathcal{M}^f$. We will also construct a function $\mathcal{G}:\{\mathcal{M}_f\mid f\in 2^\kappa\}\rightarrow 2^\kappa$, such that $\mathcal{A}_{\mathcal{G}(\mathcal{M}_f)}\cong \mathcal{M}_f$ and $f\mapsto \mathcal{G}(\mathcal{M}_f)$ is continuous. 
By Remark \ref{equal_trees}, Definition \ref{Def_final_trees_and_repr}, and the definition of $(A_f)_\alpha$, 
$$f\restriction \alpha=g\restriction \alpha \Leftrightarrow (A_f)_\alpha=(A_g)_\alpha.$$ 
For all $\alpha<\kappa$, $A\in K^\gamma_{tr}$, and a $\kappa$-representation $\mathbb{A}=\langle A_\iota\mid \iota<\kappa  \rangle$ of $A$, let us denote by $\tilde{A}_\alpha$ the set $\{a_s\mid s\in A_\alpha\}$, recall the construction of $\mathcal{M}^f_1$. 
Since for all $\alpha<\kappa$, $$(A_f)_\alpha=(A_g)_\alpha \Leftrightarrow SH(\tilde{(A_f)}_\alpha) = SH(\tilde{(A_g)}_\alpha),$$
for all $f$ we can construct a tuple $(\mathcal{M}_f, F_f)$, where $\mathcal{M}_f$ is a model isomorphic to $\mathcal{M}^f$ and $F_f:\mathcal{M}_f\rightarrow \mathcal{M}^f$ is an isomorphism, that satisfies the following: denote by $\mathcal{M}_{f,\alpha}$ the preimage $F^{-1}_f[SH(\tilde{(A_f)}_\alpha)\restriction \tau]$ and $$f\restriction \alpha=g\restriction \alpha \Leftrightarrow \mathcal{M}_{f,\alpha}=\mathcal{M}_{g,\alpha}.$$
For every $f\in 2^\kappa$ there is a bijection $E_f: dom(\mathcal{M}_f)\rightarrow \kappa$, such that for every $f,g\in 2^\kappa$ and $\alpha<\kappa$, if $f\restriction \alpha=g\restriction \alpha $, then $E_f\restriction dom(\mathcal{M}_{f,\alpha})=E_g\restriction dom(\mathcal{M}_{g,\alpha})$ (see \cite{Mor}). 
Let us denote by $\pi$ the bijection $\pi_\kappa$ from Definition \ref{struct}, define the function $\mathcal{G}$ by: 
$$\mathcal{G}(\mathcal{M}_f)(\alpha)=\begin{cases} 1 &\mbox{if } \alpha=\pi(m,a_1,a_2,\ldots,a_n) \text{ and }\\ & \mathcal{M}_{f}\models Q_m(E_f^{-1}(a_1),E_f^{-1}(a_2),\ldots,E_f^{-1}(a_n))\\
0 & \mbox{otherwise. } \end{cases}$$
To show that $G:2^\kappa\rightarrow 2^\kappa$, $G(f)=\mathcal{G}(\mathcal{M}_f)$ is continuous, let $[\zeta\restriction \alpha]$ be a basic open set and $\xi\in G^{-1}[[\zeta\restriction \alpha]]$. There is $\beta<\kappa$ such that for all $\epsilon<\alpha$, if $\epsilon=\pi(m,a_1,\ldots, a_n)$, then $E^{-1}_\xi (a_i)\in dom(\mathcal{M}_{\xi,\beta})$ holds for all $i\leq n$. Since for all $\eta\in [\xi\restriction\beta]$ it holds that $\mathcal{M}_{\eta,\beta}=\mathcal{M}_{\xi,\beta}$, for any $\epsilon<\alpha$ that satisfies $\epsilon=\pi(m,a_1,\ldots, a_n)$ $$\mathcal{M}_\eta\models Q_m(E_\eta^{-1}(a_1),E_\eta^{-1}(a_2),\ldots,E_\eta^{-1}(a_n))$$ if and only if $$\mathcal{M}_\xi\models Q_m(E_\xi^{-1}(a_1),E_\xi^{-1}(a_2),\ldots,E_\xi^{-1}(a_n)).$$
We conclude that $G$ is continuous. The result follows from Theorem \ref{main_teo_4}.
\end{proof}

For any stationary set $X\subseteq \kappa$, let us denote by $\diamondsuit_X$ the following diamond principle.

\textit{There is a sequence $\{f_\alpha\in \alpha^\alpha\mid\alpha\in X\}$ such that for all $F\in \kappa^\kappa$, the set $\{\alpha\in X\mid f_\alpha=F\restriction\alpha\}$ is stationary.}

This principle is equivalent to the classical $\diamondsuit$-principle.

For every regular cardinal $\mu<\kappa$, let us denote by $\diamondsuit_\mu$ the diamond principle $\diamondsuit_X$ when $X=S^\kappa_\mu$.

\begin{fact}[{Hyttinen-Kulikov-Moreno, \cite[Lemma 2]{HKM}}]\label{HKM1}
Assume $T$ is a countable complete classifiable theory over a countable vocabulary. If $\diamondsuit_X$ holds, then $\cong_T \ \reduc\ =^2_X$.
\end{fact}

\begin{fact}[{Friedman-Hyttinen-Kulikov, \cite[Theorem 77]{FHK13}}]\label{notred}
If a first order countable complete theory over a countable vocabulary $T$ is classifiable, then for all $\mu<\kappa$ regular, $=^2_\mu\ \not\redub\ \cong_T$.
\end{fact}

\begin{cor}\label{reduction_coro}
Let $\kappa=\lambda^+=2^\lambda$. Suppose $T_1$ is a countable complete classifiable theory, and $T_2$ be a countable complete non-classifiable theory over a countable vocabulary. 
\begin{enumerate}
	\item If $T_2$ is stable unsuperstable, then let $\theta=\gamma=\omega$.
    \item If $T_2$ is unstable, or superstable with the OTOP, then let $\theta=\omega$ and $\omega\leq\gamma<\kappa$.
    \item If $T_2$ is superstable with the DOP, then let $\theta=\mathfrak c$ and $\omega_1\leq\gamma<\kappa$.
    \end{enumerate}
If $2^{\theta}\leq\lambda=\lambda^\gamma$,
then $$\cong_{T_1}\ \reduc\ =^2_\gamma\ \reduc\ \cong_{T_2} \textit{ and } =^2_\gamma\ \not\redub\ \cong_{T_1}.$$
\end{cor}

\begin{proof}
Since $\lambda^\gamma<\kappa=\lambda^+$, $cf(\lambda)>\gamma$. By \cite{Sh1} we know that if $\kappa=\lambda^+=2^\lambda$ and $cf(\lambda)>\gamma$, then $\diamondsuit_\gamma$ holds. The result follows from Theorem \ref{maincor}, Fact \ref{HKM1},  and Fact \ref{notred}.
\end{proof}

By putting all these results together, we can prove Theorem A. As we have seen, the cardinals $\varepsilon$, $\theta$, and $\gamma$ are the ones that ensure the existence of the continuous reduction. Thus different cardinals ($\varepsilon$, $\theta$, and $\gamma$) generate different reductions, so Theorem A can be stated in more detail. Notice that under the assumptions of Fact \ref{previous_FHK}, the third item of Theorem \ref{Main_Gap} uses $=^2_\gamma$ for all regular cardinals $\omega<\gamma<\lambda$. This contrast with Fact \ref{previous_FHK}, in which the statement is about $=^2_\lambda$.

\begin{thm}[Borel reducibility Main Gap]\label{Main_Gap}
Suppose $\kappa=\lambda^+=2^\lambda$ and $2^{\mathfrak{c}}\leq\lambda=\lambda^{\omega_1}$. If $T_1$ is a countable complete classifiable shallow theory, $T_2$ is a countable complete classifiable theory not shallow, and $T_3$ is a countable complete non-classifiable theory, then the following hold:
\begin{enumerate}
    \item {\bf Gap: Classifiable vs Non-classifiable.} For $T=T_1,T_2$: $$\cong_{T}\ \reduc\ \cong_{T_3}\textit{ and } \cong_{T_3}\ \not\redub\ \cong_{T}.$$
    \item {\bf Simple Gap: Classifiable vs Non-classifiable.} For $T=T_1,T_2$ there is $\gamma<\kappa$ such that: $$\cong_{T}\ \reduc\ =^2_\gamma\ \reduc\ \cong_{T_3}\textit{ and } \cong_{T_3}\ \not\redub\ \cong_{T}.$$ In particular $$=^2_\gamma\ \not\redub\ \cong_{T}.$$
    \item {\bf Gap: Classifiable vs Unstable or Superstable.} If $\lambda$ is such that $\lambda=\lambda^{<\lambda}$. For all $\omega<\gamma<\lambda$ regular, $T_3$ is unstable or superstable, and $T=T_1,T_2$: $$\cong_{T}\ \reduc\ =^2_\gamma\ \reduc\ \cong_{T_3}\textit{ and } \cong_{T_3}\ \not\redub\ \cong_{T}.$$ In particular $$=^2_\gamma\ \not\redub\ \cong_{T}.$$
    \item {\bf Gap: Shallow vs Non-shallow.} If $\kappa=\aleph_\mu$ is such that $\beth_{\omega_1}(\mid\mu\mid)\leq\kappa$, then $$\cong_{T_1}\ \redub\ 0_\kappa\ \redub\ \cong_{T_2}\ \reduc\ \cong_{T_3}.$$ In particular,  
    $$\cong_{T_3}\ \not\redub\ \cong_{T_2}\ \not\redur\ 0_\kappa\ \not\redur\ \cong_{T_1}.$$
    \item {\bf General gap: Shallow vs Non-shallow.} If $T$ is a classifiable shallow theory such that $I(\kappa,T)<I(\kappa,T_1)$, and $\kappa=\aleph_\mu$ is such that $\beth_{\omega_1}(\mid\mu\mid)\leq\kappa$, then $$\cong_{T}\ \redub\ \cong_{T_1}\ \redub\ \cong_{T_2} \reduc\ \cong_{T_3}.$$ In particular $$\cong_{T_3}\ \not\redub\ \cong_{T_2}\ \not\redur\ \cong_{T_1}\ \not\reduc\ \cong_{T}.$$
\end{enumerate}

\end{thm}
\begin{proof}
\begin{enumerate}
	\item It follows from (2).
    \item Notice that $\kappa=\lambda^+=2^\lambda$ and $2^{\mathfrak{c}}\leq\lambda=\lambda^{\omega_1}$ imply that for all $\alpha<\kappa$, $\alpha^{\omega_1}<\kappa$. The result follows from Corollary \ref{reduction_coro} with $\gamma\in\{\omega,\omega_1\}$.
    \item Notice that $\kappa=\lambda^+=2^\lambda$ and $2^{\mathfrak{c}}\leq\lambda=\lambda^{<\lambda}$ imply that for all $\alpha<\kappa$ and $\gamma<\lambda$, $\alpha^\gamma<\kappa$. The result follows from Corollary \ref{reduction_coro}.
    \item It follows from item (1), Fact \ref{Mangraviti_Motto}, and Proposition \ref{ManMot}.
    \item It follows from item (3), Fact \ref{Mangraviti_Motto}, and Proposition \ref{ManMot}.

\end{enumerate}

\end{proof}

The first item solves Friedman-Hyttinen-Kulikov conjecture for large enough $\lambda$ (i.e. $2^{\mathfrak{c}}\leq\lambda$) and some cardinal assumptions. The second item of the previous theorem tells us that depending on the theory $T$, we use $\gamma=\omega$ or $\gamma={\omega_1}$ to construct the reduction $=^2_\gamma\ \reduc\ \cong_T$. Friedman, Hyttinen, Weinstein, and Moreno (\cite{FHK13} and \cite{Mort}) asked whether there is $\gamma<\kappa$ such that $=^2_\gamma\ \reduc\ \cong_T$ holds for any theory $T$. The previous theorem give us a partial answer to this question.

\begin{cor}\label{omega_case}
Suppose $\kappa=\lambda^+=2^\lambda$ and $\lambda=\lambda^{\omega}$. Let $T_1$ be a countable complete classifiable theory and $T_2$ be a countable complete non-classifiable theory. If $T_2$ is unstable, or superstable with the OTOP, or stable unsuperstable, then $$\cong_{T_1}\ \reduc\ =^2_\omega\ \reduc\ \cong_{T_2}\textit{ and } =^2_\omega\ \not\redub\ \cong_{T_1}.$$
\end{cor}

\begin{cor}\label{non-omega_case}
Suppose $\kappa=\lambda^+=2^\lambda$ and $2^{\mathfrak{c}}\leq\lambda=\lambda^{\omega_1}$. Let $T_1$ be a countable complete classifiable theory and $T_2$ be a countable complete non-classifiable theory. If $\omega<\gamma<\lambda$, and $T_2$ is unstable, or superstable, then $$\cong_{T_1}\ \reduc\ =^2_\gamma\ \reduc\ \cong_{T_2}\textit{ and } =^2_\gamma\ \not\redub\ \cong_{T_1}.$$
\end{cor}

A more detailed study of the cardinals $\varepsilon$, $\theta$, and $\gamma$, might lead to a complete answer to this question.

Now that we have showed that classifiable theories are less complex than non-classifiable theories, we can turn our attention to the question ``\textit{How far apart are the complexities?}".

\subsection{Main Gap Dichotomy}

Theorem \ref{SHMGT} (Shelah's Main Gap Theorem) tells us that a non-classifiable theories are more complex than classifiable theories and that their complexities are far apart. Theorem A is a Borel reducibility counterpart of the first part, but it does not tell us how far apart are the Borel reducibility complexity of different theories (except for the Borel non-reduction).
To find a counterpart of \textit{``their complexity are far apart"}, we have to study how many equivalence relations are strictly between $\cong_{T_1}$ and $\cong_{T_2}$? where $T_1$ is classifiable and $T_2$ is not.

We can make the gap as big as we want by using the \textit{Dense non-reduction theorem}. Recall that Baire measurable is weaker than being Borel, thus also continuous (Definition \ref{all_reduc}).

\begin{fact}[{Dense non-reduction; Fernandes-Moreno-Rinot, \cite[Corollary 6.19]{FMR20}}]\label{DNR} Suppose $\kappa^{<\kappa}=\kappa$ holds in $V$. There exists a cofinality-preserving forcing extension in which for all two disjoint stationary subsets $X,Y$ of $\kappa$, ${=^2_X}\not\redum{=^2_Y}$.
\end{fact}

\begin{thm}
Suppose $\kappa=\lambda^+=2^\lambda$ and $2^\mathfrak{c}\leq \lambda=\lambda^{\omega_1}$. There exists a cofinality-preserving forcing extension 
in which the following holds: 

If $T_1$ is a classifiable theory and $T_2$ is a non-classifiable theory. Then there is a regular cardinal $\gamma<\kappa$ such that, for all stationary $X\subseteq S^\kappa_\gamma$ $$\cong_{T_1}\reduc =^2_X\reduc=^2_\gamma\reduc\cong_{T_2},$$
and for all $X$, $Y\subseteq S^\kappa_\gamma$ disjoint stationary subsets, $${=^2_X}\not\redum{=^2_Y}.$$
\end{thm}
\begin{proof}
    Let $T_1$ is a classifiable theory and $T_2$ is a non-classifiable theory. We will start by the case $\kappa$ successor. Let us show that the forcing extension from Fact \ref{DNR} is the desire extension.
    
    The forcing in Fact \ref{DNR} preserves the cardinal arithmetic assumptions, thus by Theorem \ref{maincor}, there is a regular cardinal $\gamma<\kappa$ such that $=^2_\gamma\reduc\cong_{T_2}$. Recall that $\gamma\in \{\omega,\omega_1\}$, and the value of $\gamma$ depends on $T_2$. Therefore, for all stationary $X\subseteq S^\kappa_\gamma$, $X\subseteq \{\alpha<\kappa\mid cf(\alpha)\neq cf(\lambda)\}$ and by \cite{Sh1} $\diamondsuit_X$ holds. By Fact \ref{HKM1}, $\cong_T \ \reduc\ =^2_X$. Clearly, the function $$F(\eta)(\alpha)=\begin{cases} \eta(\alpha) &\mbox{if } \alpha\in X\\
    0 & \mbox{otherwise, } \end{cases}$$
    witnesses $=^2_X\reduc=^2_\gamma$. Finally, from Fact \ref{DNR} we know that in this forcing extension: for all $X$, $Y\subseteq S^\kappa_\gamma$ disjoint stationary subsets, $${=^2_X}\not\redum{=^2_Y}.$$
    
\end{proof}

In the previous theorem, we can replace the assumption to $\kappa$ inaccessible or $\kappa$ a successor cardinal such that $\kappa=\kappa^{<\kappa}=\lambda^+$, $2^\lambda>2^\omega$, and $\lambda^{<\lambda}=\lambda$. In this case the theorem holds for stationary sets $X$, $Y\subseteq S^\kappa_\gamma$ of the ground model. This can be prove by adding a Cohen real, this forces $\diamondsuit_X$ for all stationary set $X\subseteq \kappa$ of the ground model. Since the forcing in Fact \ref{DNR} preserves these diamond principles, the result follows from Fact \ref{HKM1}, Theorem \ref{maincor} (for $\kappa$ inaccessible), Fact \ref{previous_FHK} and Fact \ref{previous_strictstb_FHK} (for $\kappa$ successor), and Fact \ref{DNR}.

The forcing extension from Fact \ref{DNR} is a model in which filter reflection fails (see \cite{FMR20}). One might think that filter reflection makes the gap smaller, but it is the opposite. Some cases of filter reflection, imply that the isomorphism relation of non-classifiable theories is $\Sigma^1_1(\kappa)$-complete. 

A $\Pi^{1}_{2}$-sentence $\phi$ is a formula of the form $\forall X\exists Y\varphi$ where $\varphi$ is a first-order sentence over a relational language $\mathcal L$ as follows:
\begin{itemize}
\item $\mathcal L$ has a predicate symbol $\epsilon$ of arity $2$.
\item $\mathcal L$ has a predicate symbol $\mathbb X$ of arity $m({\mathbb X})$.
\item $\mathcal L$ has a predicate symbol $\mathbb Y$ of arity $m({\mathbb Y})$.
\item $\mathcal L$ has infinitely many predicate symbols $(\mathbb B_n)_{n\in \omega}$, each $\mathbb B_n$ is of arity $m(\mathbb B_n)$.
\end{itemize}

\begin{defn} For sets $N$ and $x$, we say that \emph{$N$ sees $x$} iff
$N$ is transitive, p.r.-closed\footnote{\textit{p.r.-closed} stands for \textit{primitive recursive closed}. A set is p.r.-closed if it is closed under the primitive recursive set functions \cite[Chapter 2]{Dev73}.}, and $x\in N$.
\end{defn}

Notice that if $N$ sees $x$, then $x\cup\{x\}\s N$.

Suppose that a set $N$ sees an ordinal $\alpha$,
and that $\phi=\forall X\exists Y\varphi$ is a $\Pi^{1}_{2}$-sentence, where $\varphi$ is a first-order sentence in the above-mentioned language $\mathcal L$.
For every sequence $(B_n)_{n\in\omega}$ such that, for all $n\in\omega$, $B_n\s \alpha^{m(\mathbb B_n)}$,
we write
$$\langle \alpha,{\in}, (B_{n})_{n\in \omega} \rangle \models_N \phi$$
to express that the two hold:
\begin{enumerate}
\item $(B_{n})_{n\in \omega} \in N$;
\item $\langle N,\in\rangle\models (\forall X\subseteq \alpha^{m(\mathbb X)})(\exists Y\subseteq \alpha^{m(\mathbb Y)})[\langle \alpha,{\in}, X, Y, (B_{n})_{n\in \omega}  \rangle\models \phi]\footnote{$\phi$ relativized to $\langle \alpha,{\in}, X, Y, (B_{n})_{n\in \omega}  \rangle$.}$,
where:
\begin{itemize}
\item    $\in$ is the interpretation of $\epsilon$.
\item $X$ is the interpretation of $\mathbb X$.
\item  $Y$ is the interpretation of $\mathbb Y$.
\item For all $n\in\omega$,  $B_n$ is the interpretation of $\mathbb B_n$.
\end{itemize}
\end{enumerate}

We write $\alpha^+$ for $|\alpha|^+$ and $\langle \alpha,{\in}, (B_{n})_{n\in \omega} \rangle \models \phi$ for $\langle \alpha,{\in}, (B_{n})_{n\in \omega} \rangle \models_{H_{\alpha^+}} \phi$.

\begin{defn}\label{reflectingdiamond}   Let $\kappa$ be a regular uncountable cardinal, and $S\s\kappa$ a stationary set.

    $\dl^*_S(\Pi^1_2)$ asserts the existence of a sequence $\vec N=\langle N_\alpha\mid\alpha\in S\rangle$ satisfying the following:
    
    \begin{enumerate}
        \item For every $\alpha\in S$, $N_\alpha$ is a set of cardinality $<\kappa$ that sees $\alpha$.
        \item For every $X\s\kappa$, there exists a club $C\s\kappa$ such that, for all $\alpha\in C \cap S$, $X\cap\alpha\in N_\alpha$.
        \item Whenever $\langle \kappa,{\in},(B_n)_{n\in\omega}\rangle\models\phi$,
        with $\phi$ a $\Pi^1_2$-sentence,
        there are stationarily many $\alpha\in S$ such that $|N_\alpha|=|\alpha|$ and
        $\langle \alpha,{\in},(B_n\cap(\alpha^{m(\mathbb B_n)}))_{n\in\omega}\rangle\models_{N_\alpha}\phi$.
    \end{enumerate}
 \end{defn}

\begin{fact}[{Fernandes-Moreno-Rinot, \cite[Theorem C]{FMR}}]\label{comp-rels}
If $\dl^*_S(\Pi^1_2)$ holds for $S$, then $=^2_S$ is $\Sigma^1_1(\kappa)$-complete.
\end{fact}

\begin{cor}
Suppose $\kappa$ be inaccessible, or $\kappa=\lambda^+=2^\lambda$ and $2^{\mathfrak{c}}\leq\lambda=\lambda^{\omega_1}$. Let $\mu_0=\omega$, $\mu_1=\omega_1$, and $S_i=\{\alpha<\kappa\mid cf(\alpha)=\mu_i\}$. If $\dl^*_{S_i}(\Pi^1_2)$ holds for all $i\in 2$, and $T$ is a countable complete non-classifiable theory, then $\cong_T$ is $\Sigma^1_1(\kappa)$-complete. 
\end{cor}

\begin{fact}[{Fernandes-Moreno-Rinot, \cite[Lemma 4.10 and Proposition 4.14]{FMR20}}]\label{forcing-comp-rels}
There exists a $<\kappa$-closed $\kappa^+$-cc forcing extension in which $\dl^*_{\check{S}}(\Pi^1_2)$ holds for all $\check{S}\subseteq\kappa$ stationary set ($S$ stationary in $V$).
\end{fact}

\begin{cor}\label{complete_non-class}
Let $\kappa$ be inaccessible, or $\kappa=\lambda^+=2^\lambda$ and $2^{\mathfrak{c}}\leq\lambda=\lambda^{\omega_1}$. There exists a $<\kappa$-closed $\kappa^+$-cc forcing extension in which for all countable complete non-classifiable theory $T$, $\cong_T$ is $\Sigma^1_1(\kappa)$-complete. 
\end{cor}

We proceed to prove Theorem B.

	\begin{thm}[Main Gap Dichotomy]\label{main_gap-dich}
  		Let $\kappa$ be inaccessible, or $\kappa=\lambda^+=2^\lambda$ and $2^{\mathfrak{c}}\leq\lambda=\lambda^{<\omega_1}$. There exists a $<\kappa$-closed $\kappa^+$-cc forcing extension in which for any countable first-order theory in a countable vocabulary (not necessarily complete), $T$, one of the following holds:
\begin{itemize}
\item $\cong_T$ is $\Delta^1_1(\kappa)$.
\item $\cong_T$ is $\Sigma^1_1(\kappa)$-complete.
\end{itemize}
	\end{thm}

\begin{proof}
  Let us show that in the forcing extension of Corollary \ref{complete_non-class} the dichotomy holds.
  
  From Fact \ref{basics_FHK} (2), if a complete theory $T$
  is classifiable, then $\cong_T$ is $\Delta_1^1(\kappa)$.  So for a
  complete countable theory $T$ the result follows from
  Corollary~\ref{complete_non-class}. Suppose $T$ is not a complete theory.  Let
  $\tau$ be the vocabulary of $T$ and $\{T_\alpha\}_{\alpha<2^\omega}$ be
  the set of all the complete theories in $\tau$ that extend
  $T$. Notice that $\cong_T\ =\ \bigcap_{\alpha<2^{\omega}}\cong_{T_\alpha}$,
  therefore if $\cong_{T_\alpha}$ is a $\Delta_1^1(\kappa)$ equivalence
  relation for all $\alpha<2^\omega$, then so is~$\cong_T$ since $2^\omega<\kappa$.
  
  Suppose $T'$ is a complete countable theory in $\tau$ that
  extends $T$ such that $\cong_{T'}$ is not a $\Delta_1^1(\kappa)$
  equivalence relation. By Fact \ref{basics_FHK} (2), $T'$ is a non-classifiable countable
  theory. By Corollary \ref{complete_non-class}, $\cong_{T'}$ is a
  $\Sigma_1^1(\kappa)$-complete equivalence relation. It is easy to see that the map $\mathcal{F}:\kappa^\kappa\rightarrow \kappa^\kappa$
  $$\mathcal{F}(\eta)=
  \begin{cases}
    \eta &\mbox{if } \mathcal{A}_\eta\models T'\\
    \xi & \mbox{otherwise, }
  \end{cases}
  $$
  where $\xi$ is a fixed element of $\kappa^\kappa$ such that
  $\mathcal{A}_\xi\not\models T'$, is a Borel reduction of $\cong_{T'}$ to $\cong_T$, see \cite[Theorem 4.10]{HKM18}.

\end{proof}

In \cite{HKM18} it was showed that there is a model of set theory in which there is a theory $T$ such that $\cong_T$ is neither $\Delta^1_1(\kappa)$, nor $\Sigma^1_1(\kappa)$-complete. Thus, Theorem \ref{main_gap-dich} is not provable in ZFC.

As we have seen the equivalence relations $=^2_S$ play an important role in the Main Gap Dichotomy. Even thought the Main Gap Dichotomy is not a theorem of ZFC, the isomorphism relation of a non-classifiable theory is still very high in the Borel reducibility hierarchy, when it is not $\Delta^1_1(\kappa)$. 

On the other hand, $=^2_S$ are Borel$^*$ sets, meaning that those relations can be characterized by a game. In particular, the game that characterizes $=_\omega^2$ ($S=S^\kappa_\omega$) has length $\omega$. Thus Fact \ref{reduction_classify_descrpitive} and Fact \ref{comp-rels}, imply that $\cong_T$ is Borel$^*$ (under the assumptions of Fact \ref{comp-rels}) for any theory $T$ (classifiable or not classifiable). In other words, $\dl^*_{S^\kappa_\omega}(\Pi^1_2)$ implies the existence of a game of length $\omega$ that captures the isomorphism of models of size $\kappa$. 

It is known that the isomorphism relation of classifiable theories is characterized by the Ehrenfeucht-Fra\"iss\'e game, $EF^\kappa_\omega$-game (see \cite{FHK13} or \cite{Sh90}). 

\textit{$T$ is classifiable if and only if for any pair of models of $T$, $\mathcal{A}$ and $\mathcal{B}$ of size $\kappa$, the second player has a winning strategy in the game $EF^\kappa_\omega(\mathcal{A},\mathcal{B})$ if and only if $\mathcal{A}$ and $\mathcal{B}$ are isomorphic.}

Generally speaking, there is a game of length  $\omega$ that captures the isomorphism of models of classifiable theories. For non-classifiable theories we will need to use the machinery of GDST to define a game of length  $\omega$ that captures the isomorphism of models of any theory (in certain models of ZFC).

The Borel$^*$-games are very useful to work with analytic relations, unfortunately are also a very abstract object of GDST. Luckily the machinery of $\dl^*_S(\Pi^1_2)$ allow us to interpret the Borel$^*$-games in a more general and natural way. 

A tree $T$ is a $\kappa^{+},\lambda$-tree if all the branches have order at most type $\lambda$ and its cardinality is less than $\kappa^+$. It is closed if every chain has a unique supremum in $T$.

\begin{defn}[Models game]
The game $M(T,H,G,\overline N, \overline D, \overline C, A, \varphi)$ is defined by the following parameters:
\begin{enumerate}
\item Let $T\subseteq \kappa^{\leq \omega}$ be a closed $\kappa^+, \omega+1$-tree and denote by $L(T)$ the set of leaves of $T$. 
\item Let $H: L(T)\rightarrow S^\kappa_\omega$ and  $G: L(T)\rightarrow \kappa$. 
\item Let $\overline N=\langle N_\alpha\mid \alpha\in S^\kappa_\omega\rangle$ be a sequence of sets such that for all $\alpha\in S^\kappa_\omega$ the following hold:
\begin{itemize}
\item $N_\alpha$ is transitive and p.r. closed,
\item $|N_\alpha|<\kappa$,
\item $\alpha\cup\{\alpha\}\subseteq N_\alpha$.
\end{itemize}
\item Let $\overline D$ be a transversal sequence of $\overline N$, i.e. $\overline D = \langle D_\alpha\mid \alpha\in S^\kappa_\omega\rangle\in \prod_{\alpha\in S^\kappa_\omega} N_\alpha$. 
\item Let $m<\omega$ and $\overline{C}=\{C^i_\alpha\mid \alpha<\kappa, i<m\}$ a set of sets.
\item Let $A$ be a predicate of arity $m_A$, $\{X_i\mid i<2\}$ predicates of arity $m_i$ respectively, and $\varphi$ a $\Pi^1_2$-sentence involving predicates $A$ and $\{X_i\mid i<2\}$. 
\end{enumerate} 
For all $i<m$ let $$Z^i=\{b\in L(T)\mid A\cap H(b)^{m_A}\text{ and }C^i_{G(b)}\text{ are both in }N_{H(b)}\}.$$ We say that $b$ is $i$-valid if $b\in Z^i$ and $$\langle H(b),\in,A\cap H(b)^{m_A}, D_{H(b)}, C^i_{G(b)}\rangle\models_{N_{H(b)}}\varphi.$$ 

The game $M(T,H,G,\overline N, \overline D, \overline C, A, \varphi)$ is played by $\mathbb{I}$ and $\mathbb{I}\mathbb{I}$ as follows:
The game starts at the root of $T$ and $\mathbb{I}$ chooses an immediate successor of it. If after $n$ moves the game is at the node $x$, then:
\begin{itemize}
\item if $n$ is even, then $\mathbb{I}$ chooses an immediate successor of $x$.
\item if $n$ is odd, then $\mathbb{I}\mathbb{I}$ chooses an immediate successor of $x$.
\end{itemize}
The game finishes after $\omega$ moves at the unique limit, $b\in L(T)$, of the chosen nodes.
Player $\mathbb{I}\mathbb{I}$ wins if one of the following holds:
\begin{itemize}
\item for all $i<m$, $b$ is $i$-valid,
\item for all $i<m$, $b$ is not $i$-valid.
\end{itemize}
\end{defn}

We write $\mathbb{I}\mathbb{I}\uparrow M(T,H,G,\overline N, \overline D, \overline C, A, \varphi)$ when $\mathbb{I}\mathbb{I}$ has a winning strategy for the game $M(T,H,G,\overline N, \overline D, \overline C, A, \varphi)$.

\begin{lemma}
Suppose $\dl^*_S(\Pi^1_2)$ holds. For all $\Sigma^1_1(\kappa)$ equivalence relation $R$, there are $T_R$, $H_R$, $G_R$, $\overline{N}_R$, $\overline{D}_R$, $A_R$, and $\varphi_R$ such that for all $\eta,\xi\in \kappa^\kappa$ there is $\overline{C(\eta,\xi)}$ such that $$\mathbb{I}\mathbb{I}\uparrow M(T_R,H_R,G_R,\overline N_R, \overline D_R, \overline{C(\eta,\xi)}, A_R, \varphi_R)\text{ iff }\eta\ R\ \xi.$$
\end{lemma}

\begin{proof}
Let $R$ be a $\Sigma^1_1(\kappa)$ equivalence relation.
Let $T$ be the tree of all the increasing sequence of length at most $\omega$ ordered by endextension, i.e. $\zeta:\kappa^{\leq\omega}\rightarrow\kappa$. For all $b\in L(T)$ define $H(b)=G(b)=sup(b)$. Let $\overline{N}=\langle N_\alpha\mid \alpha\in S^\kappa_\omega\rangle$ be a $\dl^*_S(\Pi^1_2)$ sequence. From \cite[Theorem 3.5]{FMR}, there are $A$ a predicate of arity 5 and a $\Pi^1_2$-sentence $\varphi$ such that $$\eta\ R\ \xi \ \Leftrightarrow \ \langle\kappa,\in,A,\eta,\xi\rangle\models \varphi,$$ and $\varphi$ implies the sentence ``$R$ is transitive and reflexive", a $\Pi^1_2$-sentence.

From the Transversal lemma \cite[Proposition 3.1]{FMR} there exists a transversal $\overline{D}=\langle D_\alpha\mid \alpha\in S^\kappa_\omega\rangle\in\prod_{\alpha\in S^\kappa_\omega}N_\alpha$ satisfying the following:

For every $\zeta\in\kappa^\kappa$,
whenever $\langle \kappa,{\in},A,\eta,\xi\rangle\models \varphi$,
there are stationarily many $\alpha\in S^\kappa\omega$ such that
\begin{enumerate}
\item $D_\alpha=\zeta\restriction\alpha$, and
\item $\langle \alpha,{\in},A\cap \alpha^{5},\eta\restriction\alpha,\xi\restriction\alpha\rangle\models_{N_\alpha}\varphi$.
\end{enumerate}
Finally, let $\overline{C(\eta,\xi)}=\{C^i_\alpha\mid \alpha<\kappa, i<2\}$, where $C^0_\alpha=\eta\restriction \alpha$ and $C^1_\alpha=\xi\restriction \alpha$.

By \cite[Theorem 3.5]{FMR}, $\eta\ R\ \xi$ holds if and only if there is a club $E\subseteq \kappa$ such that for all $\alpha\in E\cap S^\kappa_\omega$, $\eta\restriction\alpha,\xi\restriction\alpha\in N_\alpha$ and $$\langle\alpha,\in,A\cap\alpha^5, D_\alpha,\eta\restriction\alpha\rangle\models_{N_\alpha}\varphi\ \Leftrightarrow\ \langle\alpha,\in,A\cap\alpha^5, D_\alpha,\xi\restriction\alpha\rangle\models_{N_\alpha}\varphi.$$

\begin{claim}
$$\mathbb{I}\mathbb{I}\uparrow M(T,H,G,\overline N, \overline D, \overline{C(\eta,\xi)}, A, \varphi)\text{ iff }\eta\ R\ \xi.$$
\end{claim}
\begin{proof}
$\Leftarrow$) Suppose $\eta\ R\ \xi$. Let $E$ be the club mentioned above. By the way $T$ was defined, there is a strategy $\sigma$ for $\mathbb{I}\mathbb{I}$ such that the game always end in a leaf $b\in L(T)$ such that $sup(b)\in E$. Thus $H(b)=G(b)\in E$. Therefore, $b$ is 0-valid if and only if it  is 1-valid, and $\sigma$ is a winning strategy for $\mathbb{I}\mathbb{I}$.

$\Rightarrow$) Suppose $\mathbb{I}\mathbb{I}\uparrow M(T,H,G,\overline N, \overline D, \overline{C(\eta,\xi)}, A, \varphi)$. Let $\sigma$ be a winning strategy for $\mathbb{I}\mathbb{I}$. We can represent $\sigma$ as a function $\sigma: \kappa^{<\omega}\rightarrow \kappa$, so that if $\mathbb{I}\mathbb{I}$ plays following $\sigma$, then $\mathbb{I}\mathbb{I}$ moves from $x$ to $x^\frown\langle \sigma(x)\rangle$. Let $E^\sigma$ be the club of closed points of $\sigma$ (i.e. the ordinals $\alpha<\kappa$ such that $\sigma[\alpha^{<\omega}]\subseteq \alpha$). Thus, for all $\alpha\in E^\sigma\cap S^\kappa_\omega$, $\eta\restriction\alpha,\xi\restriction\alpha\in N_\alpha$ implies $$\langle\alpha,\in,A\cap\alpha^5, D_\alpha,\eta\restriction\alpha\rangle\models_{N_\alpha}\varphi\ \Leftrightarrow\ \langle\alpha,\in,A\cap\alpha^5, D_\alpha,\xi\restriction\alpha\rangle\models_{N_\alpha}\varphi.$$ Since $\overline{N}$ is a $\dl^*_S(\Pi^1_2)$ sequence, there is a club $E\subseteq E^\sigma$ such that $\eta\restriction\alpha,\xi\restriction\alpha\in N_\alpha$. By \cite[Theorem 3.5]{FMR}, we conclude that $\eta\ R\ \xi$
\end{proof}
\end{proof}

We know that for any theory, $T$, the isomorphism relation of $T$ is $\Sigma_1^1(\kappa)$. Thus the existence of a $\dl^*_S(\Pi^1_2)$ sequence implies the existence of a game $\mathfrak{G}(\mathcal{A},\mathcal{B})$ of length $\omega$ such that for any theory $T$ and models $\mathcal{A},\mathcal{B}\models T$, $$\mathcal{A}\cong\mathcal{B}\Leftrightarrow \mathbb{I}\mathbb{I}\uparrow\mathfrak{G}(\mathcal{A},\mathcal{B}).$$

\subsection{Few equivalence classes}

In Section \ref{section_intro}, Shelah's Main Gap was presented as a milestone in the program of the spectrum problem at uncountable cardinals. This program started with the study of categoricity and ended with Hart-Hrushovski-Laskowski theorem with all the possible values for the spectrum function \cite[Theorem 61]{HHL}. Mangraviti and Motto Ros studied the categoricity in GDST, Fact \ref{MM_categoric} characterizes uncountable categorical theories as those for which $\cong_T$ is a clopen. In their article, Mangraviti and Motto Ros show that there is no theory $T$ such that $\cong_T$ is open not closed, due to the following result.

\begin{fact}[{Folklore; \cite[Proposition 2.17]{MM}}]\label{open_to_clopen}
Let $E$ be an equivalence relation on $\kappa^\kappa$. $E$ is open if and only if $E$ is clopen.
\end{fact}

Recall $\cong_T^\mu$  and $0_\varrho$ from Section \ref{section_intro}. We know that if $\cong_T$ has $\varrho\leq\kappa$ equivalence classes, then $\cong_T$ and $0_\varrho$ are Borel bireducible. 
It is clear that $0_\varrho$ is more simple than $\cong_T$. Thus, to determine which equivalence relation is more complex we would need to use another type of reduction.
To have a ``microscopic" picture of this complexities, we have to introduce a new type of reduction which is stronger than all the ones we have used so far ($\redur$, $\redum$, $\redub$, $\reduc$, and $\redul$).

\begin{defn}
    Let $S\subseteq \kappa$ be unbounded. We say that a function $f:\kappa^\kappa\rightarrow\kappa^\kappa$ is $S$-recursive if there is a function $H:\kappa^{<\kappa}\rightarrow\kappa^{<\kappa}$ such that for all $\alpha\in S$ and $\eta\in\kappa^\kappa$, $f(\eta)(\theta)=H(\eta\restriction\alpha)(\theta)$ for all $\theta<min(S\backslash \alpha+1)$. 
\end{defn}

We say that $f:\kappa^\kappa\rightarrow\kappa^\kappa$ is recursive if there is a club $C$ such that $f$ is $C$-recursive.

\begin{prop}\label{rec_cont_Lipt_rev}
    Let $S\subseteq\kappa$ be unbounded and $f: \kappa^\kappa\rightarrow\kappa^\kappa$ an $S$-recursive function.
    \begin{enumerate}
        \item $f$ is continuous.
        \item  If $S$ is a club that satisfies the following: 
        
        \textit{$(\dagger)$ $\alpha_m=min(S)$ is such that for all $\eta,\xi\in \kappa^\kappa$ and $\beta<\alpha_m$, $\eta\restriction\beta=\xi\restriction\beta$ implies $f(\eta)\restriction\beta=f(\xi)\restriction\beta$.}
        
        Then $f$ is Lipschitz.
    \end{enumerate}
\end{prop}
\begin{proof}
\begin{enumerate}
    \item Let $\xi,\eta\in \kappa^\kappa$ and $\alpha<\kappa$ be such that $\xi\in f^{-1}[[\eta\restriction\alpha]]$. Since $S$ is unbounded, there is $\beta=min(S\backslash \alpha)$. By the definition of $S$-recursivity, for all $\theta<min(S\backslash \beta+1)$, $f(\xi)(\theta)=H(\xi\restriction\beta)(\theta)$. Thus for all $\zeta\in [\xi\restriction\beta]$ and $\theta\leq\beta$, $f(\zeta)(\theta)=H(\zeta\restriction\beta)(\theta)=H(\xi\restriction\beta)(\theta)=f(\xi)(\theta)$. Since $\alpha\leq\beta$, $f(\zeta)\in [[\eta\restriction\alpha]]$. We conclude that $f$ is continuous.
    \item Let $\alpha_m=min(S)$. Let $\eta,\xi\in \kappa^\kappa$ be such that $\Delta(\eta,\xi)=\alpha<\kappa$. So $\eta\restriction\alpha=\xi\restriction\alpha$. We will show that $\Delta(f(\eta),f(\xi))>\alpha$. 
    
    If $\alpha<\alpha_m$, by $(\dagger)$, $\alpha_m\leq\Delta(f(\eta),f(\xi))$ as wanted. 
    
    For the other case, suppose $\alpha_m\leq\alpha$.

    {\bf Case $\alpha\notin S$.}
    
    Since $\alpha_m\leq\alpha$, $\alpha_m\in S\cap \alpha\neq\emptyset$. Let $\alpha'=sup(S\cap \alpha)$. Since $S$ is closed, $\alpha'\in S$. Let $\beta=min (S\backslash \alpha'+1)$. So for all $\theta<\beta$, $f(\eta)(\theta)=H(\eta\restriction\alpha')(\theta)$ and $f(\xi)(\theta)=H(\xi\restriction\alpha')(\theta)$. Since $\alpha'<\alpha<\beta$, $H(\eta\restriction\alpha')=H(\xi\restriction\alpha')$. Thus $f(\eta)(\theta)=f(\xi)(\theta)$ for all $\theta<\beta$. We conclude that $\Delta(f(\eta),f(\xi))\ge\beta>\alpha$.

    {\bf Case $\alpha\in S$.}

   Since $f$ is $S$-recursive by $H$, $f(\eta)(\theta)=H(\eta\restriction\alpha)(\theta)=H(\xi\restriction\alpha)(\theta)=f(\xi)(\theta)$ for all $\theta< min(S\backslash \alpha+1)$. We conclude that $\Delta(f(\eta),f(\xi))>\alpha$.

\end{enumerate}

\end{proof}

\begin{prop}\label{not_rec_cont_Lipt_rev}
    \begin{enumerate}
        \item There is a Lipschitz function that it is not $S$-recursive for all $S\subseteq \kappa$.
        \item There is an unbounded set $S\subseteq \kappa$ such that there is an $S$-recursive function that it is not Lipschitz.
    \end{enumerate}
\end{prop}
\begin{proof}
    \begin{enumerate}
        \item The identity function, $id_\kappa:\kappa^\kappa\rightarrow\kappa^\kappa$, is clearly Lipschitz. For all unbounded set $S\subseteq \kappa$ we can find $\alpha\in S$ and $\eta,\xi \in \kappa^\kappa$ such that $\eta(\alpha)\neq\xi(\alpha)$ and $\eta\restriction\alpha=\xi\restriction\alpha$. So, for all $H:\kappa^{<\kappa}\rightarrow\kappa^{<\kappa}$, $H(\eta\restriction\alpha)=H(\xi\restriction\alpha)$. Thus either $id_\kappa(\eta)(\alpha)\neq H(\eta\restriction\alpha)(\alpha)$ or $id_\kappa(\xi)(\alpha)\neq H(\xi\restriction\alpha)(\alpha)$ (or both). We conclude that $id_\kappa$ is not $S$-recursive for all $S\subseteq \kappa$.
        \item Let $S=S^\kappa_\omega$ and $\kappa>\omega_1$, it is easy to see that the following function is $S^\kappa_\omega$-recursive but not Lipschitz.

$$f(\eta)(\alpha)=\begin{cases} 0 &\mbox{if } \alpha<\omega_1\\
0 & \mbox{if } \alpha>\omega_1 \mbox{ and } \eta(\omega_1+2)\neq 1\\
1 & \mbox{otherwise. } \end{cases}$$  
    \end{enumerate}
\end{proof}

By Proposition \ref{rec_cont_Lipt_rev}, any $\kappa$-recursive function is Lipschitz. By Proposition \ref{not_rec_cont_Lipt_rev} (2) $S^\kappa_\omega$-recursive does not imply $\kappa$-recursive. The other direction is also true, i.e. $\kappa$-recursive does not imply $S^\kappa_\omega$-recursive. The following function is $\kappa$-recursive and not $S^\kappa_\omega$-recursive for $\kappa>\omega_1$:

$$f(\eta)(\alpha)=\begin{cases} 0 &\mbox{if } \alpha=0\mbox{ or } \alpha \mbox{ is a limit ordinal}\\
\eta(\alpha-1) & \mbox{otherwise } \end{cases}$$  

The existence of an $S$-recursive reduction of $E_0$ to $E_1$ is denoted by $E_0\hookrightarrow_{R(S)}E_1$. We will denote by $E_0\hookrightarrow_{R}E_1$ the existence of a recursive reduction.
Many known continuous reductions are indeed $S$-recursive reductions. This is due to the Approximation Lemma (below).

Let $S\subseteq \kappa$ be a stationary set, we say that an equivalence relation $E\subseteq \kappa^\kappa\times\kappa^\kappa$ has an $S$-approximation if there is $\langle E_\alpha\mid \alpha<\kappa\rangle$ a sequence of relations, $E_\alpha\subseteq \kappa^\alpha\times\kappa^\alpha$, such that the following hold:
\begin{enumerate}
    \item There is $C\subseteq \kappa$ a club such that for all $\alpha\in C$, $E_\alpha$ is an equivalence relation.
    \item For all $\eta,\xi\in \kappa^\kappa$, if $\eta\ E\ \xi$, then there is $D\subseteq C$ a club, such that for all $\alpha\in D$, $$\eta\restriction\alpha\ E_\alpha\ \xi\restriction\alpha.$$
    \item For all $\eta,\xi\in \kappa^\kappa$, if $\neg(\eta\ E\ \xi)$, then there is $S'\subseteq S$ a stationary set, such that for all $\alpha\in S'$, $$\neg(\eta\restriction\alpha\ E_\alpha\ \xi\restriction\alpha).$$
\end{enumerate}

\begin{lemma}[Approximation Lemma]\label{approxlemma}
    Suppose $S\subseteq \kappa$ is a stationary set, and $E\subseteq \kappa^\kappa\times\kappa^\kappa$ is an equivalence relation with an $S$-approximation, $\langle E_\alpha\mid \alpha<\kappa\rangle$. Then $$E\hookrightarrow_{R(S)} =^\kappa_S.$$
\end{lemma}
\begin{proof}
    Let $C\subseteq \kappa$ be a club that witnesses that $\langle E_\alpha\mid \alpha<\kappa\rangle$ is an $S$-approximation of $E$. For all $\alpha\in C$, let $\langle x^\alpha_i\mid 0<i<\kappa\rangle$ be an enumeration of the $E_\alpha$-equivalence classes (this can be done since $\kappa^{<\kappa}=\kappa$). Let us define $F:\kappa^\kappa\rightarrow\kappa^\kappa$ as follows:
        $$F(\eta)(\alpha)=\begin{cases} i &\mbox{if }\alpha\in C\cap S\mbox{ and } \eta\restriction\alpha\in x_i^\alpha,\\
    0 & \mbox{otherwise. }\end{cases}$$  
    It is easy to see from the definition of $S$-approximation that $\eta\ E\ \xi$ if and only if $F(\eta)\ =_S^\kappa\ F(\xi)$.

    Let us define $H:\kappa^{<\kappa}\rightarrow\kappa^{<\kappa}$ as follows:
        $$H(\eta\restriction\alpha)=\begin{cases} F(\eta)\restriction \alpha' &\mbox{if }\alpha\in S\mbox{ and } \alpha'=min(S\backslash (\alpha+1)),\\
    \bar0_\alpha & \mbox{otherwise. }\end{cases}$$ 
    Where $\bar0_\alpha$ is the function constant to $0$ with domain $\alpha$.

    It is easy to see that $H$ witnesses that $F$ is $S$-recursive.

\end{proof}

For a stronger version of the Approximation Lemma, see \cite[Lemma 6.6]{MP25}.

\begin{remark}
    In the previous proof, by changing $S$ to $C$ in the definition of $H$, we can prove that $F$ is $C$-recursive. Also, for all $\beta<min(C)$ and $\eta\in \kappa^\kappa$, $F(\eta)(\beta)=0$. By Proposition \ref{rec_cont_Lipt_rev}, $F$ is Lipschitz.
\end{remark}

The following results are proved by finding a $S$-approximation sequence and using the Approximation Lemma.

\begin{fact}[{Hyttinen-Moreno, \cite[Theorem 2.8]{HM}}]
    Assume $T$ is a countable complete classifiable theory over a countable vocabulary, $S\subseteq\kappa$ a stationary set, and $\mu$ a regular cardinal. Then  $\cong_T \ \hookrightarrow_{R}\ =^\kappa_S$, in particular  $\cong_T \ \hookrightarrow_{R(S)}\ =^\kappa_S$.
\end{fact}

\begin{fact}[{Hyttinen-Kulikov-Moreno, \cite[Lemma 2]{HKM}}]\label{second_HKM}
Assume $T$ is a countable complete classifiable theory over a countable vocabulary. Let $S\subseteq\kappa$ a stationary set. If $\diamondsuit_S$ holds, then $\cong_T \ \hookrightarrow_{R}\ =^2_S$, in particular $\cong_T \ \hookrightarrow_{R(S)}\ =^2_S$.
\end{fact}

\begin{fact}[Fernandes-Moreno-Rinot]\label{second_FMR}
There exists a $<\kappa$-closed $\kappa^+$-cc forcing extension in which for all $\check{S}\subseteq\kappa$, $=^2_S$ is $\Sigma^1_1(\kappa)$-complete by an $S$-recursive function, i.e. for all $\Sigma^1_1(\kappa)$ equivalence relation $E$, $$E \ \hookrightarrow_{R(S)}\ =^2_S.$$ In particular the function is also $\kappa$-recursive.
\end{fact}

We can also define $S$-recursive functions for bounded sets. Let $S\neq\emptyset$ be bounded and $\beta=sup(S)$.
We say that a function $f:\kappa^\kappa\rightarrow\kappa^\kappa$ is $S$-recursive if the following hold
\begin{itemize}
\item There is a function $H:\kappa^{<\beta}\rightarrow\kappa^{<\kappa}$ such that for all $\alpha\in S$ and $\eta\in\kappa^\kappa$, $f(\eta)(\theta)=H(\eta\restriction\alpha)(\theta)$ for all $\theta<min(S\backslash \alpha+1)$.
\item For all $\eta,\xi\in \kappa^\kappa$, if $\eta\restriction\beta=\xi\restriction\beta$, then $f(\eta)=f(\xi)$.
\end{itemize}

If $S=\{\beta\}$ for some $\beta<\kappa$, then a function $f:\kappa^\kappa\rightarrow\kappa^\kappa$ is $S$-recursive if for all $\eta,\xi\in \kappa^\kappa$, $\eta\restriction\beta=\xi\restriction\beta$ implies $f(\eta)=f(\xi)$. 
Let us denote by $\redua$ the existence of an $S$-recursive reduction when $S=\{\alpha\}$. It is clear that any $\{\alpha\}$-recursive function, is $S$-recursive, where $S=\kappa\backslash\alpha$. Thus $\{\alpha\}$-recursive functions are Lipschitz.

These $S$-recursive functions are the Motto Ros's full functions (see \cite{Mot2}). 

Recall the relation $0_\varrho$. We used this relation to study the the complexity of the isomorphism relation of theories with at most $\kappa$ non-isomorphic models. We can generalize this relation to understand these complexities better.

\begin{defn}[Counting $\alpha$-classes]
    Let $\alpha<\kappa$ be an ordinal and $0<\varrho\leq \kappa$ is a cardinal. Let us define the equivalence relation $\alpha_\varrho\in \kappa^\kappa\times\kappa^\kappa$ as follows: $\eta\ \alpha_\varrho\ \xi$ if and only if one of the following holds:
\begin{itemize}
\item  $\varrho$ is finite:
    \begin{itemize}
        \item $\eta(\alpha)=\xi(\alpha)<\varrho-1$;
        \item $\eta(\alpha),\xi(\alpha)\ge\varrho-1$.
    \end{itemize}
\item $\varrho$ is infinite:

    \begin{itemize}
        \item $\eta(\alpha)=\xi(\alpha)<\varrho$;
        \item $\eta(\alpha),\xi(\alpha)\ge\varrho$.
    \end{itemize}
\end{itemize}
\end{defn}

Clearly if $\alpha<\beta<\kappa$ are ordinals and $0<\varrho\leq \kappa$ is a cardinal, then $\alpha_\varrho\ \redup\ \beta_\varrho$ and $\beta_\varrho\ \not\redup\ \alpha_\varrho$. 

\begin{lemma}\label{difcomplexity}
Suppose $\kappa>2^\omega$ and $T$ is a countable first-order theory in a countable vocabulary (not
necessarily complete) such that $\cong_T$ has $\varrho\leq\kappa$ equivalence classes. Then for all $\alpha<\kappa$, $$\cong_T\ \redub\ \alpha_\varrho\textit{ and }\ \alpha_\varrho\ \redup\  \cong_T.$$ In particular $\alpha_\varrho\ \redul\  \cong_T$. Even more, if $T$ is not categorical then $\cong_T\ \not\reduc\ \alpha_\varrho$. 
\end{lemma} 

\begin{proof}
Let $T'$ be a complete theory extending $T$, then $\cong_{T'}\ \redub\ \cong_{T}$. Thus $\cong_{T'}\ \redur\ \cong_{T}$ and $\cong_{T'}$ has at most $\varrho\leq\kappa$ equivalence classes.
By Fact \ref{SHMGT} and Fact \ref{basics_FHK}, $T'$ is classifiable and shallow, and $\cong_{T'}$ is $\kappa$-Borel.  Therefore for all the complete theories, $T'$, that extend $T$, $\cong_{T'}$ is $\kappa$-Borel. Let
  $\tau$ be the vocabulary of $T$ and $\{T_\alpha\}_{\alpha<2^\omega}$ be
  the set of all the complete theories in $\tau$ that extend
  $T$. Since $\cong_T\ =\ \bigcap_{\alpha<2^{\omega}}\cong_{T_\alpha}$,
  $\cong_T$ is a $Borel$ equivalence
  relation. By Proposition \ref{ManMot}, we conclude that $\cong_T\ \redub\ \alpha_\varrho$.
  
  Since $\cong_T$ has $\varrho$ equivalence classes, then there is a sequence $\langle \zeta_i\mid i<\varrho\rangle$ such that for all $\eta\in \kappa^\kappa$, there is $i<\varrho$ such that $\eta\ \cong_T\ \zeta_i$ and for all $i,j<\varrho$, $$\zeta_i\ \cong_T\ \zeta_j\ \Leftrightarrow\ i=j.$$
  The function $\mathcal{F}:\kappa^\kappa\rightarrow \kappa^\kappa$, $\mathcal{F}(\eta)=\zeta_i$ where $\eta(\alpha)=i$, is a reduction of $\alpha_\varrho$ to $\cong_T$. Clearly $\mathcal{F}$ is $\{\alpha+1\}$-recursive. So $\alpha_\varrho\ \redup\  \cong_T$.
  
  Notice that $\alpha_\varrho$ is open.
  Let us suppose, towards contradiction, that $T$ is not categorical and $\cong_T\ \reduc\ 0_\varrho$. By Proposition \ref{reduction_classify_descrpitive}, $\cong_T$  is open. From Fact \ref{MM_categoric} and Fact \ref{open_to_clopen}, $T$ is categorical, a contradiction.
\end{proof}

From Hart-Hrushovski-Laskowski \cite[Theorem 6.1]{HHL} we know that there are classifiable shallow theories that are not $\kappa$-categorical. Thus $I(\ \cdot\ ,T)$ could have values different from 0 and 1, below $\kappa$, when $\aleph_\mu=\kappa=\lambda^+=2^\lambda$ is such that $\beth_{\omega_1}(\mid\mu\mid)\leq\kappa$.

Notice that the only property of $\cong_T$ that we used in the proof, was that it has $\varrho\leq\kappa$ equivalence classes.
Thus for all equivalence relation $E$ with $\varrho\leq\kappa$ equivalence classes and $\alpha<\kappa$, $\alpha_\varrho\ \redup\  E$. So for all $\varrho_1<\varrho_2\leq\kappa$, $\alpha_{\varrho_1}\ \redup\ \alpha_{\varrho_2}$.

We can refine the general gap between the Shallow and Non-shallow notions appearing in Theorem \ref{Main_Gap}.

\begin{thm}[General gap Shallow and Non-shallow]\label{GGSNS}
Suppose $\aleph_\mu=\kappa=\lambda^+=2^\lambda$ is such that $\beth_{\omega_1}(\mid\mu\mid)\leq\kappa$, and $2^{\mathfrak{c}}\leq\lambda=\lambda^{\omega_1}$. Let $T_0$ and $T_1$ be countable complete classifiable shallow theories with $1<I(\kappa,T_0)<I(\kappa,T_1)=\varrho\leq\kappa$, $T_2$ be a countable complete classifiable theory non-shallow, and $T_3$ be a countable complete non-classifiable theory.
Then for all $\alpha<\kappa$, $$\cong_{T_0}\ \redub\ \alpha_\varrho\ \redup\ \cong_{T_1}\ \redub\ \alpha_\kappa\ \redup\ \cong_{T_2}\ \reduc\ \cong_{T_3}.$$ In particular $$\cong_{T_3}\ \not\redub\ \cong_{T_2}\ \not\redur\ \alpha_\kappa\ \not\redur\ \cong_{T_1}\ \not\reduc\ \alpha_\varrho\ \not\redur\ \cong_{T_0}.$$
\end{thm}

We can study the gap between a classifiable shallow and a classifiable non-shallow.
Let us study Question \ref{ManmotQ} (Mangraviti-Motto Ros question \cite[Question 6.9]{MM}). 

By Proposition \ref{ManMot} and Fact \ref{reduction_classify_descrpitive}, we know that if $E_1$ is a $\kappa$-Borel equivalent relation with at most $\kappa$ equivalence classes, then for all $E_0$ $$E_0 \redub E_1  \Leftrightarrow E_0\redur E_1 \text{ and } E_0 \text{ is $\kappa$-Borel}.$$
For all $\alpha<\beta$ cardinals, we denote by $Card(\alpha,\beta)$ the set of cardinals in the interval $(\alpha,\beta)$ ordered by cardinality. We define $Card[\alpha,\beta]$, $Card[\alpha,\beta)$, and $Card(\alpha,\beta]$ in a similar way.
Thus, we can understand the Borel reducibility between $\kappa$-Borel equivalence relation with at most $\kappa$ equivalence classes as $Card[1,\kappa]$.

Under the assumptions CS, if $T_1$ and $T_2$ are such that $I(\kappa,T_1)<I(\kappa,T_2)<\kappa$, then by Fact \ref{SHMGT} and Fact \ref{basics_FHK} $\cong_{T_1}$ and $\cong_{T_2}$ are $\kappa$-Borel. 
From the previous observation, 
there are exactly $Card(I(\kappa,T_1),I(\kappa,T_2))$ equivalence relations strictly in between $\cong_{T_1}$ and $\cong_{T_2}$ (with respect to Borel reducibility). 
On the other hand, if we use Lipschitz reduction, then there is a copy of $Card(I(\kappa,T_1),I(\kappa,T_2)]$ in between $\cong_{T_1}$ and $\cong_{T_2}$. 

Notice that if $T_1$ a classifiable shallow theory and $T_2$ is a classifiable non-shallow, then the gap between $\cong_{T_1}$ and $\cong_{T_2}$ contains a copy of $Card(I(\kappa,T_1),\kappa]$. So, from Fact \ref{SHMORL}, there at least $\kappa$ many equivalence relations in between $\cong_{T_1}$ and $\cong_{T_2}$, in the Borel reducibility hierarchy.

An understanding of the gap between the counting $\alpha$-classes equivalence relations, will give us a more detailed picture of the gap between the isomorphism relations of classifiable shallow theories and classifiable non-shallow theories. By Theorem \ref{GGSNS} we know that some of those gaps are not empty by Borel reduction and not continuous reductions, e.g. for some $\varrho_1<\varrho_2$, there is an equivalence relation $E$ such that  $$0_{\varrho_1}\ \redul\ E\ \redub\ 0_{\varrho_2}$$ and $$0_{\varrho_2}\ \not\redur\ E\ \not\reduc\ 0_{\varrho_1}.$$
Understanding the gap between $\alpha_\kappa$ and the isomorphism relation of classifiable non-shallow theories, would give us more information about the gap in Mangraviti-Motto Ros question.

\subsection{On Morley's Conjecture}
Morely's conjecture tells us that the spectrum function is increasing. As it was discussed before, Morley's conjecture implies that for all $\omega<\mu<\kappa$ and theory $T$, $\cong_T^\mu\redub \cong_T$ holds when $\cong_T^\mu$ is $\kappa$-Borel.
This can be extended to all theories and improved to continuous reductions.

\begin{prop}\label{Borel_init_Mor}
    Let $\omega<\mu<\delta\leq\kappa$ be cardinals. For every first-order countably theory in a relational countable language $T$, not necessarily complete, $$\cong_T^\mu\ \reduc\ \cong_T^\delta.$$ In particular $$\cong_T^\mu\ \hookrightarrow_\mu\ \cong_T^\delta.$$
\end{prop}

\begin{proof}
    Let $\omega<\mu<\delta<\kappa$ be cardinals and $T$ be a first-order countably complete theory in a relational countable language. Let $\varrho$ be the number of equivalence classes of $\cong_T^\mu$. 
    By Fact \ref{SHMORL}, there are sequences $\langle \xi_i\mid i<\varrho\rangle$ and $\langle \zeta_i\mid i<\varrho\rangle$ such that for all $i\neq j$, $\xi_i\ \not\cong_T^\mu \xi_j$, $\zeta_i\ \not\cong_T^\delta \zeta_j$, and for all $\eta\in \kappa^\kappa$ there is $i<\varrho$ such that $\xi_i\ \cong_T^\mu \eta$.
    Let us define $F:\kappa^\kappa\rightarrow\kappa^\kappa$ by $F(\eta)=\zeta_i$ where $i<\varrho$ is such that $\xi_i\ \cong_T^\mu \eta$. It is clear that $F$ is a reduction from $\cong_T^\mu$ to $\cong_T^\delta$.

    Let us show that $F$ is continuous. Let $\zeta,\eta\in \kappa^\kappa$ and $\beta<\kappa$ be such that $F(\eta)\in [\zeta\restriction\beta]$. Notice that for all $\xi\in [\eta\restriction \mu]$, $\mathcal{A}_{\eta\restriction\mu}=\mathcal{A}_{\xi\restriction\mu}$.
    By the way $F$ was defined, for all $\xi\in [\eta\restriction \mu]$, $F(\eta)=F(\xi)$. Thus, for all $\xi\in [\eta\restriction \mu]$, $F(\xi)\in [\zeta\restriction\beta]$ and we conclude that $F$ is continuous. in particular, $F$ is ${\mu}$-recursive.

\end{proof}

Notice that the only property of $\cong_T^\delta$ that we used was that it has at least $\varrho$ different equivalence classes and the size of the model.

We actually proved that $\cong_T^\mu\ \hookrightarrow_\mu\ E$, where $\cong_T^\mu$ had $\varrho$ equivalence classes and $E$ is an equivalence relation with at least $\varrho$ equivalence classes. 
In particular, since $\kappa^{<\kappa}=\kappa$, for any first-order countably complete theory in a relational countable language $T$ and $\mu<\kappa$, $\cong_T^\mu\ \reduc\ E$, for:
\begin{itemize}
    \item $E=\alpha_{\varrho'}$, $\varrho'\ge\varrho$ for all $\alpha<\kappa$.
    \item $\cong_{T'}$, where $T'$ is not a classifiable shallow theory.
    \item $id_2$, the identity relation of $2^\kappa$.
\end{itemize}

Notice that for all $\alpha<\kappa$, $\alpha_\varrho$ and $\cong_T^\mu$ are continuously bireducible, $\alpha_\varrho\ \reduc\ \cong_T^\mu$, and  $\cong_T^\mu$ is open. Therefore, if $T$ and $T'$ are theories such that $\cong_T^\mu$ and $\cong_{T'}^\mu$ have the same number of equivalence classes, then $\cong_T^\mu$ and $\cong_{T'}^\mu$ are continuous bireducible. 
In particular, if $\mu<\kappa$, $T$ is a classifiable non-shallow theory and $T'$ is non-classifiable theory, then by Fact \ref{SHMGT}, $\cong_T^\mu$ and $\cong_{T'}^\mu$ are continuous bireducible. 

The $S$-recursive functions can be used to study the isomorphism relation of models of small size. It is easy to see that when $\cong^\mu_T$ has $\varrho$ equivalence classes, $\cong^\mu_T\ \hookrightarrow_{\mu}\ \mu_\varrho$, $\mu_\varrho\ \not\hookrightarrow_{\mu}\ \cong^\mu_T$, and since $\mu$ is a cardinal, for all $\alpha<\mu$, $\alpha_\varrho\ \hookrightarrow_{\mu}\ \cong^\mu_T$.

Suppose there are cardinals $\omega<\mu<\mu'<\kappa$ such that $2^{\mu}=2^{\mu'}=\varrho$. Let $T$ be a not a classifiable shallow theory. By Fact \ref{SHMGT}, there are $\varrho$ non isomorphic models of size $\mu$ and of size $\mu'$. It is clear that $\cong_T^{\mu}$ and $\cong_T^{\mu'}$ are continuous bireducible. 
By the way the relations $\cong_T^\mu$ were defined, it is easy to see that $$\cong_T^{\mu}\ \hookrightarrow_{\mu}\ \mu_\varrho\ \hookrightarrow_{\mu+1}\ \cong_T^{\mu'}\text{ and }\cong_T^{\mu'}\ \not\hookrightarrow_{\mu+1}\ \mu_\varrho\ \not\hookrightarrow_{\mu}\ \cong_T^{\mu}$$
Indeed, suppose that $T$ is a theory such that its language, $\mathbb{L}(T)$, contains at least one relational symbol, $\mathfrak{R}$, of arity $0<n$. If $T$ has $\varrho$ models of size $\mu$, then $\cong_T^\mu\ \hookrightarrow_{\mu}\ \mu_\varrho$ and for all $\alpha<\mu$, $\cong_T^\mu\ \not\redup\ \alpha_\varrho$.

On the other hand, we can find the isomorphism relation of classifiable non-shallow theories in between $\cong_T^\lambda$ and $\cong_T$, due to Fact \ref{SHMGT} and Theorem \ref{Main_Gap}.

 \begin{cor}\label{TC3}
 Suppose $\kappa=\lambda^+=2^\lambda$ and $2^{\mathfrak{c}}\leq\lambda=\lambda^{<\omega_1}$. If $T_1$ is a classifiable non-shallow theory and $T_2$ is a non-classifiable theory, then 
$$\cong_{T_2}^\lambda\ \reduc\ \cong_{T_1}\ \reduc\ \cong_{T_2}.$$  
\end{cor}

In the generalize Baire space $\kappa^\kappa$ we can say a lot about the isomorphism relation of different theories when the models have size $\kappa$, $\cong_T$. But very few about the isomorphism relations when the models have size $\mu$ less than $\kappa$, $\cong_T^\mu$. Nevertheless, the isomorphism relation of models of size less than $\kappa$ can be used to describe the gap in more detail, under the assumptions of Fact \ref{Mangraviti_Motto}.

\begin{prop}\label{TC1}
Let $\kappa=\aleph_\gamma$ be such that $\beth_{\omega_1}(\mid\gamma\mid)\leq\kappa$ and $\kappa=\lambda^+=2^\lambda$. Suppose $T_1$ is classifiable shallow, $T_2$ classifiable non-shallow, and $T_3$ non-classifiable. Then $$\cong_{T_1}\ \redub\ 0_\kappa\ \redul\ \cong_{T_3}^\lambda\ \reduc\ \cong_{T_2}$$ and $$\cong_{T_2}\ \not\redur\ \cong_{T_3}^\lambda\ \not\redur\ \cong_{T_1}.$$
\end{prop}

\begin{proof}
It follows from Fact \ref{SHMGT}, Theorem \ref{GGSNS},  and Proposition \ref{Borel_init_Mor}.
\end{proof}

Let us study the reducibility between the identity and the isomorphism relations. This was studied by Friedman, Hyttinen, and Weinstein for $\kappa$ inaccessible.

\begin{fact}[{Friedman-Hyttinen-Kulikov, \cite[Theorem 36]{FHK13}}]\label{idred}
    Assume that $\kappa$ is strongly inaccessible and $T$ is a countable theory. If the number of equivalence classes of $\cong_T$ is greater than $\kappa$, then $$id_2\ \reduc\ \cong_T.$$
\end{fact}

\begin{cor}
    Assume that $\kappa$ is strongly inaccessible, $T$ is a countable theory, and $\mu<\kappa$ be a cardinal. If the number of equivalence classes of $\cong_T$ is greater than $\kappa$, then $$\cong_T^\mu\ \reduc\ id_2\ \reduc\ \cong_T.$$
\end{cor}

To study the case of $\kappa$ successor, we need to introduce the equivalence relation $E_0^{<\kappa}$, the \textit{equivalence modulo bounded}. Let us define $E_0^{<\kappa}$ as:
$$E_0^{<\kappa}:=\{(\eta,\xi)\in 2^\kappa\times 2^\kappa\mid \exists\alpha<\kappa\ [\forall\beta>\alpha\ (\eta(\beta)=\xi(\beta)])\}.$$

\begin{fact}[{Friedman-Hyttinen-Kulikov, \cite[Theorem 60]{FHK13}}]\label{bounded_to_stat}
Suppose $X\subseteq \kappa$ a stationary set such that $\diamondsuit_X$ holds. Then $E_0^{<\kappa}\ \reduc\ =^2_X$.
\end{fact}

\begin{cor}
    Suppose $\kappa=\lambda^+=2^\lambda$ and $\lambda=\lambda^{\omega}$. If $T$ is unstable, or superstable with the OTOP, or stable unsuperstable, then $$E_0^{<\kappa}\ \reduc\ \cong_{T}.$$
    If $\lambda$ also satisfies $2^{\mathfrak{c}}\leq\lambda=\lambda^{\omega_1}$, then $$E_0^{<\kappa}\ \reduc\ \cong_{T}$$ holds for $T$ superstable with the DOP.
\end{cor}

\begin{proof}
Since $\lambda^\gamma<\kappa=\lambda^+$ implies $cf(\lambda)>\gamma$, by \cite{Sh1} we know that if $\kappa=\lambda^+=2^\lambda$ and $cf(\lambda)>\gamma$, then $\diamondsuit_\gamma$ holds. The corollary follows from the previous fact and Theorem \ref{maincor}.
\end{proof}

\begin{fact}[{Friedman-Hyttinen-Kulikov, \cite[Theorem 34]{FHK13}}]\label{all_bounded}
\begin{itemize}
\item $E_0^{<\kappa}$ is $\kappa$-Borel.
\item $E_0^{<\kappa}\ \not\redub\ id_2$.
\item $id_2\ \reduc\ E_0^{<\kappa}$.
\end{itemize}
\end{fact}

 \begin{lemma}\label{lemma_identity}	
 Let $\kappa=\lambda^+=2^\lambda$ and $2^{\mathfrak{c}}\leq\lambda=\lambda^{<\omega_1}$. If $T$ is a non-classifiable theory, then $$id_2\ \reduc\ E_0^{<\kappa}\ \reduc\ \cong_{T}$$ and $$\cong_{T}\ \not\redub\ E_0^{<\kappa}\ \not\redub\ id_2$$ 
\end{lemma}

\begin{proof}
The only statement that doesn't follows directly from the previous results is $\cong_{T}\ \not\redub\ E_0^{<\kappa}$. This follows from Fact \ref{basics_FHK}, Fact \ref{reduction_classify_descrpitive}, and Fact \ref{all_bounded} item 1.
\end{proof}

 \begin{cor}\label{TC2}	
Suppose $\kappa=\lambda^+=2^\lambda$, and $2^{\mathfrak{c}}\leq\lambda=\lambda^{<\omega_1}$. If $T$ is a non-classifiable theory, then $$\cong_T^\lambda\ \reduc\ id_2\ \reduc\ E_0^{<\kappa}\ \reduc\ \cong_T$$ and $$\cong_{T}\ \not\redub\ E_0^{<\kappa}\ \not\redub\ id_2\ \not\redur\ \cong_T^\lambda.$$ 
\end{cor}

Corollary \ref{TC3}, Proposition \ref{TC1} and Corollary \ref{TC2} prove Theorem C.
The previous results can be extended to non-complete theories, by using the ideas in the proof of Theorem B and Morley's conjecture for not complete theories (see \cite{Sh90} page 642).

We used the complexity of the relation $id_2$ to study the complexity of the relation $\cong_T^\lambda$. We didn't study the relation $id_\kappa$ (the identity relation of $\kappa^\kappa$), since the existence of a diamond sequence $\langle f_\alpha\in \alpha^\alpha\mid \alpha\in X\rangle$ for a stationary set $X$, implies $id_\kappa \hookrightarrow_{R(\kappa)} id_2$. The reduction is given by the function

$$\mathcal{F}(\eta)(\alpha)=\begin{cases} 1 &\mbox{if } \alpha\in X\mbox{ and } \eta\restriction\alpha=f_\alpha\\
0 & \mbox{otherwise.} \end{cases}$$  

Indeed the usual reductions defined by the use of a diamond sequence are $\kappa$-recursive.

\section*{Acknowledgements}
This project has received funding from the Austrian Science Fund
(FWF) under the Lise Meitner Programme (grant M 3210-N). This project has received funding from the European Research Council (ERC) under the
European Union’s Horizon 2020 research and innovation programme (grant agreement No
101020762). This project has received funding from the project \textit{Perspectives on computational logic}, funded by the Research Council of Finland, project number 369424.

The author wants to thank the reviewer for the careful reading of the manuscript.
The author wants to express his gratitude to Tapani Hyttinen for introducing the topic to him and fruitful discussions. The author wants to express his
gratitude to Matthias Aschenbrenner, Ido Feldman, Allen Gehret, Martin Hils, Assaf Rinot, Corey Switzer, and Anda Tanasie  for helpful conversations and feed backs on a preliminary version of this manuscript.

\providecommand{\bysame}{\leavevmode\hbox to3em{\hrulefill}\thinspace}
\providecommand{\MR}{\relax\ifhmode\unskip\space\fi MR }
\providecommand{\MRhref}[2]{%
  \href{http://www.ams.org/mathscinet-getitem?mr=#1}{#2} }

\end{document}